%% file: twists.tex
\documentclass[11pt,a4paper]{amsart}

\usepackage{praeambel}
\usepackage{ap_makros}
\usepackage{md_tikz}

\title[Counting Resonances with Unitary Twists]{Counting Resonances on Hyperbolic Surfaces with Unitary Twists}


\author[M.~Doll]{Moritz Doll}
\address{Moritz Doll, University of Bremen, Department~3 -- Mathematics, 
Bibliothekstr.~5, 28359 Bremen, Germany}
\email{doll@uni-bremen.de}
\author[K.~Fedosova]{Ksenia Fedosova}
\address{Ksenia Fedosova, Albert Ludwigs University of Freiburg, Mathematical Institute, Ernst-Zermelo-Str. 1, 79104 Freiburg im Breisgau, Germany}
\email{ksenia.fedosova@math.uni-freiburg.de}
\author[A.~Pohl]{Anke Pohl}
\address{Anke Pohl, University of Bremen, Department~3 -- Mathematics, 
Bibliothekstr.~5, 28359 Bremen, Germany}
\email{apohl@uni-bremen.de}

\begin{document}

\begin{abstract}
    We present the Laplace operator associated to a hyperbolic surface~$\Gamma\setminus\mathbb{H}$ and a unitary representation of the fundamental group~$\Gamma$, extending the previous definition for hyperbolic surfaces of finite area to those of infinite area. We show that the resolvent of this operator admits a meromorphic continuation to all of~$\mathbb{C}$ by constructing a parametrix for the Laplacian, following the approach by Guillop\'e and Zworski. We use the construction to provide an optimal upper bound for the counting function of the poles of the continued resolvent.
\end{abstract}

\maketitle

\tableofcontents

\input{twists_mainpart}

\bibliography{ap_bib} 
\bibliographystyle{amsalpha}

\setlength{\parindent}{0pt}

\end{document}

%% file: twists_mainpart.tex
\section{Introduction}

We consider a finitely generated Fuchsian group~$\group \subset \Isom^+(\h)$ together with a finite-dimensional Hermitian vector space $V$ and a unitary representation ${\twist \colon  \group \to \Unit(V)}$ of~$\group$ on~$V$. We let $X \coloneqq \group \bs \h$ denote the associated orbifold. We emphasize that we allow $\group$ to contain elliptic elements; hence $X$ is not necessarily a manifold. Nevertheless we call $X$ a \emph{hyperbolic surface} even though it need not be a genuine manifold.

The representation $\twist$ induces a Hermitian vector orbibundle $\bundle \to X$ with typical fiber $V$. On $\bundle$ we consider the Laplace operator, $\LapTwist$, which extends to an unbounded positive self-adjoint operator on $L^2(X, \bundle)$. The resolvent \[\ResTwist(s) = (\LapTwist - s(1-s))^{-1}\] is defined for $s\in\C$ with $\Re s > 1/2$ and $s(1-s)$ not in the spectrum of~$\LapTwist$, as a bounded operator $L^2(X,\bundle) \to L^2(X,\bundle)$.

\medskip

Our first main result is that the resolvent admits a meromorphic continuation to $s \in \C$.

\begin{theorem}\label{thm:resolv_merom}
Let $X = \group \bs \h$ be a geometrically finite hyperbolic surface and
$\twist \colon  \group \to \Unit(V)$ be a finite-dimensional unitary
representation.
The resolvent~$\ResTwist(s)$ admits a meromorphic continuation to $s \in \C$
with poles of finite multiplicity as an operator
\begin{align*}
\ResTwist(s) \colon  L^2_\cpt(X, \bundle) \to L^2_\loc(X, \bundle)\,.
\end{align*}
\end{theorem}

The poles of $\ResTwist(s)$ are called \emph{resonances}, the set of resonances is denoted by $\ResSet_{X,\twist}$ and the multiplicity of each resonance $s \in \ResSet_{X,\twist}$ is denoted by~$m_{X,\twist}(s)$. We define the \emph{resonance counting function} $N_{X,\twist}(r)$ as the sum of all resonances $s \in \C$ with $\abs{s} < r$ counted with multiplicities
(see \eqref{eq:resonance_counting_function}).

\medskip

Our second main result is that the counting function of the resonances admits a
quadratic upper bound.
\begin{theorem}\label{thm:upper-bound}
The resonance counting function satisfies
\begin{align*}
N_{X,\twist}(r) = O(r^2) \qquad\text{as $r \to \infty$}.
\end{align*}
\end{theorem}

\subsection*{Prior Results}
For finite area $X$ (hence no funnel ends) possibly with unitary twists, we refer to the articles of Phillips~\cite{Phillips_scattering_twist,Phillips_perturb_twist} and the books by Venkov~\cite{Venkov_book} and Hejhal~\cite{Hejhal1,Hejhal2}.
For hyperbolic surfaces without elliptic points and no twists, Guillop\'e~\cite{Guillope} showed the meromorphic continuation of the resolvent and
Guillop\'e--Zworski~\cite{GuZw95} proved the upper bound on the number of resonances.
For asymptotically hyperbolic manifolds and infinite volume hyperbolic surfaces without cusps, the meromorphic continuation was proved by Mazzeo--Melrose~\cite{MaMe87}, see also Guillarmou~\cite{Guillarmou05}.

\subsection*{Structure}
We will mainly follow the approach of Guillop\'e--Zworski~\cite{GuZw95}. However, since our Laplacian acts on sections in a vector orbibundle, we have to modify the construction of the meromorphic continuation of the resolvent. In particular, the gluing construction of Guillop\'e--Zworski is more involved, with a new obstacle arising from finding a suitable way to restrict the representation to the funnel ends and the cusp ends. For the model calculations of the funnel ends and the cusp ends, we can diagonalize the representation to reduce the Fourier expansion of the resolvent kernel to the untwisted case. To prove the upper bound on the number of resonances, we relate the resonances to zeros of a determinant of a certain operator that is closely related to the resolvent and estimate then the number of zeros of this determinant. We use the explicit representations of the resolvent for the funnel ends and the cusp ends.
Since the resonances of the funnel depend of the eigenvalues of the representation this is more delicate than in the untwisted case.

\subsection*{Outline}
The article is structured as follows: we collect the necessary background material about  hyperbolic spaces and representations of Fuchsian groups in Section~\ref{sec:geometry}. In particular, we present the vector bundle associated to~$\twist$, its restriction to the boundary as well as the notion of section that are smooth up to the boundary. For the two model ends, the funnel and the cusp, we explicitly compute the meromorphic continuation of the resolvent in Section~\ref{sec:model}. The main analysis of the Laplacian or rather its resolvent is performed in Section~\ref{sec:resolvent},
where we prove Theorem~\ref{thm:resolv_merom}, establishing that the resolvent $(\LapTwist - s(1-s))^{-1}$ admits a meromorphic extension to $\C$ with poles of finite multiplicity.
In Section~\ref{sec:upper-bound} we prove Theorem~\ref{thm:upper-bound}, the upper bound for the resonances.

\subsubsection*{Acknowledgements}

KF and AP wish to thank the Institut Mittag-Leffler where part of this work was
done during the conference ``Thermodynamic Formalism -- Applications to
Geometry, Number Theory, and Stochastics.'' All authors acknowledge partial support by the Trimester Program ``Dynamics: Topology and Numbers'' by the Hausdorff Research Institute for Mathematics in Bonn. AP's research is funded by the Deutsche Forschungsgemeinschaft (DFG, German Research Foundation) -- project no.~441868048 (Priority Program~2026 ``Geometry at Infinity''). MD was partially funded by a Universit\"at Bremen ZF 04-A grant.

\section{Preliminaries}
We denote the Euclidean distance function on $\C$ by~$d_{\C}$. For $z \in \C$ we define the Japanese bracket $\ang{z} \coloneqq \left(1 + \abs{z}^2\right)^{1/2}$.

\subsection*{Conventions}
Let $(X,g_X)$ be an oriented Riemannian manifold. We denote the canonical measure by $d\mu_X$. We denote by~$\CcI(X)$ the space of complex-valued compactly supported smooth (i.e., infinitely often differentiable) functions on~$X$, endowed with its canonical LF topology. Its topological dual space, the space of distributions, is denoted by~$\CmI(X)$. Further we denote the canonical dual pairing between~$u \in \CcI(X)$ and~$f \in \CmI(X)$ by~$\ang{f, u} \coloneqq f(u)$. We denote the dual space of $\CI(X)$ by $\CcmI(X)$, which is the space of compactly supported distributions. We always identify operators $A \colon  \CcI(X) \to \CmI(X)$ with their Schwartz kernel $K_A \in \CmI(X \times X)$ via
\begin{align*}
\ang{Au,v} = \ang{K_A, u \otimes v}\,,
\end{align*}
that is
\begin{align*}
A \coloneqq K_A \in \CmI(X \times X)\,.
\end{align*}
Or, in the sense of distributions, we may write
\[
\int_X (Au)(z) v(z) \,\meas_X(z) = \int_{X\times X} A(z,z') u(z) v(z') \,\meas_{X\times X}(z,z')\,.
\]

Let also $Y$ be a Riemannian manifold. Let $E \to X$ and $F \to Y$ be smooth vector bundles over $X$ and $Y$, respectively. If $A\colon \CcI(Y,F) \to \CmI(X,E)$ is a bounded linear operator, then its Schwartz kernel is~$A \in\nobreak \CmI(X \times Y, E \boxtimes F')$ (see~Kumano-go~\cite[Theorem~8.7]{KumanoGo} for the case that $A$ is smoothing), where the exterior tensor product $E \boxtimes F'$
is defined via
\[
(E \boxtimes F')_{(x,y)} = E_x \otimes F_y'\,.
\]
If $U_X \subseteq X$ and $U_Y\subseteq Y$ are coordinate neighborhoods in
which $E|_{U_X} \cong U_X \times V_X$ and $F|_{U_Y} \cong U_Y \times V_Y$ for some fixed vector spaces $V_X$ and $V_Y$, then the bundle $E \boxtimes F'$ has the local trivialization
\begin{align*}
(E \boxtimes F')|_{U_X \times U_Y} \cong U_X \times U_Y \times V_X \times
V_Y' \cong U_X \times U_Y \times \Hom(V_Y, V_X)\,.
\end{align*}
In particular, for $X = Y$ and $E = F$ and any coordinate neighborhood $U \subset X$, we have
\begin{align}\label{eq:boxtimes-local-coord}
(E \boxtimes E')|_{U \times U} \cong U \times U \times \End(V)\,.
\end{align}
In these coordinates, a section $u \in \CI(U \times U, E \boxtimes E')$ is of
the form 
\[
u \colon  U \times U \to \End(V)\,.
\]

Let $\mathcal{U}$ be a Banach space and $\Omega \subset \C$ open. A function $f \colon  \Omega \to \mathcal{U}$ is holomorphic if it is complex differentiable on
all of~$\Omega$. If $f \colon  \Omega \to \mathcal{U}$ is locally bounded, then
$f$ is holomorphic if and only if for every functional~$v \in \mathcal{U}'$ the
composed function $v \circ f \colon  \Omega \to \C$ is holomorphic (see, e.g., \cite[Chapter~3, Theorem~3.12]{Kato}). The function~$f$ is meromorphic if it is holomorphic
except on a set of isolated points in~$\Omega$ and $f$ has poles at these
isolated points.

For any bounded operator $A$, we set $\abs{A} \coloneqq (A^* A)^{1/2}$.
Its non-zero eigenvalues are called the \textit{singular values} of the operator $A$, which we denote by $\mu_k(A)$, $k\in\N$, listed in decreasing order.

We denote the identity operator on function spaces by $\I$ and we will omit it if it is clear from the context. For instance, we write the resolvent of an operator $A$ as $(A - \lambda)^{-1}$ instead of $(A - \lambda \I)^{-1}$.
The identity on a finite-dimensional space $V$ will be denoted by $\id_V$.

\subsection*{Estimates and Asymptotics}
Let $X$ be a set and $a,b \colon  X \to \R$ functions. We write
\[
a\lesssim b\qquad\text{or}\qquad a(x) \lesssim b(x)
\]
if there exists a constant $C > 0$ such that for all $x \in X$,  $\abs{a(x)} \leq C \abs{b(x)}$. Let now $X=\R$ and suppose that $x_0 \in \R\cup\{\pm\infty\}$. We say that
\[
a \in O(b) \qquad\text{or}\qquad a = O(b) \quad \text{as $x \to x_0$}
\]
if $a(x)/b(x)$ is bounded as $x\to  x_0$. If $a \in O(b)$ as $x \to x_0$ and $b \in O(a)$ as $x \to x_0$, then we write
\[
a \asymp b \qquad\text{or}\qquad a(x) \asymp b(x)\quad \text{ as $x \to x_0$}\,.
\]

\section{Geometry of Hyperbolic Surfaces}\label{sec:geometry}

\subsection{Hyperbolic Surfaces}

We denote the hyperbolic plane by~$\h$ and identify it with its
upper half-plane model
\[
\h = \{z \in \C \setmid \Im z > 0\}
\]
with Riemannian metric
\begin{equation}\label{eq:RiemmetricH}
g_\h(x,y) = \frac{dx^2 + dy^2}{y^2} \qquad (z=x+iy\in\h)\,.
\end{equation}
We let $\Isom^+(\h)$ denote the group of orientation-preserving Riemannian
isometries of~$\h$, endowed with the compact-open topology. As the action of~$\Isom^+(\h)$ on~$\h$ is faithful, we identify the elements of~$\Isom^+(\h)$ with their action. If $\group$ is a discrete subgroup of~$\Isom^+(\h)$, then the orbit space
\[
 X = \group\bs\h
\]
naturally carries the structure of a good hyperbolic Riemannian orbifold.
It inherits its orbifold Riemannian metric $g_X$ from~$\h$, by pushing the Riemannian metric $g_\h$ defined in~\eqref{eq:RiemmetricH} to~$X$ via the canonical quotient map
\begin{equation}\label{eq:def_pigroup}
\pi_\group \colon  \h \to \group\bs\h\,.
\end{equation}
Equivalently, the Riemannian metric of~$X$ is defined via a family of compatible, equivariant Riemannian metrics in the orbifold charts, for which we may choose here, by slight abuse of concept, all of~$\h$ as a global chart. For simplicity, we will call any such orbifold~$X$ a \emph{hyperbolic surface}, allowing a slight abuse of notion. We emphasize that we allow hyperbolic surfaces to have singularity points, which are caused by those non-identity elements in~$\group$ whose action on~$\h$ has fixed points.

We call a subset~$\funddom$ of~$\h$ a \emph{fundamental domain} for the action of~$\group$ on~$\h$ if
\begin{enumerate}[label=$\mathrm{(\alph*)}$, ref=$\mathrm{\alph*}$]
\item $\funddom$ is connected,
\item for any $g\in\group$, $g.\funddom^\circ \cap \funddom^\circ = \emptyset$,
and
\item $\h = \bigcup \{ g.\overline\funddom \setmid g\in\group\}$.
\end{enumerate}
We remark that we do not require any specific properties of the boundary of a fundamental domain. In particular, we deviate slightly from the usual definition of fundamental domains and admit also closed or semi-closed sets as fundamental domain. For certain types of investigations, it is helpful to identify the hyperbolic space~$X$ with any of its fundamental domains (including its side pairings).

The group~$\Isom^+(\h)$ of orientation-preserving isometries can and shall be identified with the projective special linear group $\PSL(2,\R)$, which acts by M\"obius transformations on the upper half-plane~$\h$. For $g \in \PSL(2,\R)$ and $z \in \h$, we denote by $g.z$ the action of $g$ on $z$.
Every $g \in \PSL(2,\R)$, $g\not=\id$, can be classified as either elliptic, parabolic, or hyperbolic if $\abs{\tr g} < 2$, $\abs{\tr g} = 2$, or $\abs{\tr g} > 2$, respectively. If $g$ is hyperbolic, then we find $h \in \PSL(2,\R)$ such that
\begin{align*}
h^{-1} g h = h_\ell\,,
\end{align*}
where $h_\ell .z = e^\ell z$ and $\ell= 2\operatorname{arcosh}(\abs{\tr g}/2)$.
If $g \in \PSL(2,\R)$ is parabolic, we can conjugate~$g$ to the element $T \in \PSL(2,\R)$ which is given by $T.z \coloneqq z+1$.

If the fundamental group~$\group$ of a hyperbolic surface~$X = \group\bs\h$ is finitely generated, then $X$ has only finitely many ends. In this case, $X$ is  called \emph{geometrically finite}.
\begin{center}
\framebox{
\begin{minipage}{.75\textwidth}
Throughout this article, we will restrict our consideration to hyperbolic surfaces that are geometrically finite.
\end{minipage}
}
\end{center}

Among the geometrically finite hyperbolic surfaces, we will sometimes consider separately the \emph{elementary} ones. These are those hyperbolic surfaces~$X = \group\bs\h$ for which $\group$ is either generated by a single element (thus,
$\group$ is cyclic) or $\group$ is conjugate to a group~$\langle a,b\rangle \subseteq \Isom^+(\h)$, where $a$ acts as a dilation,
\[
 a.z = qz \quad\text{for some $q>1$ and all $z\in\h$}\,,
\]
and $b$ is the reflection on the unit circle,
\[
 b.z = -\frac1z\quad\text{for all $z\in\h$}\,.
\]

\subsection{Ends of Hyperbolic Surfaces and Decomposition}\label{sec:ends}

It is well-known that any non-elementary, geometrically finite hyperbolic surface has only two types of ends, namely funnels and cusps, and has only finitely many of these. In this section, we will recall these types of ends and relate them to the elementary surfaces.

\subsubsection{Funnel ends}\label{sec:funnel_ends}
For $\ell \in (0,\infty)$, we define the model funnel $F_\ell$ as the Riemannian
manifold 
\begin{align*}
F_{\ell} \coloneqq (0,\infty)_r \times (\R / 2\pi \Z)_\phi\,,
\\
g_{F_\ell} \coloneqq dr^2 + \frac{\ell^2}{4\pi^2} \cosh^2 r\, d\phi^2\,.
\end{align*}
The subscript at sets denotes here and further in the article the typical variable name for its elements. The simplest hyperbolic surface with funnel ends is a \emph{hyperbolic cylinder}, which is given by the quotient
\[
C_{\ell} \coloneqq \ang{h_\ell} \bs \h\,,
\]
where $h_\ell\colon  z\mapsto e^{\ell}z $ is the standard hyperbolic element
with displacement length $ \ell \in (0,\infty)$. A (semi-closed) fundamental domain for~$C_\ell$ is given by 
\begin{align}\label{eq:funddom_hypcyl}
\funddom = \{ z\in \h \setmid 1\leq \abs{z} < e^{\ell} \}\,.
\end{align}
Set $\omega \coloneqq 2\pi / \ell$. On the fundamental domain, we introduce geodesic coordinates $(r,\phi) \in \R \times (\R/2\pi\Z)$ via
\begin{align}\label{eq:cylinder-iso}
z = e^{\omega^{-1} \phi} \frac{e^r + i}{e^r - i}\,.
\end{align}
In these coordinates, the hyperbolic metric of~$C_\ell$, as induced from~$\h$, reads
\begin{align}\label{eq:metric-hyp-cylinder}
g_{C_\ell} = dr^2 + \omega^{-2} \cosh^2(r) d\phi^2\,.
\end{align}
Consequently, we use the identification $C_\ell \cong \R \times (\R/2\pi\Z)$.
We also see that the hyperbolic cylinder is the union of two funnel ends joined by a single period geodesic, tracing out the set $\{r = 0\}$. In particular, the model funnel $F_\ell$ is given by $F_\ell = C_\ell \cap \{r > 0\}$ under this identification. We note that it has geodesic boundary.

\subsubsection{Cuspidal ends}
The model cusp is given by (geodesic-horocyclic coordinates)
\begin{align*}
F_\infty \coloneqq (0,\infty)_r \times (\R / 2\pi\Z)_\phi\,, 
\\
g_{F_\infty} \coloneqq dr^2+ e^{-2r} \, \frac{d\phi^2}{4\pi^2}\,.
\end{align*}
The \emph{parabolic cylinder} is given by the quotient $C_\infty \coloneqq \ang{T} \bs \h$, where $T.z = z+1$ is the unit shift. A geodesically convex fundamental domain for~$C_\infty$ is given by
\begin{align*}
\funddom = \{z \in \h \setmid \Rea z \in (0,1]\}\,.
\end{align*}
The map
\[
(r,\phi) \mapsto z = \frac{\phi}{2\pi} + i e^r
\]
induces a diffeomorphism
\[
\R \times (\R / 2\pi \Z) \cong C_\infty
\]
with respect to which the hyperbolic metric on~$C_\infty$ reads
\begin{align*}
g_{C_\infty} = dr^2 + e^{-2r} \frac{d\phi^2}{4\pi^2}\,.
\end{align*}
Under this identification, the model cusp $F_\infty$ is given by $C_\infty \cap \{r > 0\}$. The level sets $\{r = c\}$ are the \emph{horocycles} centered at~$\infty$. (All other horocycles in~$\h$ are the images of these particular ones under the action of~$\Isom^+(\h)$.)

\subsubsection{Decomposition}\label{sec:decomposition}

The hyperbolic surface $X$ decomposes into
\[
X = K \sqcup X_f \sqcup X_c\,,
\]
where $X_f$ and $X_c$ are disjoint unions of funnels and cusps, respectively, and $K$ is compact (the \emph{compact core} of~$X$). That is
\begin{align*}
X_f &= \bigsqcup_{j=1}^{n_f} X_{f,j}\,,
\\
X_c &= \bigsqcup_{j=1}^{n_c} X_{c,j}\,,
\end{align*}
where each $X_{f,j}$ is isometric to $F_{\ell_j}$ for some $\ell_j \in (0,\infty)$ and each $X_{c,j}$ is isometric to $F_\infty$. The boundary of $X_f$ is geodesic and each cusp end has area $1$. See, e.g., \cite[Section~2.4.1]{Borthwick_book}

A subset~$U$ of $\h$ is called \emph{$\group$-stable} if for each $g\in\group$ either
\[
 g.U = U \quad\text{or}\quad g.U \cap U = \emptyset\,.
\]
For any $\group$-stable set $U\subseteq\h$ we define the \emph{stabilizer of $U$} as
\[
 \group_U \coloneqq \{g\in\group\setmid g.U = U \}\,.
\]
Then $\group_U$ is a subgroup of $\group$. As discussed in \cite[\S 10.4]{Beardon} (see also \cite{Garland_Raghunathan} for hyperbolic surfaces of finite area, or \cite[Proof of Theorem 2.23]{Borthwick_book}) for each connected end $F$ of $X$ there exists an open connected set $U \subset \h$ such that
\begin{enumerate}[label=$\mathrm{(\alph*)}$, ref=$\mathrm{\alph*}$]
\item $\pi_\group(U) = F$.
\item The set $U$ is $\group$-stable.
\item The stabilizer group $\group_U$ equals the cyclic group~$\ang{\gamma}$ for some element $\gamma \in \group$. The element $\gamma$ is a primitive hyperbolic (parabolic) element of $\group$ if $F$ is a funnel (a cusp).
\end{enumerate}

\subsubsection{Boundary defining functions}

We fix a funnel boundary defining function $\rho_f \in \CI(X, (0,\infty))$ 
satisfying
\begin{align*}
\rho_f|_{X_{f,j}}(r,\phi) = \frac{1}{\cosh(r)}
\end{align*}
for all $j = 1, \dotsc, n_f$ in the coordinates as above and $\rho_f = 1$ on $X_c$. Similarly, we fix a cusp boundary defining function $\rho_c \in \CI(X, (0,\infty))$ with the properties that
\begin{align*}
\rho_c|_{X_{c,j}}(r,\phi) = e^{-r}
\end{align*}
for all $j = 1, \dotsc, n_c$ and $\rho_c = 1$ on $X_f$. Moreover, set
\begin{equation}\label{eq:rho_def}
\rho \coloneqq \rho_f \rho_c\,.
\end{equation}

\subsection{Geometry at Infinity and Compactifications}

We will take advantage of the geodesic compactification of the hyperbolic plane~$\h$. In order to present this compactification, we let 
\[
 \wh\C\coloneqq \C\cup\{\infty\}
\]
denote the Riemann sphere, that is, the one-point compactification of~$\C$. We 
consider $\h$ canonically as a subset of~$\wh\C$. Then the \emph{geodesic 
compactification} of~$\h$, denoted~$\overline\h$, is the topological closure 
of~$\h$ in~$\wh\C$. The \emph{geodesic boundary} of~$\h$, denoted~$\partial\h$, 
is then the topological boundary. Thus,
\[
 \partial\h = \R \cup \{\infty\}\,.
\]
We may also define~$\overline\h$ and~$\partial\h$ intrinsically, in which 
case the points at infinity are characterized as equivalence classes of 
asymptotic geodesics, and neighborhood bases for the points at infinity are 
provided by nested families of horoballs. We refer for this approach to, e.g., 
\cite{Eberlein}. 

The action of~$\PSL(2,\R) \cong \Isom^+(\h)$ on~$\h$ extends continuously to~$\overline\h$. We will take advantage of the standard classification of the elements of~$\PSL(2,\R)$ using the number of the fixed points of their action on~$\overline\h$, as we recall now. Let $g\in\PSL(2,\R)$. Then $g$ is the identity in~$\PSL(2,\R)$ if and only if the action of $g$ on~$\overline\h$ has more than two fixed points. If this action has exactly two fixed points, then both fixed points are on the boundary of~$\h$. This is the case if and only if $g$ is hyperbolic. If the action of~$g$ has a single fixed point, then this fixed point is on the boundary of~$\h$ if and only if $g$ is parabolic, and it is in the interior of~$\h$ if and only if $g$ is elliptic.

Let $\group$ be a discrete subgroup of~$\Isom^+(\h)$. We denote by 
$\limSet(\group)$ the \emph{limit set} of the group $\group$, that is the set 
of all accumulation points (limit points) in~$\overline\h$ of the action of~$\group$ on~$\h$. The discreteness of~$\group$ implies that $\limSet(\group) \subseteq\partial\h$. We denote by 
\[
\Omega(\group) \coloneqq \overline{\h} \setminus \limSet(\group)
\] 
the set of \emph{ordinary points} and by~$\paraFix(\group)$ the set of points fixed by parabolic elements. The geodesic compactification of the hyperbolic surface~$X=\group\bs\h$ is then 
\begin{align*}
\overline{X} = \group \bs \bigl(\Omega(\group)\cup\paraFix(\group)\bigr)\,,
\end{align*}
which is also our choice of compactification in this article. We remark that~$\overline X$ is not necessarily a manifold with boundary nor an orbifold as each cusp of~$X$ is compactified by a single point.

The \emph{geodesic boundary} or \emph{ideal boundary}~$\pa_\infty X \coloneqq 
\pa \overline{X} = \overline{X} \setminus X$ of~$\overline X$ can be identified 
with
\begin{align*}
\pa_\infty X \cong \{1,\dotsc,n_c\} \sqcup \bigsqcup_{j = 1}^{n_f} (\R / 2\pi \Z)\,,
\end{align*}
using the isomorphisms between the ends of~$X$ and the model ends (see  Section~\ref{sec:ends}). Since funnel and cusp ends behave structurally different, we will split the boundary into two parts,
\begin{align}\label{eq:split-infty}
\pa_\infty X = \pa_c X \sqcup \pa_f X\,,
\end{align}
where $\pa_c X \cong \{1,\dotsc,n_c\}$ and $\pa_f X \cong \bigsqcup_{j = 1}^{n_f} (\R / 2\pi \Z)$. When focusing on the geometry of the funnel ends, we may want to rescale the Riemannian metric of~$X$ by a factor of~$\rho^2$. The metric $\rho^2 g$ has the property that it restricts to a metric on $\pa_f X$. For any single funnel end, under the isomorphism with the model funnel, the metric on~$\pa_f X_{f,j}$ reads 
\[
g_{\pa_\infty F_\ell} \coloneqq \rho^2 g_{F_\ell}|_{\pa_\infty F_\ell} = \frac{\ell^2}{4\pi^2} d\phi^2\,.
\]

The set of \emph{smooth functions} $\CI(\overline{X})$ consists of the 
functions $u \in \CI(X)$ that satisfy
\begin{align*}
u\vert_{X_{c,j}}(\rho, \phi) &= u_0(\rho) + O(\rho^\infty)\,, \quad (\rho,\phi) \in X_{c,j}\,,\ j \in \{1,\dotsc, n_c\}
& \intertext{for some $u_0\in \CI_b(X_{c,j})$ not depending on $\phi$, and}
u\vert_{{X_{f,j}}} &\in \CI(X_{f,j}) \text{ extends smoothly to } \overline{X_{f,j}}\,.
\end{align*}

\subsection{Blow-Up of the Funnel Ends}\label{sec:blow-up}

Guillop\'e and Zworski~\cite[pp.~607-608]{GuZw97} showed that the
untwisted resolvent kernel has a good description when pulled back to the stretched product space. In this section, we recall the definition of this space and the natural coordinates for the case of the hyperbolic plane. 

Let $M$ be a manifold with corners and $Y$ be a sufficiently nice submanifold of~$\pa M$. (The submanifold $S$ has to be a $p$-submanifold in the language of \cite{Melrosemwc}.) The \emph{blow-up} of $Y$ in $M$ is defined as the disjoint union
\begin{align*}
[M; Y] \coloneqq S^+ NY \sqcup M \setminus Y\,,
\end{align*}
where $S^+NY$ denotes the inward-pointing unit normal bundle of~$Y$ in~$M$.
The blow-down map $\beta \colon  [M; Y] \to M$ is defined by setting $\beta|_{M \setminus Y} = \id$ and $\beta|_{S^+N_yY} = \pi_{S^+NY}$ for all $y \in Y$, where $\pi_{S^+NY}$ is the canonical projection $S^+NY \to Y$ on base points.
There is a natural smooth structure on $[M; Y]$ and $\beta$ is a smooth map 
between manifolds with corners. Intuitively, blow-ups introduce polar coordinates in a geometrically invariant way. For more details, we refer to Melrose~\cite{Melrosemwc}.

\begin{defi}
Let $X$ be a hyperbolic surface. The stretched product is given by
\begin{align}\label{eq:def_stretchprod}
X \times_0 X \coloneqq [\overline{X} \times \overline{X}; \pa_f \diagon]
\xrightarrow{\beta} \overline{X} \times \overline{X}\,,
\end{align}
where $\diagon \coloneqq \{(x,x) \setmid x \in \overline{X}\}$ 
denotes the diagonal in $\overline{X} \times \overline{X}$ and
\[
\pa_f \diagon \coloneqq \diagon \cap (\pa_f X \times \pa_f X)\,.
\]
\end{defi}

Some remarks are in order. 

\begin{remark}
\begin{enumerate}[label=$\mathrm{(\roman*)}$, ref=$\mathrm{\roman*}$]
\item The hyperbolic surfaces that we consider here are orbifolds and possibly not genuine manifolds. However, but since the orbifold points (i.e., the singularity points) are contained in a compact set and the blow-up is only performed near infinity, this causes no problems. 
\item In~\eqref{eq:def_stretchprod} we use $\pa_f \diagon$ instead 
of all of the whole diagonal in $\partial_\infty X\times \partial_\infty X$, i.e., the set $\diagon \cap (\partial_\infty X \times \partial_\infty X)$, for the submanifold to be blown up. This has the consequence that cusps are not blown up in a stretched product. 
\item The stretched product $X \times_0 X $ is also called the $\mathscr{V}_0$-stretched product (for example in~\cite{MaMe87}) to distinguish it from the stretched product used in the $b$-calculus. 
\item The stretched product $X \times_0 X$ is a manifold with corners. Its boundary faces are given by the \emph{front face}, $\beta^{-1}(\pa_f \diagon)$, and the \emph{side faces}, $\beta^{-1}(\pa_f X \times X)$ and $\beta^{-1}(X \times \pa_f X)$. 
\item One has to be careful when pulling back boundary defining functions on $X \times X$ via the blow-down map: if $\rho$ is a boundary defining function for $\pa_f X \times X$ on $X \times X$, then $\beta^*\rho$ vanishes on $\beta^{-1} \left(\overline{\pa_f X \times X}\right) = \beta^{-1}(\pa_f \diagon) \cup \beta^{-1}(\pa_f X \times X)$ and its derivative vanishes on $\beta^{-1}(\pa_f \diagon)$. Therefore $\beta^*\rho$ is \emph{not} a boundary defining function on $X \times_0 X$.
\end{enumerate}
\end{remark}

We define the diagonal in $X \times_0 X$ as
\begin{align*}
\diagon_0 \coloneqq \overline{ \beta^{-1}(\diagon \setminus \pa_f \diagon) }\,.
\end{align*}
We note that $\diagon_0$ is a $2$-dimensional submanifold of $X \times_0 X$. If $E \to X$ is a vector bundle over $X$, we define the exterior tensor 
product over $X \times_0 X$ by 
\[
E \boxtimes_0 E \coloneqq \beta^*( E \boxtimes E)\,.
\]
In the case $X = \h$, we have a particularly simple description of the blow-up space.
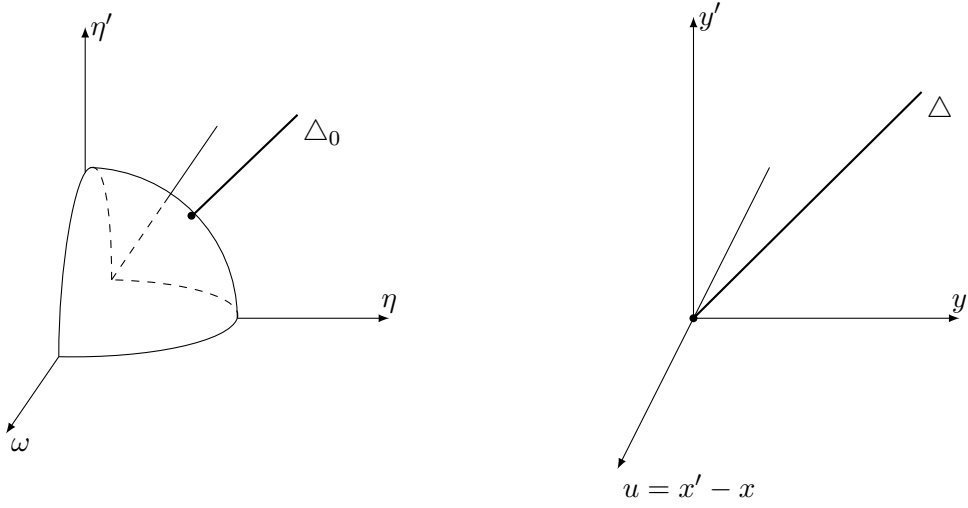
\begin{figure}
\centering
\input{blowup}
\caption{On the right: the cross section of 
$\overline\h\times\overline\h$, determined by fixing the value of~$x$; on the 
left: the corresponding cross section of the blown-up space 
$\h \times_0 \h$, determined by $\omega'=x$ being fixed.}
\label{fig:blow-up}
\end{figure}
We set $\pa_f\h \coloneqq \partial \h$ and then
\[
 \pa_\infty\diagon = \pa_f\diagon = \diagon \cap (\pa_f \h \times \pa_f \h) = 
\{(z,z) \setmid z\in\pa\h\}\,.
\]
Hence,
\[
\h \times_0 \h = [\overline{\h} \times \overline{\h}; \pa_\infty \diagon]\,.
\]
Guillop\'e and Zworski~\cite{GuZw97} provided the following global coordinates 
on $\h\times\h$ and local coordinates on $\h\times_0 \h$, which are very 
convenient for our considerations. We set 
\begin{equation}\label{eq:def_chart}
 \chart \coloneqq \{ (r,\eta,\eta',\omega,\omega')\in \R_+^3 \times \R^2 
\setmid \eta^2 + \eta'^2 + \omega^2 = 1 \}\,.
\end{equation}
Then the map 
\begin{equation}\label{eq:chart_coord}
 \chart \to \h\times\h\,,\quad (r,\eta,\eta',\omega,\omega') \mapsto 
(\omega',r\eta, \omega'+r\omega, r\eta')
\end{equation}
is a diffeomorphism. We remark that the condition~$\eta^2+\eta'^2+\omega^2 = 1$
in~\eqref{eq:def_chart} yields injectivity. The obvious extension of the above diffeomorphism to 
\begin{equation}\label{eq:blow-up_chart}
\overline\chart\coloneqq \{ (r,\eta,\eta',\omega,\omega')\in \R_{\geq 0}^3 \times \R^2 
\setmid \eta^2 + \eta'^2 + \omega^2 = 1 \}
\end{equation}
surjects onto 
\begin{equation}\label{eq:image_chart}
\bigr(\overline\h\setminus\{\infty\}\bigl) \times
\bigr(\overline\h\setminus\{\infty\}\bigl)\,.
\end{equation}
Restricted to $\overline\chart$, the coordinate function $r$ is locally the boundary defining function of the front face, the coordinates $\eta$ and $\eta'$ are locally the boundary defining functions for the side faces of $\h \times_0 \h$. In other words, $\overline\chart$ with the implied coordinates is a local chart of the stretched product~$\h\times_0\h$. In these coordinates and restricted to~$\overline\chart$, the blow-down map~$\beta$ is given 
by~\eqref{eq:chart_coord}, extended to~$\overline\chart$. We note that away from $\infty$, the coordinate function $y$ is a boundary defining function of $\pa_\infty \h \times \h$, but its pull-back to $\h \times_0 \h$ is \emph{not} a boundary defining function of the side face $\beta^{-1}( \pa_\infty \h \times \h )$ near the front face $\beta^{-1}( \pa_\infty \diagon )$.

For a full description of the stretched product~$\h\times_0\h$ and the map~$\beta$, further local charts are needed. These may be constructed analogously after applying the isometry $z\mapsto -1/z$ of~$\h$, which maps the boundary point~$\infty$ to~$0$, on either one of the factors of $\h \times_0 \h$. A total of three charts is then sufficient to cover~$\h \times_0 \h$.

For $z,z'\in \h$, define the pair-point invariant  
\begin{equation}\label{eq:sigma_coordfree}
\sigma(z,z') \coloneqq \cosh^2(d_\h(z,z')/2)\,,
\end{equation}
where $d_\h(z,z')$ is the hyperbolic distance between $z$ and $z'$.
We note that (see \cite[Eq. (4.18)]{Borthwick_book})
\begin{align}\label{eq:sigma}
\sigma(x+iy,x'+iy') = \frac{(x-x')^2 + (y+y')^2}{4yy'}\,.
\end{align}
Under the isomorphism in~\eqref{eq:chart_coord} and then in the full 
chart~$\overline\chart$ of the blown-up space, the function~$\sigma$ is 
given by
\begin{align}\label{eq:sigma-blownup}
(\beta^*\sigma)(r,\eta,\eta',\omega,\omega') =  \frac{1 + 2\eta\eta'}{4\eta\eta'}\,.
\end{align}
Taking advantage of the fact that $\sigma$ is a point-pair invariant and hence 
satisfies
\[
 \sigma(gz,gz') = \sigma(z,z')
\]
for all~$g\in\Isom^+(\h)$ and $z,z'\in\h$, we can easily cover the stretched 
product~$\h\times_0\h$ with charts with respect to which~$\beta^*\sigma$ is 
given as in~\eqref{eq:sigma-blownup} (separately for each chart, with respect 
to the coordinates of the considered chart). Thus, $1/(\beta^*\sigma)$ is 
smooth on all of~$\h \times_0 \h$.

\subsection{Bundles}\label{sec:bundles}

In this section we present the necessary details concerning bundle structures on 
orbifolds, specialized to the case of hyperbolic surfaces. In the case that the 
hyperbolic surface is a manifold, all the notions reduce to the classical notions for manifolds. All the notions will be provided in the smooth ($\CI$) category.

Let $\group$ be a Fuchsian group and let $X\coloneqq \group\backslash\h$ denote the associated hyperbolic surface. We let $[z]$ denote the element of $X$ that is represented by $z\in\h$. The projection map
\begin{equation}\label{eq:fibrebundle}
 \pi_\group\colon  \h\to X\,,\quad z\mapsto [z]\,,
\end{equation}
is a fiber orbibundle. The fiber of $[z]\in\h$ is
\[
\pi_\group^{-1}([z]) = \group.z\,,
\]
which is isomorphic to
\[
 \group/\group_z\,,
\]
where $\group_z$ denotes the stabilizer group of $z\in\h$ under $\group$ (or, in 
the language of orbifolds, the \emph{local group} at $[z]$). If $\group$ does not contain elliptic elements, then $\pi_\group$ is the universal covering of $X$. Let $\twist\colon \group\to\GL(V)$ be any finite-dimensional representation. (In this section we do not need to restrict to unitary representations.) In what follows we will define the \emph{bundle associated} to the fiber orbibundle in~\eqref{eq:fibrebundle} and the representation~$\twist$.

The action of $\group$ on $\h$ and the action of $\group$ on $V$ combines to an 
action of~$\group$ on  $\h\times V$ by setting
\[
 g.(z,v) \coloneqq (g.z, \twist(g)v)
\]
for $g\in\group$ and $(z,v)\in \h\times V$. We let
\[
 \h\times_\twist V \coloneqq \group\backslash\big(\h\times V\big) \coloneqq \{ 
\group.(z,v) \setmid (z,v) \in \h\times V\}
\]
denote the space of orbits of this $\group$-action on $\h\times V$. Clearly, the 
space $\h\times_\twist V$ carries an orbifold structure, given by the global chart $\h\times V$ with the action of $\group$. We set
\[
 E_\twist \coloneqq \h\times_\twist V
\]
and let $[z,v]$ denote the element of $E_\twist$ that is represented by 
$(z,v)\in\h\times V$. We let
\begin{equation}\label{eq:assbundle}
 \pi_1\colon  E_\twist \to X\,,\quad [z,v] \mapsto [z]\,,
\end{equation}
denote the projection map. The map~$\pi_1$ in~\eqref{eq:assbundle} is a fiber 
orbibundle, the \emph{bundle associated} to the fiber orbibundle~$\pi_\group$ 
in~\eqref{eq:fibrebundle} and the representation~$\twist$. For each $[z]\in X$, 
its fiber
\[
 \pi_1^{-1}([z]) = \{ [z,v] \setmid v\in V \}
\]
is isomorphic to
\[
 \group_z\backslash V\,.
\]
A \emph{section} of $\pi_1 \colon  E_\twist \to X$ is a smooth map $s\colon  X\to E_\twist$ such that
\[
 \pi_1\circ s = \id_X\,.
\]
We denote the space of sections of $\pi_1$ by $\CI(X; E_\twist)$. The following interpretation of sections as maps $\h\to V$ is folklore. However, since we require global maps instead of only local ones, we provide a proof for convenience.

We say a function $f$ belongs to the space $\CI(\h;V)_\ssec$ if $f\in \CI(\h;V)$ and  for all $z \in \h$ and $g \in \group$, we have
\begin{equation}\label{eq:sec_equiv}
 f(g.z) = \twist(g)f(z)\,.
\end{equation}
\begin{lemma}\label{lem:unraveledsections}
The map $\CI(\h;V)_\ssec \to \CI(X,\bundle), f \mapsto s$, with $s$ given by
\begin{align*}
s([z]) \coloneqq [z,f(z)]
\end{align*}
is an isomorphism.
\end{lemma}

\begin{proof}
We suppose first that $s\in \CI(X; E_\twist)$ is a section of~$\pi_1$. We will 
show that there is a unique map~$f\in \CI(\h;V)$ such that for all~$z\in\h$ we 
have
\begin{equation}\label{eq:sec_map}
 s([z]) = [z,f(z)]\,.
\end{equation}
To that end we recall that for each~$z\in\h$, the fiber~$\pi_1^{-1}([z])$ is 
isomorphic to~$\group_z\backslash V$. Thus, the element~$v\in V$ with~$s([z]) = 
[z,v]$ is unique up to left-action of~$\group_z$. Let $Z$ denote the set of 
points $z\in\h$ for which $\group_z = \{\id\}$. Then, for any~$z\in Z$, there 
is a unique~$v_z\in V$ such that~$s([z]) = [z,v_z]$. This yields a unique map
\[
 f\colon  Z \to V\,,\quad z\mapsto v_z\,,
\]
satisfying~\eqref{eq:sec_map} on~$Z$. It remains to show that~$f$ extends to a 
smooth map on~$\h$ and satisfies~\eqref{eq:sec_equiv}. For any~$z\in\h$ and any~$v\in V$ such that $s([z]) = [z,v]$ there exist open neighborhoods~$U_z$ of~$z$ in~$\h$ and $U_v$ of~$v$ in~$V$ and a smooth map $\tilde s\colon  U_z \to U_z\times U_v$ such that
\[
 \xymatrix{
 U_z \ar[r]^{\tilde s} \ar[d]_{\pi_\group} & U_z\times U_v \ar[d]^{\pi_\group}
 \\
 X \ar[r]^{s} & E_\twist
 }
\]
commutes. Since~$\h\setminus Z$ is a countable union of isolated points, we may choose the neighborhood~$U_z$ such that it is connected,
intersects the set~$\h\setminus Z$ in at most the point~$z$ and is $\group$-stable.
Let
\begin{align*}
 \pr_1&\colon  U_z\times U_v\to U_z
 \intertext{and}
 \pr_2&\colon  U_z\times U_v\to U_v
\end{align*}
denote the projection on the first and second component, respectively. Again 
using that~$\h\setminus Z$ consists of isolated points, we may assume without 
loss of generality that 
\[
 \pr_1\circ \tilde s = \id_{U_z}\,.
\]
(There exists an element of~$\group_z$ such that the composition of~$\tilde s$ with this element calibrates~$\tilde s$ to such a map.) Then $\tilde s(w) = (w, f(w))$ for all $w\in Z\cap U_z$ or, in other words,
\begin{equation}\label{eq:fislift}
 \pr_2\circ\tilde s = f \quad\text{on $Z\cap U_z$}.
\end{equation}
Since $\pr_2\circ\tilde s$ is smooth, \eqref{eq:fislift} yields a smooth 
extension of~$f$ to~$z$. Because~${\h\setminus Z}$ is discrete, this extension is 
well-defined and unique. Since $z$ was arbitrary, $f$ extends (uniquely) to a 
smooth map on $\h$, satisfying \eqref{eq:sec_map} on all of~$\h$. For any $z\in\h$ and any $g\in\group$ we have
\begin{align*}
 [g.z, f(g.z)] & = s([g.z]) = s([z]) = [z,f(z)] = [g.z, \chi(g)f(z)]\,.
\end{align*}
Therefore,
\[
 f(g.z) = \chi(h)\chi(g)f(z)
\]
for some $h\in\group_{g.z}$. Using that for $z\in Z$, we have $h=\id$, and that 
$Z$ is dense in $\h$, we find that $f(g.z)= \chi(g)f(z)$ on all of 
$\h\times\group$. This shows that~$f$ satisfies~\eqref{eq:sec_equiv}. Let now $f\in \CI(\h;V)_\ssec$ and set
\[
 s([z]) \coloneqq [z,f(z)]
\]
for all $z\in\h$. It is straightforward to show that $s\in \CI(X;E_\twist)$.
\end{proof}

We denote by
\[ \CcI (X; E_\twist) \coloneqq \{ f \in  \CI(X; E_\twist) \setmid  \text{$\supp f$ is compact} \}\,,
\]
the space of compactly supported sections of $E_\twist$. We may use the isomorphism from Lemma~\ref{lem:unraveledsections} to identify this space with 
a subspace of~$\CI(\h;V)_\ssec$. Under the isomorphism from Lemma~\ref{lem:unraveledsections}, the space~$\CcI (X; E_\twist)$ corresponds to
\[
\CcI(\h;V)_\ssec \coloneqq \{ f\in \CI(\h;V)_\ssec \setmid \text{$\overline{\funddom} \cap \supp f$  is compact}\}\,,
\]
where $\funddom$ is any fundamental domain for the action of $\group$ on $\h$. 
We remark that the definition of $\CcI(\h;V)_\ssec$ does not depend on the 
choice of the fundamental domain.

\subsection{Definition of the Laplacian}\label{sec:def_Laplacian}
The goal of this subsection is to define the action of the Laplace operator on $\CcI(X; E_\chi)$, to study its extensions and to choose the one that is suitable for us. We will use the advantage of the fact that our orbibundle, $E_\chi$, is associated to the representation, $\chi$. It is possible to define, e.g., a connection Laplacian on a general orbifold bundle, using the notion of a connection (for a definition of a connection on an orbifold bundle see, e.g., \cite[Definition 4.3.1]{ChenRuan2001}), but we will use a short-cut. 

We choose a basis of~$V$, say $(v_j)_{j=1}^{\dim V}$. To every function, $f \in \CI(\h;V)$, we associate functions $f_j \colon \HH \to \C$ such that 
\[
f(z) = \sum_{j=1}^{\dim V} f_j(z) v_j\,, \quad z \in \HH\,.
\]
The action of the Laplace-type operator, $\Delta^\#$, is given as follows:
\[
\Delta^\# f(z) \coloneqq  \sum_{j=1}^{\dim V} (\Delta_\HH f_j(z)) v_j\,.
\]
Note that the definition does not depend on the basis due to its linearity.
Additionally, $\Delta^\#$ preserves the equivariance \eqref{eq:sec_equiv},
thus it defines an operator, $\Delta_{X,\twist}$, on $\CI(X;\bundle) \cong \CI(\h;V)_\ssec$. 

We now define the appropriate $L^2$-space on which $\Delta^\#$ will be essentially self-adjoint.
We recall that $V$ is a finite-dimensional Hermitian vector space with inner product $(\cdot,\cdot)_V$ and we denote the corresponding norm by $\norm{\cdot}_V$. Since $\twist\colon \group \to U(V)$ is a unitary representation of 
$\group$, we have $(\twist(\group)u, \twist(\group)v)_V = (u,v)_V$ for all 
$u,v \in V$. Hence, $(\cdot,\cdot)_V$ induces a (well-defined) Hermitian bundle metric $(\cdot,\cdot)_{\bundle}$ on $\bundle = \h \times_\twist V$ by
\begin{align*}
( [g,v], [g,w])_{\bundle} = (v,w)_V\,.
\end{align*}
We define a sesquilinear product on sections of $\bundle$ by
\begin{align}\label{eq:bundle-metric-infty}
(u,v)_{L^2(X,\bundle)} \coloneqq \int_X (u,v)_{\bundle} d\mu_X\,.
\end{align}
This inner product defines the space $L^2(X,\bundle)$.
We will drop the subscript from the inner product if it is clear from the context. Moreover, we define the bilinear product $\ang{u,v} \coloneqq (u,\bar{v})$.

\begin{lemma}
There exists a unique self-adjoint extension of \[\Delta_{X,\twist} \colon  \CcI(X,\bundle) \to L^2(X,\bundle)\] with respect to $\ang{\cdot, \cdot }$. 
\end{lemma}

\begin{proof}
By Selberg's Lemma \cite[Lemma~8]{Selberg_lemma},
there is a finite covering $\widetilde{X} = \tilde\group \bs \h$ of $X$ such that $\tilde{\Gamma}$ has no elliptic fixed points; see also \cite{Cassels_selberglemma} and \cite[Theorem II]{Alperin}.
Define the vector bundle $\widetilde{E}$ over $\widetilde{X}$ as the pull-back of $\bundle$ under the covering map $\widetilde{X} \to X$ and let $\widetilde{\Delta} \colon  \CI(\widetilde{X},\widetilde{E})\to \CI(\widetilde{X},\widetilde{E})$ be the pulled-back Laplacian.
Since $\widetilde{X}$ is complete, the operator $\tilde{\Delta}$ is essentially self-adjoint as an operator on $L^2(\widetilde{X},\widetilde{E})$.

Assume that $\LapTwist$ is not essentially self-adjoint.
Recall that a symmetric operator $A$ is essentially self-adjoint if and only if $\ker(A^* - i) = 0$ and $\ker(A^* + i) = 0$.
Without loss of generality, we have that $\ker(\LapTwist^* - i) \neq 0$, that is, there exists a non-zero $v \in L^2(X,\bundle)$ such that $v \in \ker(\LapTwist^* - i)$.
The pull-back of $v$ to $\widetilde{X}$ is in $L^2(\widetilde{X},\widetilde{E})$ and therefore the kernel of $\tilde{\Delta}^* - i$ is non-trivial, which contradicts the essential self-adjointness of $\tilde{\Delta}$.
Hence, $\LapTwist$ is essentially self-adjoint.
\end{proof}

Henceforth, we will denote the unique self-adjoint extension also by $\LapTwist$.

\subsection{Base-Compactified Bundles and their Sections}
In this section, we define the restriction of vector bundles to the boundary at infinity. 
We continue using the notation from the previous sections. For any $F\subseteq X$ we let
\[
 E_\twist\vert_F \coloneqq \pi_1^{-1}(F)
\]
and denote by $\pi_1\vert_F\colon E_\twist\vert_F \to F$ the restriction of the fiber orbibundle $\pi_1$ to $F$. For any subgroup $\sgroup$ of $\group$ we set $\twist_\sgroup\coloneqq \twist\vert_\sgroup$. If $\sgroup = \ang{\gamma}$ is a cyclic group generated by $\gamma \in \group$, then we also use
\[
 \twist_\gamma \coloneqq \twist_{\ang{\gamma}}\,.
\]

\begin{prop}\label{prop:model-ends}
For each $j \in \{1,\dotsc,n_\bullet\}$ and $\bullet \in \{f,c\}$, we have that
\[
E_\twist\vert_{X_{\bullet,j}}\ \cong\ U\times_{\twist_\gamma} V \subset E_{\twist_\gamma}\,,
\]
where $U$ and $\gamma$ are as in Section~\ref{sec:decomposition} and $E_{\twist_\gamma} = \h \times_{\twist_\gamma} V$ is the vector bundle associated to~$\twist_\gamma$ over the hyperbolic cylinder $\ang{\gamma} \bs \h$.
\end{prop}

\begin{proof}
By definition of $U$, we have that $\pi_\group(U) = X_{\bullet,j}$. Therefore,
\[
E_\twist\vert_{X_{\bullet,j}} = E_\twist\vert_{\pi_\group(U)} = \pi_1^{-1}\bigl(\pi_\group(U)\bigr) = \{ [z,v] \in \h\times_\twist V \setmid z\in\group.U\}\,.
\]
For any $[z,v]\in\h\times_\twist V$ with $z \in \group.U$ we obviously find a representative in~$U\times V$. Further, for any $z_1,z_2\in U$ and $v_1,v_2\in V$ we see that $[z_1,v_1] = [z_2,v_2]$ if and only if there exists $g\in\group$ such that
\begin{equation}\label{eq:find_g}
 g.z_1 = z_2 \quad\text{and}\quad \twist(g)v_1 = v_2\,.
\end{equation}
Since $U$ is $\group$-stable, this condition is equivalent to finding $g\in 
\group_U$ satisfying~\eqref{eq:find_g}. Thus, $[z_1,v_1]$ and $[z_2,v_2]$ are 
equal as elements in $E_\twist$ if and only if they are equal as elements in 
$U\times_{\twist_\gamma} V$.
\end{proof}

\begin{remark}
Under the isomorphisms of Section~\ref{sec:ends}, $\pi_{\ang{\gamma}}(U)$ corresponds to the set $F_\ell \cong \{r > 0\} \subset C_\ell$ if we start with a funnel end, and to $F_\infty \cong \{r > 0\} \subset C_\infty$ in the case of a cusp end.
\end{remark}

Parallel to the scalar case, we can define smooth sections on the compactified space. The definition of the corresponding bundle is analogous to the compactification of the base space.

\begin{defi}
The base-compactified bundle of $E_\twist \to X$ is given by
\begin{align*}
\overline E_\twist = \Omega^*(\group) \times_\twist V\,.
\end{align*}
\end{defi}
We have the natural projection $\overline E_\twist \to \overline{X}$.
Note that $\overline E_\twist$ is not compact, since the typical fiber is $V$.
Since $\overline{X}$ is not a manifold (not even an orbifold) is not a differentiable fiber bundle. This also makes the definition of smooth sections more involved.

Let $j \in \{1, \dotsc, n_c\}$. We set 
\[
n_{c,j}^\twist \coloneqq \dim E_1(\twist(\gamma_j))\,,
\]
where $\gamma_j \in \group$ is the unique element (modulo taking inverses) as defined in Proposition~\ref{prop:model-ends}.
Let $\Pi_j$ denote the orthogonal projection onto $E_1(\twist(\gamma_j))$.
Recall that $\rho$ is the boundary boundary defining function as given in \eqref{eq:rho_def}.
We note that in each funnel and cusp, we may use coordinates $(\rho, \phi)$, where $\phi \in \R /2 \pi \Z$.
We define smooth sections $u \in \CI(\overline{X}, \bundle)$ as smooth sections $u \in \CI(X, \bundle)$ such that
\begin{align*}
u\vert_{X_{c,j}}(\rho, \phi) &= \Pi_j u_{0,j}(\rho) + O(\rho^{\infty})\,, \quad (\rho,\phi) \in X_{c,j}\,,\ j \in \{1,\dotsc, n_c\}
&
\intertext{for some $u_{0,j}\in \CI_b(X_{c,j},\bundle)$ not depending on $\phi$, and}
u\vert_{X_{f,j}} &\in \CI(X_{f,j},\bundle) \text{ extends smoothly to $\overline{X_{f,j}}$}\,.
\end{align*}
Let $u \in \CI(\overline{X}, \bundle)$. We define the restriction 
$u|_{\pa_\infty X}$ of $\phi$ to the boundary at infinity $\pa_\infty X$.
Since the boundary at infinity is a disjoint union of funnel and cusp ends, we 
have can consider both cases separately. For $X_{f,j}$, we define
\begin{align*}
u|_{\pa_\infty X_{f,j}}(\phi) &= \lim_{\rho \to 0} u(\rho, \phi)
\end{align*}
in the coordinates as above.
We define the restriction of $u$ to the cusp end as
\begin{equation*}
u|_{\pa_\infty X_{c,j}} = \lim_{\rho \to 0} \Pi_j\, u(\rho, \phi)\,.
\end{equation*}
Since $u(\rho,\phi) = u_0(\rho) + O(\rho^\infty)$, we see that $u|_{\pa_\infty X_{c,j}}(\rho,\phi) = \lim_{\rho\to 0}\Pi_j u_{0,j}(\rho)$ is independent of $\phi$. We define the space
\begin{align*}
\CI(\pa_\infty X, \bundle) = \{ u|_{\pa_\infty X} \setmid u \in \CI(\overline{X}, \bundle) \}\,.
\end{align*}
If we set
\begin{align}\label{eq:nctwist}
n_c^\twist \coloneqq \sum_{j=1}^{n_c} n_{c,j}^\twist\,,
\end{align}
then we have the decomposition
\begin{align}\label{eq:bundle-boundary}
\CI(\pa_\infty X, \bundle) &\cong \CI(\pa_f X, \bundle) \sqcup 
\C^{n_c^\twist} 
\\
&\cong \bigsqcup_{j=1}^{n_f} \CI(\R/2\pi \Z, \bundle) \sqcup 
\C^{n_c^\twist}\,.\nonumber
\end{align}

We consider the sesquilinear product
\begin{align}\label{eq:sesquilinear-product-pa_infty}
(u,v)_{L^2(\pa_\infty X,\bundle)} = \int_{\pa_f X} (u,v)_V \, d\mu_{\pa_f X} + \sum_{j=1}^{n_c^\twist} u_j \overline{v_j}\,.
\end{align}
The $L^2$-space is the completion of $\CI$ with respect to this inner product 
and has the decomposition
\begin{align*}
L^2(\pa_\infty X, \bundle) \cong L^2(\pa_f X, \bundle) \oplus 
\C^{n_c^\twist}\,.
\end{align*}
We further note that the bilinear product is given by
\begin{align}\label{eq:bilinear-product-pa_infty}
\ang{u,v} = (u, \overline{v})\,.
\end{align}

\section{Model Calculations}\label{sec:model}
In this section, we prove meromorphic continuatibility of the twisted resolvents in the case of the hyperbolic cylinder~$C_\ell$ and the parabolic cylinder~$C_\infty$. These will be used later on to prove the meromorphic continuatibility on a general hyperbolic surface~$X$ using the decomposition of~$X$ into a compact core and model ends.

We start with recalling the definition and the properties of the hypergeometric function in Section \ref{sec:hypergeometric_function}. We recall the explicit expression for the resolvent of the Laplacian on $\h$ in Section~\ref{sec:resolv_plane}. In Section~\ref{sec:fourier_cylinders}, we discuss the Fourier expansion of an integral kernel of the resolvent of the Laplace operator on an abstract cylinder. Unlike in the non-twisted case, the integral kernel will not be periodic in the horocycle direction, that is going to affect the Fourier expansion. 

We will give two alternative proofs of the meromorphic continuation of the integral kernel of the resolvent of the Laplace operator on a model cylinder. The first proof, to be found in Section \ref{sec:twisted_resolvent_on_cylinders} (more precisely, in Proposition~\ref{prop:Rhkernel_conv}), uses a variant of the method of images to provide a meromorphic continuation on an abstract cylinder. Unlike in the existing literature, we treat the cases of the parabolic and the hyperbolic cylinder uniformly. The second approach is based on the Fourier expansions for the model resolvents. The method is more involved but it allows to obtain more information about the location of the poles; see Proposition \ref{prop:fourier-cylinder} for the hyperbolic cylinder, Proposition \ref{prop:fourier-funnel} for the funnel and Proposition \ref{prop:fourier-cusp}. Moreover, this method allows to obtain upper and lower bounds for the resonance counting function for the hyperbolic cylinder in Proposition \ref{prop:rescount_hypcyl} and for the funnel in Remark \ref{rem:upper-bound-funnel}. For the cusp, the only resonance is at $s = 1/2$ with multiplicity equal to the dimension of the $1$-eigenspace of the endomorphism $\twist(p)$ with $p\in\Gamma$ being a generator for a realization of the stabilizer group of the considered cusp. For a more extensive discussion of Fourier decompositions in the presence of non-unitary representations, we refer to \cite{FPR}.

\subsection{Hypergeometric Function}\label{sec:hypergeometric_function}

In this section we recall the necessary background information on the hypergeometric function. We refer to~\cite[Section 9.1]{Olver74} and~\cite[Chapter II, Section 2.1.1]{ErdelyiI}.

We recall that the hypergeometric function $\hgfunc(a,b;c;z)$ is defined for 
$a,b\in \C$, $c \in \C \setminus ( -\N_0)$ and $z\in\C$, $|z| < 1$, by the power 
series 
\begin{align*}
\hgfunc(a,b;c;z) \coloneqq \sum_{n=0}^\infty 
\frac{\gammafunc(a+n)\gammafunc(b+n)\gammafunc(c)}{
\gammafunc(a)\gammafunc(b)\gammafunc(c+n)} \cdot \frac{z^n}{n!}\,,
\end{align*}
and the regularized hypergeometric function is defined for $a,b\in \C$, $c \in \C \setminus ( -\N_0)$ and $z\in\C$, $\abs{z} < 1$, by 
\begin{align}\label{eq:def_FF}
\FF(a,b;c;z) \coloneqq&\; \frac{1}{\gammafunc(c)} \cdot \hgfunc(a,b;c;z)
\\
=&\; \sum_{n=0}^\infty 
\frac{\gammafunc(a+n)\gammafunc(b+n)}{\gammafunc(a)\gammafunc(b)}\frac{1}{
\gammafunc(c+n)} \cdot \frac{z^n}{n!}\,.\nonumber
\end{align}
In fact, the right hand side of the previous formula defines $\FF(a,b;c;z)$ not 
only for $c \in \C \setminus ( -\N_0)$, but for $c \in \C$ (see~\cite[Theorem 9.1]{Olver74}). For $x \in (1,\infty)$ and $s\in\C \setminus (-\N_0)$ we define the function $g_s(x)$ by 
\begin{align*}
g_s(x) &\coloneqq \frac{\gammafunc(s)^2}{4\pi} x^{-s} \FF(s,s;2s;x^{-1})\,.
\end{align*}
Since the Gamma function~$\gammafunc$ is meromorphic on~$\C$ and the regularized hypergeometric function is holomorphic in the first three arguments, $g_s(x)$ is meromorphic in $s \in \C$ with poles at $s = -\N_0$ for each fixed~$x$.

\subsection{Resolvent on the Hyperbolic Plane}\label{sec:resolv_plane}
The resolvent of the Laplacian, $\Delta_\h$, on the hyperbolic plane is given by 
(see, e.g., \cite[Proposition~4.2, Eq.~(4.6)]{Borthwick_book})
\begin{equation}\label{eq:def_Rhkernel}
\begin{aligned}
R_\h(s; z,z') &= \frac{\gammafunc(s)^2}{4\pi} \sigma(z,z')^{-s} 
\FF(s,s;2s;\sigma(z,z')^{-1})
\\
&= g_s(\sigma(z,z'))\,,
\end{aligned}
\end{equation}
where $\sigma$ was defined in~\eqref{eq:sigma_coordfree}, $s\in\C\setminus(-\N_0)$, $z,z'\in\h$, $z\not=z'$. If $\sigma(z,z') > 1$ and $s \in \C \setminus ( -\N_0)$, this formula and the Taylor expansion \eqref{eq:gs_taylor} imply
\begin{align}\label{eq:resolv_altform}
R_\h(s; z,z') = \frac{1}{4\pi} \sum_{n=0}^\infty
\frac{\gammafunc(s+n)^2}{n!\gammafunc(2s+n)} \sigma(z,z')^{-(s+n)}\,.
\end{align}

From the explicit formula for the resolvent, we obtain a precise description of the kernel pulled back to the blow-up space $\h \times_0 \h$ (see~\cite[Lemma~2.1]{GuZw97}). 

In the following proposition we establish the meromorphic continuation of the resolvent~$R_\h$ to all of~$\C$. We will now use the blown-up space $\h \times_0 \h$ (see Section~\ref{sec:blow-up}) to describe the free resolvent. This result is well-known (see, e.g, \cite[Section~6.6]{Borthwick_book}). We let $\rho_0, \rho_0'$ be boundary defining functions of the side faces of $\h \times_0 \h$ that are given by $\beta^*(\pa_\infty \h \times \h)$ and $\beta^*(\h \times \pa_\infty \h)$, respectively. We denote by $\mc I_0^m(\h \times_0 \h, \diagon_0)$ the conormal distribution on $\h \times_0 \h$ of order $m \in \R$ conormal with respect to $\diagon_0$ that vanishes to infinite order on the side faces. We note that a family of operators $R_s\colon L^2_\cpt(\h) \to L^2_\loc(\h)$, defined for $s\in\C$, $\Rea s > x_0$ for some $x_0\in\R$, has a meromorphic continuation to all of~$\C$ if and only if for any $\eta \in \CcI(\h)$, the family $\eta R_s \eta\colon L^2(\h) \to L^2(\h)$ has a meromorphic continuation to all of~$\CC$, by the topologies of~$L^2_\cpt(\h)$ and~$L^2_\loc(\h)$.

\begin{prop}\label{prop:continuation-hyperplane}
The resolvent $s \mapsto R_\h(s)$ initially defined for $\Re s > 1/2$ admits a meromorphic continuation to $s \in \C$ as an operator
\[
R_\h(s) \colon  L^2_\cpt(\h) \to L^2_\loc(\h)\,,
\]
where the poles are at $s \in -\N_0$. For $s \in \C \setminus ( -\N_0)$, the integral kernel~$R_\h(s;\cdot,\cdot)$ of~$R_\h(s)$ satisfies
\begin{align}\label{eq:split_pb_kernel}
\beta^* R_\h(s; \cdot) \in \mc I_0^{-2}(\h \times_0 \h, \diagon_0) + (\rho_0 \rho_0')^s \CI(\h \times_0 \h)\,.
\end{align}
\end{prop}

\begin{proof}
The meromorphic continuation of the resolvent kernel follows from \eqref{eq:def_Rhkernel} since the hypergeometric function is meromorphic and the location of the poles can be read off from~\eqref{eq:resolv_altform}. The integral kernel is locally $L^2$-integrable (see, e.g., \cite[Proposition~4.2]{Borthwick_book}) and hence $R_\h(s) \colon  L^2_\cpt \to L^2_\loc$ as a meromorphic family by the equivalence of weak and strong holomorphy in Banach spaces~\cite[Proposition~A.3]{Arendt}. We can write the pull-back of the resolvent to $\h \times_0 \h$ for $\sigma(z,z') \not = 1$ as
\begin{align*}
(\beta^* R_\h)(s; r, \eta, \eta', \omega, \omega') = 
\frac{\gammafunc(s)^2}{4\pi} \left( \frac{4\eta\eta'}{1 + 2\eta \eta'} \right)^s 
\FF\left(s,s;2s;\frac{4\eta\eta'}{1 + 2\eta\eta'}\right)\,,
\end{align*}
where we used~\eqref{eq:def_Rhkernel} and~\eqref{eq:sigma-blownup}. We choose a smooth function $\psi \in \CI(\h \times_0 \h)$ such that $\psi$ vanishes on a neighborhood of $\diagon_0$. Then $(1 - \psi) \beta^* R_\h$ is conormal with respect to $\diagon_0$ by~\cite{MaMe87}. Since $\diagon_0$ does not intersect the support of~$\psi$ and since $\eta$ and $\eta'$ are locally the boundary defining function for the side faces, we obtain
\[
\psi \beta^*R_\h(s;\cdot) \in (\rho_0\rho_0')^s \CI(\h \times_0 \h)\,. 
\]
This completes the proof.
\end{proof}

The structure of the kernel implies that for each $N \in \N$
\begin{align*}
R_\h(s)\colon  \rho^N L^2(\h) \to \rho^{-N}L^2(\h)\,,
\quad \Re(s) > \frac12 - N\,,
\end{align*}
where $\rho$ is a boundary defining function on $\overline\h$.

\subsection{Twisted Resolvent on Cylinders: Structural Considerations}\label{sec:twisted_resolvent_on_cylinders}

We consider the case of a cylinder 
\[
C_g = \ang{g} \bs \h\,,
\]
where $g \in \Isom^+(\h) \setminus\nobreak \{\id\}$, and of any finite-dimensional representation 
\[
\twist\colon \ang{g} \to \GL(V)\,.
\]
Let $\funddom$ be a geodesically convex fundamental domain for~$C_g$ that contains exactly one point from each orbit in $\h$. Formally, $R_{C_g,\twist}(s)$ is given by its Schwartz kernel
\begin{equation} \label{eq:resolvent-generic_cylinder}
\begin{aligned}
R_{C_g,\twist}(s;z,z') &= \sum_{k \in \Z} \twist(g^k) R_\h(s;z,g^k . z')
\\
&= \id_V R_\h(s; z, z') + \sum_{\substack{k\in \Z\\ k\not=0}} \twist(g^k) 
R_\h(s; z, g^k.z')\,.
\end{aligned}
\end{equation}
If $\twist\colon \ang{g} \to \Unit(V)$ is a unitary representation, then it is well-known (see, e.g., \cite[p.~17]{Venkov_book}) that \eqref{eq:resolvent-generic_cylinder} is the Schwartz kernel of the resolvent. In this case, its lift to $\h \times \h$ is a distribution $u_s \in \CmI(\h \times \h, \Unit(V))$
that satisfies the partial differential equation
\begin{equation}\label{eq:resolvent-pde}
\left\{\begin{aligned}
(\Delta_\h - s(1-s)) u_s(z,z') = \delta(z-z') \id_V 
\\
u_s(g.z,z') = \twist(g) u_s(z,z')
\end{aligned}\right.
\end{equation}
and $u_s$ is symmetric, that is $u_s(z,z')^* = u_s(z', z)$. We note that in the case that $\twist$ is not unitary, we did not define a Laplace operator and therefore have no notion of a (functional analytical) resolvent. However, we will show that \eqref{eq:resolvent-generic_cylinder} is also well-defined in this case. 

We now provide an abstract proof of the convergence of the infinite series 
in~\eqref{eq:resolvent-generic_cylinder}. We will see (further below) that the hypotheses of Proposition~\ref{prop:Rhkernel_conv} are satisfied for hyperbolic cylinders as well as for parabolic cylinders. The abstract framework allows us to treat both situations at once. We also note that for elliptic elements~$g$, Proposition~\ref{prop:Rhkernel_conv} does not apply. However, for elliptic elements, the meromorphic continuation is immediate as the series~\eqref{eq:resolvent-generic_cylinder} becomes a finite sum.

\begin{prop}\label{prop:Rhkernel_conv}
Let $\twist\colon \ang{g} \to \GL(V)$ be a representation, $a,b \in (0,\infty)$, 
and $c \in \R$ with $b - c \geq a$. Suppose that there exists an increasing function $f \colon  \N \to (0,\infty)$, such that
\begin{enumerate}[label=$\mathrm{(\roman*)}$, ref=$\mathrm{\roman*}$]
\item \label{assump:series-converges} 
for all $x \in \R$ with $x < c$, the series
\begin{align*}
\sum_{\substack{k\in\Z \\ k\not=0}} e^{f(\abs{k}) x}
\end{align*}
converges;
\item for all $z,z' \in \funddom$, we have
\begin{align}\label{eq:sigma_asymp}
\sigma(z,g^k.z') \asymp e^{f(|k|)a} \qquad\text{as $k\to\pm\infty$}\,,
\end{align}
with implied constants depending continuously on~$z$ and~$z'$;
\item\label{Rhkerneliii} for all $k\in\Z$, $k\not=0$, we have
\begin{equation}\label{eq:chigk_bound}
\norm*{\twist(g^k)} \lesssim e^{f(|k|)b}\,,
\end{equation}
where the implied constant is independent of $k$.
\end{enumerate}
Then $R_{C_g,\twist}(s)$ is well-defined as an operator
\begin{align*}
R_{C_g,\twist}(s) \colon  L^2_\cpt(C_g,\bundle) \to L^2_\loc(C_g,\bundle)
\end{align*}
for $\Re s > \frac{b - c}{a}$ and $s \not \in -\N_0$. 

If, in addition, there exists a discrete set $\mathcal{L} \subset \C$ such 
that 
\begin{itemize}
\item for every $z,z' \in \funddom$, the function
\begin{align*}
s\mapsto H(s;z,z') \coloneqq \sum_{k \in \Z} \twist(g^k) \sigma(z,g^k.z')^{-s}
\end{align*}
admits a meromorphic continuation, denoted by $H(s;z,z')$ as well, to all of~$\C$ with poles contained in $\mathcal{L}$, and
\item for each $s \in \C \setminus \mathcal{L}$, the map 
\[
(z,z') \mapsto H(s;z,z')
\]
is continuous on~$\h\times\h$,
\end{itemize}
then the resolvent admits a meromorphic continuation to $s \in \C$ as an 
operator
\begin{align*}
R_{C_g,\twist}(s) \colon  L^2_\cpt(C_g,\bundle) \to L^2_\loc(C_g,\bundle)
\end{align*}
and the poles are contained in the union of $-\N_0$ and the set of poles of 
$H(s;z,z')$, as a function of~$s$.
\end{prop}

We split the proof into two parts: the  convergence and the meromorphic 
continuation. For both parts, we have to prove that certain series converge.
This is the content of the following two lemmas.

\begin{lemma}\label{lem:twisted_Rhkernel_conv}
Under the hypotheses of Proposition~\ref{prop:Rhkernel_conv}, the infinite 
series
\[
 \sum_{\substack{k\in\Z \\ k\not=0}} \twist(g^k) R_\h(s;z,g^k.z')
\]
converges compactly and absolutely on
\[
 \left\{ s\in\C\setminus (-\N_0) \setmid \Rea s > \frac{b-c}{a} \right\} 
\times \funddom \times \funddom\,.
\]
\end{lemma}

\begin{proof}
Taking advantage of \eqref{eq:def_Rhkernel}, we can define $R_\h(s; z, g^k.z')$
for any  $k \in \Z$,  $s\in\C\setminus (-\N_0)$ and $z,z'\in\funddom$. We start by studying the growth asymptotics of
\[
\twist(g^k)R_\h(s;z,g^k.z')
\]
as $k\to\pm\infty$. By \eqref{eq:sigma_asymp}, we find $a>0$ be such that
\[
\sigma(z,g^k.z') \asymp e^{f(|k|) a} \qquad\text{as $k\to\pm\infty$}\,,
\]
where the implied constants depend continuously on~$z$ and~$z'$. We note that the convergence of the series in Proposition~\ref{prop:Rhkernel_conv}\eqref{assump:series-converges} for $x<0$ implies that $f(|k|) \to\nobreak\infty$ as $k\to\pm\infty$. Together with~\eqref{eq:def_FF} this yields
\[
\FF\bigl(s,s; 2s; \sigma(z,g^k.z')^{-1} \bigr) \asymp 1 \qquad\text{as 
$k\to\pm\infty$}\,,
\]
where the implied constant depends continuously on~$s$, $z$ and~$z'$.
Substituting this into \eqref{eq:def_Rhkernel}, we obtain
\begin{equation}\label{asymp:Rhszgkz}
R_\h(s; z, g^k.z') = O\bigl( e^{-f(|k|) a \Re(s)} \bigr)\,,
\end{equation}
where the implied constant depends continuously on~$s$, $z$ and~$z'$.
Combining \eqref{eq:chigk_bound} and \eqref{asymp:Rhszgkz} shows 
\[
\twist(g^k)R_\h(s; z, g^k.z') = O\bigl( e^{f(|k|)(b - a\Re(s))} \bigr)\,,
\]
where the implied constant depends continuously on~$s$, $z$ and~$z'$. By hypothesis, 
\[
\sum_{\substack{k\in\Z\\ k\not= 0}} e^{f(|k|) x}
\]
converges for all $x<c$.  It follows that 
\[
\sum_{\substack{k\in\Z \\ k\not=0 }} \twist(g^k) R_\h(s;z, g^k.z') 
\]
converges compactly and absolutely on 
\[
\left\{ s\in\C\setminus(-\N_0) \setmid \Rea s > \frac{b-c}{a}\right\} \times 
\funddom \times \funddom\,.
\]
This completes the proof.
\end{proof}

In order to prove the meromorphic continuation, we will use the Taylor expansion of the hypergeometric function appearing in the Schwartz kernel of $\R_\h(s)$. For $N \in \N_0$, $s \in \C \setminus (-N - \N_0)$, and $x \in (1,\infty)$, 
define the truncated Taylor expansion of $g_s$ by
\begin{equation}\label{eq:gs_taylor}
g_{s,N}(x) = \frac{1}{4\pi} \sum_{n=N}^\infty \frac{\gammafunc(s+n)^2}{n! 
\gammafunc(2s+n)} x^{-(s+n)}\,.
\end{equation}
In particular, we have that $g_{s,0}(x) = g_{s}(x)$. The function $g_{s,N}(x)$ is meromorphic in $s \in \C$ with poles at $s \in (-N - \N_0)$. Further, for any fixed $s \in \C \setminus ( -N - \N_0)$, we have
\begin{equation}\label{eq:gsN_growth}
g_{s,N}(x) = O_s(x^{-(\Re s+N)})\qquad\text{as $x \to \infty$}\,,
\end{equation}
where the implied constant depends continuously on~$s$. Let $N \in \N$ and $s \in \C$ with $\Re s > 0$. We define the distribution $F_N(s) \in \CmI(\h \times \h)$ by
\begin{equation}\label{eq:reshyper-expansion}
\begin{aligned}
F_N(s;z,z') &\coloneqq R_\h(s;z,z') - \frac{1}{4\pi}\sum_{n=0}^{N-1} 
\frac{\gammafunc(s+n)^2}{n! \gammafunc(2s+n)} \sigma(z,z')^{-(s+n)}\,.
\end{aligned}
\end{equation}
For $z,z' \in \h$ with $z \not = z'$, we have that $F_N(s;z,z') = 
g_{s,N}(\sigma(z,z'))$.
Proposition~\ref{prop:continuation-hyperplane} implies that the associated 
operator~$F_N$ is holomorphic on $\{\Re s > -N\}$.

\begin{lemma}\label{lem:twisted_Rhremainder_conv}
    Let $N \in \N_0$ be arbitrary.
    Under the hypotheses of Proposition~\ref{prop:Rhkernel_conv}, the infinite series
    \[
    \sum_{k\in\Z \setminus \{0\}} \twist(g^k) F_N(s;z,g^k.z')
    \]
    converges compactly and absolutely to a holomorphic function on the set $\{s \in \C \setminus\nobreak (-N-\nobreak\N_0) \colon \Rea s > \frac{b-c}{a} - N\}$.
\end{lemma}

\begin{proof}
We have the estimate
\[
\norm*{ \sum_{k\in\Z \setminus \{0\}} \twist(g^k) F_N(s;z,z'+k)} \leq 
\sum_{k\in\Z \setminus \{0\}} \norm*{ \twist(g^k) } \abs*{ F_N(s;z,z'+k) }\,.
\]
Since $z, z' \in \funddom$ and $k \not = 0$, $\sigma(z,g^k.z') \not = 1$ and we 
have that $F_N(s;z,g^k.z') = g_{s,N}(\sigma(z,g^k.z')$. Hence, using \eqref{eq:gsN_growth} and \eqref{eq:sigma_asymp}, we obtain the estimate
\[
\abs{F_N(s;z,g^k.z')} \asymp \sigma(z,g^k.z')^{-(\Rea s +N)} \asymp 
e^{-f(\abs{k}) a(\Rea s +N)}
\]
uniformly on any compact subset of $\C\times \funddom \times \funddom$. Together 
with Proposition~\ref{prop:Rhkernel_conv}\eqref{Rhkerneliii}, we obtain~\eqref{eq:sigma_asymp}. Thus,
\[
\sum_{k\in\Z \setminus \{0\} } \left\| \twist(g^k) F_N(s;z,g^k.z')\right\| 
\lesssim \sum_{k\in\Z \setminus \{0\}} e^{f(\abs{k}) (b - a(\Rea s + N))}\,.
\]
The latter series converges for
\[
\Rea s > \frac{b - c}{a} - N\,.
\]
Since $F_N$ is holomorphic on $\{\Rea s > -N\}$ with the exception of the points in $-(N + \N_0)$, the infinite sum
\[
  \sum_{k\in\Z \setminus \{0\}} \twist(g^k) F_N(s;z,g^k.z')
\]
defines a holomorphic function on $\{s \in \C \setminus (-N - \N_0) \setmid \Rea 
s > \frac{b-c}{a} - N\}$.
\end{proof}

With these preparations, we can now provide a proof of Proposition~\ref{prop:Rhkernel_conv}.

\begin{proof}[Proof of Proposition~\ref{prop:Rhkernel_conv}]
Using \eqref{eq:resolvent-generic_cylinder}, the first claim follows directly from Proposition~\ref{prop:continuation-hyperplane} and Lemma~\ref{lem:twisted_Rhkernel_conv}, since the infinite series in \eqref{eq:resolvent-generic_cylinder} defines a smoothing integral operator when acting on compactly supported distributions. For the meromorphic continuation, we use \eqref{eq:reshyper-expansion} to find for any $N\in\N$ (and all $s\in\C$, $(z,z')\in \funddom \times \funddom$ in the domain of convergence)
\begin{align*}
R_{C_g, \twist}(s;&\, z,z') 
\\
&= \sum_{k \in \Z} \twist(g^k) R_\h(s;z,g^k.z') 
\\
& = \frac1{4\pi}\sum_{n=0}^{N-1} \frac{\gammafunc(s+n)^2}{n! 
\gammafunc(2s+n)} H(s+n;z,z') + \sum_{k\in\Z} \twist(g^k) F_N(s;z,g^k.z')\,.
\end{align*}
By Lemma~\ref{lem:twisted_Rhremainder_conv} the second sum without the $k = 
0$ summand converges for $\Re s > \frac{b-c}{a} - N$. For the case $k = 0$, we have that
\[
\left. \twist(g^k) F_N(s;z,g^k.z')  \right|_{k = 0} = \id_V F_N(s;z,z')\,,
\]
which is holomorphic on $\Re(s) > - N$.
\end{proof}

\subsection{Fourier Expansion of Sections on Cylinders}\label{sec:fourier_cylinders}
As in the case of the meromorphic continuation, we can treat parts of the Fourier expansion uniformly for hyperbolic and parabolic cylinders.
In this section, we will assume that the representation is unitary.
Denote by $C_\bullet$ the model (hyperbolic or parabolic) cylinder with $\bullet \in \R \cup \{\infty \}$. Under the identification $C_\bullet \cong \R \times (\R / 2\pi \Z)$ the generating isometry becomes
\[
h \colon (r,\phi) \mapsto (r,\phi + 2\pi)\,.
\]
Let $B \in \Unit(V)$ with eigenvectors $(\psi_j)_{j=1}^n$ and eigenvalues $(e^{2\pi i\vartheta_j})_{j=1}^n$.
If a representation $\twist\colon \ang{h} \to \Unit(V)$ is a generated by a $\twist(h) = B$, then every section $f \in \CI(C_\bullet, \bundle)$ is periodic in the sense
\begin{align*}
\ang{f(r,\phi+2\pi), \psi_j} = e^{2\pi i \vartheta_j} \ang{f(r,\phi), \psi_j}\,.
\end{align*}
This implies that we obtain a Fourier-type expansion of the form
\begin{align*}
\ang{f(r,\phi), \psi_j} = \sum_{k \in \Z} e^{i\phi(k + \vartheta_j)} f_{k,j}(r)\,.
\end{align*}
Consequently, if a bounded linear operator $A\colon \CI(C_\bullet, \bundle) \to \CI(C_\bullet, \bundle)$ with integral kernel $K_A$ is diagonal in the basis $\{\psi_j\}$, then each component $K_A(r,\phi,r',\phi')^j \coloneqq \ang{K_A(r,\phi,r,\phi') \psi_j, \psi_j}$ admits a Fourier-type expansion
\begin{align}\label{eq:kernel_fourier-exp}
K_A(r,\phi,r',\phi')^j = \sum_{k \in \Z} e^{i\phi(\vartheta_j + k)} a_k^j(r,r') e^{-i\phi'(\vartheta_j + k)}\,.
\end{align}
Of course the Fourier expansion is also valid if $A$ is not an integral operator, then we have to use the Schwartz kernel theorem and the Fourier coefficients are distributions on $\R \times \R$.

\subsection{Hyperbolic Cylinder}
To prove the meromorphic continuation of the resolvent of the hyperbolic cylinder $C_\ell = \ang{h_\ell}\bs \h$, we will use the fundamental domain $\funddom = \{ z\in \h \setmid 1\leq \abs{z} < e^{\ell} \}$ from Section~\ref{sec:ends}. Moreover, we note that the hyperbolic cylinder consists of two funnel ends in the sense of Section~\ref{sec:funnel_ends}, the model funnel $F_\ell$ is given by $F_\ell = C_\ell \cap \{r > 0\}$ under the isomorphism~\eqref{eq:cylinder-iso}.

Let $V$ be a $d$-dimensional vector space with some fixed norm $\norm{\cdot}$ 
and $A \in \GL(V)$ an invertible endomorphism of~$V$. We use $\EV(A)$ to denote the multiset of eigenvalues of $A$, thus eigenvalues are counted with multiplicity. We also denote the operator norm on $\GL(V)$ by $\norm{\cdot}$.
We consider the (not necessarily unitary) representation $\twist\colon \ang{h_\ell} \to \GL(V)$ defined by
\begin{align*}
\twist (h_\ell)\coloneqq A\,.
\end{align*}
We recall that $\group = \ang{h_\ell}$ acts on $\h \times V$ via $h.(z,v)=(hz, \twist(h)v)$ for $h \in \group$ and this action defines the vector bundle
$\EchiCl \coloneqq \ang{h_\ell} \bs (\h \times V) \to C_\ell$.

We provide the twisted resolvent of the hyperbolic cylinder as its lift to~$\h$.
To that end, we pick a geodesically convex fundamental domain~$\funddom$ for the 
hyperbolic cylinder~$C_\ell$ in~$\h$. One possible choice is the set in~\eqref{eq:funddom_hypcyl}. In the coordinates of~$\h$, the Schwartz kernel of the resolvent on~$C_\ell$ is (formally) given by
\begin{equation}\label{eq:resolvent_hypcyl}
R_{C_\ell,\twist}(s; z,z') = \sum_{k\in \ZZ} A^k R_\h(s; z,e^{k\ell} z')
\end{equation}
for $z,z' \in \funddom$. In the following proposition, we prove that the resolvent, $R_{C_\ell,\twist}(s)$, is well-defined in some half-plane and admits a meromorphic continuation to $s \in \C$. In the case when the representation $\twist$ is trivial, we recover the well-known classical result (see, e.g., \cite[Proposition 5.1]{Borthwick_book}). We set
\begin{align*}
C \coloneqq \ell^{-1} \max\{\log \norm{\twist(h_\ell)}, \log 
\norm{\twist(h_\ell)^{-1}}\}\,.
\end{align*}
We fix a choice of logarithm defined on all of~$\C\setminus\{0\}$ (it does not need to be continuous). None of our results is qualitatively affected by this choice.

\begin{prop}\label{prop:r_hypcyl_merom}
The resolvent $R_{C_\ell, \twist}(s)$ is well-defined for $\Re s > C$ and 
admits a meromorphic continuation to $s \in \C$ as an operator
\begin{align*}
R_{C_\ell, \twist}(s)\colon  L^2_\cpt(C_\ell,\EchiCl) \to L^2_\loc(C_\ell,\EchiCl)
\end{align*}
The multiset of poles of $s\mapsto R_{C_\ell,\twist}(s)$ shall be denoted by $\ResSet_{C_\ell,\twist}$. It is contained (as a set) in
\begin{equation}\label{eq:poles_hcyl}
    -\N_0 \cup \bigcup_{\lambda \in \EV(\twist(h_\ell))} \bigcup_{p \in \{\pm 1\}}
    \left(-\N_0 + p\ell^{-1}\left(\log \lambda + 2\pi i \Z \right) \right)\,.
\end{equation}
Let $\psi \in \CI_b(C_\ell \times_0 C_\ell)$ with $\supp \psi \cap \diagon_0 = \emptyset$. For $s \not \in \ResSet_{C_\ell,\twist}$, the kernel of $R_{C_\ell,\twist}(s)$ satisfies
\begin{align*}
\psi\beta^* R_{C_\ell,\twist}(s;\cdot,\cdot) &\in
(\rho_0 \rho_0')^s \CI(C_\ell \times_0 C_\ell, \bundle \boxtimes_0 \bundle) 
\\
&\hphantom{\in (\rho_0 \rho_0')^s} + \beta^*\left( (\rho_f \rho_f')^s \CI(\overline{C_\ell} \times \overline{C_\ell}, \bundle \boxtimes \bundle) \right)\,.
\end{align*}
\end{prop}

\begin{proof}
Formula~\eqref{eq:sigma} for the point-pair invariant~$\sigma$ implies that
\begin{equation}\label{eq:more_detailed_formula_for_sigma}
\sigma(z, e^{k\ell} z') = 1 + \frac{\abs{z-e^{k\ell}z'}^2}{4 y y' e^{k\ell}}\,.
\end{equation}
Therefore we have the asymptotics
\[
 \sigma(z,e^{k\ell}z') \asymp e^{|k|\ell}\qquad\text{as $k\to\pm\infty$}\,,
\]
where the implied constants depend continuously on~$z$ and~$z'$. We can now use Proposition~\ref{prop:Rhkernel_conv} with $f(n) = \ell n$, $a = 1$, $b = C$, and $c = 0$ to show that $R_{C_\ell,\twist}(s)$ is well-defined for $\Re s > C$.

For the meromorphic continuation, we will not use Lemma~\ref{lem:twisted_Rhremainder_conv} as we can obtain a slightly stronger result. By \eqref{eq:more_detailed_formula_for_sigma}, we can write
\begin{align*}
\sigma(z, e^{k\ell} z')^{-s} = (yy')^{s} e^{-s k\ell}
\left( a + b e^{-k\ell} + c e^{-2k\ell} \right)^{-s}\,,
\end{align*}
where $a = (1/4)\abs{z'}^2$, $b = (1/2)(yy' - xx')$ and $c = (1/4) \abs{z}^2$.
We note that $a, c > 1/4$ on $\funddom \times \funddom$. By the same arguments as in Lemma~\ref{lem:twisted_Rhremainder_conv}, we obtain that
\begin{align*}
\abs{(yy')^{-s}R_N(s;z,e^{k\ell}z')} \asymp (yy')^N e^{-k\ell (\Re s + N)}\,,
\end{align*}
where the implied constant depend on $s$, but not on $z,z'$. This implies that
\begin{align*}
(yy')^{-s}\sum_{k \in \Z \setminus\{0\}} A^k F_N(s;z,e^{k\ell}z')
\end{align*}
converges uniformly on $\funddom \times \funddom$ for $\Re s > C - N$ and $s \not \in -\N$. 

We also have to show that
\begin{align}\label{eq:def_Hltwist}
H_{\ell,\twist}(s;z,z') \coloneqq \sum_{k \in \ZZ \setminus\{0\}} A^k 
\sigma(z, e^{k\ell} z')^{-s}
\end{align}
admits a meromorphic continuation to $s \in \C$. To this end, write for $M > 0$ to be chosen later,
\begin{align}\label{eq:split_Hltwist}
H_{\ell,\twist}(s;z,z') = \sum_{\substack{|k| < M\\k \not = 0}} A^k 
\sigma(z, e^{k\ell} z')^{-s} + \sum_{|k| \geq M} A^k \sigma(z, e^{k\ell} z')^{-s}\,.
\end{align}
The finite sum is obviously entire as a function of~$s$. It suffices to show that the infinite sum on the right hand side of~\eqref{eq:split_Hltwist} admits a meromorphic continuation.

For this we note that
\begin{align*}
(yy')^{-s}e^{s k\ell} \cdot \sigma(z,e^{k\ell}z')^{-s} = 
\left( a + b e^{-k\ell} + c e^{-2k\ell} \right)^{-s}
\end{align*}
is holomorphic as a function in the variable~$e^{-k\ell}$ for $k \gg 1$ (depending on $s,z,z'$).
We may choose $M \geq 0$ large enough and use the Taylor expansion for all $k \geq M$ to write 
\begin{equation}\label{the_series_in_sigma_via_alpha_s}
\begin{aligned}
(yy')^{-s}\sigma(z, e^{k\ell} z')^{-s} &=
e^{-sk\ell} \sum_{p=0}^\infty \alpha_p(s) (e^{-k\ell})^p
\\
&= \sum_{p=0}^\infty \alpha_p(s) e^{-(s+p)k\ell}
\end{aligned}
\end{equation}
for suitable entire functions $\alpha_p\colon  \C \to \C$, $p\in \N_0$, 
depending on $z,z' \in \overline\funddom$. The series above is absolutely convergent for any $s\in \C$.  The absolute convergence allows us to interchange the order of summation (for $s\in\C$ with $\Rea s > C$) to obtain
\begin{align} 
(yy')^s\sum_{k \ge M} A^k &\sigma(z, e^{k\ell} z')^{-s} \nonumber
\\
&= \sum_{k \ge M} A^k  
\sum_{p=0}^\infty \alpha_p(s) e^{- (s+p) k\ell} \label{eq:sumofAksigma-s} 
\\
&= \sum_{p=0}^\infty  \alpha_p(s) \sum_{k \ge M} A^ke^{- (s+p) k\ell} 
\nonumber  \\
&= \sum_{p=0}^\infty \alpha_p(s) (\id_V - e^{-(s+p)\ell} A)^{-1} 
(e^{-(s+p) \ell} A)^M. \label{eq:sumofAksigma-s_last_line} 
\end{align}
We note that for every $p \in \N_0$, the operator $(e^{-(s+p) \ell} A)^M$ is 
entire in $s\in \C$. Additionally, we note that for $\Re(s) > - p + C$, we have
\[
\norm{e^{-(s+p)\ell} A} \le e^{-(s+p-C)\ell} < 1
\]
Thus the operator $\id_V - e^{-(s+p)\ell} A$ is invertible. This implies that the operator $(\id_V - e^{-(s+p)\ell} A)^{-1}$ is holomorphic on $\{\Re(s)>-p+C\}$. Moreover, the operator $(\id_V - e^{-(s+p)\ell} A)^{-1}$ is meromorphic on~$\C$ with the set of poles being
\begin{equation}\label{set:sforwhichcertaindenominatorhasazero}
\bigcup_{\lambda \in \EV(A) } \left(-p + \ell^{-1} (\log\lambda+2 \pi i 
\Z) \right)\,.
\end{equation}
We note that for any fixed $N \in \N_0$, the operators $\{(\id_V - 
e^{-(s+p)\ell}A)^{-1} \}_{p \ge N}$ are holomorphic on the half-plane
\[
\Re(s)>-N+ C\,.
\]
We are left to prove that \eqref{eq:sumofAksigma-s_last_line} is meromorphic 
on $\Re(s)>-N+ C$. To this end, we split the summation in \eqref{eq:sumofAksigma-s_last_line} into two parts, depending on whether  $p < N$ or $p \ge N$. The former part is obviously meromorphic and the poles are contained in the set
\begin{equation}\label{eq:plus_set_of_poles}
\bigcup_{\lambda \in \EV(A) } \left(-\N_0 + \ell^{-1} (\log(\lambda)+2 \pi i \Z) \right)\,.
\end{equation} 
The latter can be rewritten as 
\begin{align*}
\sum_{p=N}^\infty \alpha_p(s) & (\id_V - e^{-(s+p)\ell} A)^{-1} 
(e^{-(s+p)\ell} A)^M  
\\ 
&= \sum_{p=0}^\infty \alpha_{p+N}(s) (\id_V - e^{-(s+p+N)\ell} A)^{-1} (e^{-(s+p+N)\ell} A)^M\,.
\end{align*}
Doing the calculation in~\eqref{eq:sumofAksigma-s} backwards shows that this series converges absolutely on $\{\Rea s > -N + C\}$. For $k < -M$ with sufficiently large $M$ an analogous calculation leads to
\begin{align*} 
(yy')^s\sum_{k \le - M} & A^k  \sigma(z, e^{k\ell} z')^{-s} 
\\
& = \sum_{p=0}^\infty \beta_p(s) (\id_V - e^{-(s+p) \ell} A^{-1})^{-1} (e^{-(s+p)\ell} A^{-1})^M \,,
\end{align*}
which is meromorphic on $\Re(s)>-N+ C$ with poles contained in the set
\begin{align}\label{eq:minus_set_of_poles}
\bigcup_{\lambda \in \EV(A) } \left(-\N_0 - \ell^{-1} (\log(\lambda)+2 \pi i \Z) \right)\,.
\end{align}
Increasing $N$ and combining \eqref{eq:plus_set_of_poles} and \eqref{eq:minus_set_of_poles}, we obtain that the right hand side of \eqref{eq:split_Hltwist} is meromorphic on $\C$ with possible poles at
\begin{equation*}
    s \in \bigcup_{\lambda \in \EV(A)} \bigcup_{p \in \{\pm 1\}} \left(-\N_0 + p \ell^{-1} (\log(\lambda)+2 \pi i \Z) \right)\,.
\end{equation*}
This completes the proof of the meromorphic continuation of~\eqref{eq:split_Hltwist}. Similar to the proof of Proposition~\ref{prop:Rhkernel_conv}, we decompose
\begin{align*}
R_{C_\ell,\twist}(s;z,z') &= \id_V R_\h(s;z,z') + \sum_{k \in \Z\setminus\{0\}} R_\h(s;z,e^{k\ell}z') 
\\
&= \id_V R_\h(s;z,z') + \frac{1}{4\pi} \sum_{n=0}^{N-1} \frac{\gammafunc(s+n)^2}{n!\,\gammafunc(2s+n)} H_{\ell,\twist}(s+n;z,z') 
\\
&\hphantom{= \id_V R_\h(s;z,z')\ } + \sum_{k\in \Z\setminus\{0\}} A^k F_N(s;z,e^{k\ell}z')\,.
\end{align*}
All terms on the right-hand side are meromorphic and this gives the meromorphic continuation of $R_{C_\ell,\twist}(s,z,z')$. Since $(x,y) \mapsto y$ restricted to $\funddom$ is a boundary defining function of $C_\ell$, we conclude that
\begin{align*}
R_{C_\ell,\twist}(s;\cdot,\cdot) \in 
(\rho_f \rho_f')^s \CI(\overline{C_\ell} \times \overline{C_\ell}, \bundle \boxtimes \bundle)
\end{align*}
away from the diagonal. Combining this with the structure of the free resolvent, Proposition~\ref{prop:continuation-hyperplane}, this yields the asserted structure of the resolvent kernel.
\end{proof}

\subsection{Fourier Analysis of the Hyperbolic Cylinder}
In this section, we calculate the resonances of~$\Delta_{C_\ell,\twist}$, or equivalently, the poles of~$R_{C_\ell,\twist}$. We use an approach different from the one in Proposition \ref{prop:r_hypcyl_merom}; namely, we calculate the Fourier coefficients of $R_{C_\ell,\twist}(s)$ in a more direct way. Using the Fourier coefficients, we will be able to calculate the exact multiplicities of the resonances as opposed to the mere inclusion statement in 
Proposition~\ref{prop:r_hypcyl_merom}. We will restrict our considerations to the case that $\twist$ is unitary. We recall from Section~\ref{sec:funnel_ends} that in geodesic coordinates (see \eqref{eq:cylinder-iso}) the metric on $C_\ell$ is given by
\begin{align*}
g_{C_\ell} = dr^2 + \omega^{-2} \cosh(r)^2 \,d\phi^2\,,
\end{align*}
where $\omega = 2\pi / \ell$. Hence, the volume form is given by
\begin{align*}
d\mu_{C_\ell} = \omega^{-1} \cosh(r) \,dr\,d\phi\,,
\end{align*}
and the Laplacian acting on functions is
\begin{align*}
\Delta_{C_\ell} = -\pa_r^2 - \tanh(r)\pa_r - \frac{\omega^2}{\cosh^2(r)} \pa_\phi^2\,.
\end{align*}
Therefore the differential equation for the resolvent kernel becomes
\begin{equation}\label{eq:res-eq-hyper0}
(\Delta_{C_\ell} - s(1-s)) R_{C_\ell,\twist}(s;r,\phi,r',\phi') = 
\frac{\omega \id_V}{\cosh(r)} \delta(r-r') \delta(\phi - \phi')\,,
\end{equation}
where $\phi,\phi' \in [0,2\pi)$ and $\Delta_{C_\ell}$ acts on each component of 
$R_{C_\ell,\twist}(s;\cdot,\cdot, r', \phi')$. Let $(\psi_j)_{j=1}^n$ be an eigenbasis of $\twist(h_\ell)$ with eigenvalues $\lambda_j = e^{2\pi i\vartheta_j}$. We can define the twisted delta-distribution by
\begin{align*}
\hat{\delta}(\cdot) \psi_j \coloneqq \hat{\delta}_{\vartheta_j}(\cdot) \psi_j\,,
\end{align*}
where 
\begin{align*}
\hat{\delta}_{\vartheta_j}(\cdot) \coloneqq \sum_{n \in \Z} e^{i \vartheta_j (\cdot )}\delta(\cdot - 2 \pi n)\,.
\end{align*}
We note that $\hat\delta(\phi - \phi')$ is a distribution in $\CmI(C_\ell \times C_\ell,\bundle \boxtimes \bundle)$. With this, the equation \eqref{eq:res-eq-hyper0} may be rewritten as
\begin{equation}\label{eq:res-eq-hyper}
(\Delta_{C_\ell} - s(1-s)) R_{C_\ell,\twist}(s;r,\phi,r',\phi') = 
\frac{\omega}{\cosh(r)} \delta(r-r') \hat{\delta}(\phi - \phi')\,.
\end{equation}
We set 
\begin{align}\label{eq:RCell_matcoeff}
R_{C_\ell,\twist}(s;z,z')^j \coloneqq \ang{R_{C_\ell,\twist}(s;z,z') \psi_j, \psi_j}\,.
\end{align}
By the arguments leading to \eqref{eq:kernel_fourier-exp}, the Schwartz kernel of the resolvent admits a Fourier decomposition (in $\phi-\phi'$)
\begin{align}\label{eq:fourier-exp-hyper}
R_{C_\ell,\twist}(s;z,z')^j = \frac{1}{\ell} \sum_{k \in \Z} e^{i \phi(k+\vartheta_j) } b_k^j(s;r,r') e^{-i\phi' (k+\vartheta_j)}
\end{align}
for some $b_k^j(s;\cdot ,\cdot) \in \CmI(\R \times \R)$. We want to find the explicit expression for these functions. 

For $\kappa \in \R$, $r\in \R$ and $s \in \C \setminus \left(\Z \pm i \omega \kappa\right)$, we define the functions 
\begin{align}
a_\kappa(s) &\coloneqq 2^{-2s} \gammafunc(s + i\omega\kappa) \gammafunc(s - 
i\omega\kappa)\,,\label{eq:a_kappa_definition}
\\
v_\kappa(s;r) &\coloneqq (\cosh r)^{-s} \FF\left(s + i\omega\kappa, s - 
i\omega\kappa; s +\frac{1}{2}; \frac{ 1 - \tanh r}{2}\right)\,, 
\label{eq:v_definition}
\\
v_\kappa(s;r,r') & \coloneqq
\begin{cases}
a_\kappa(s)v_\kappa(s;-r)v_\kappa(s;r') & \text{for $r\leq r'$,}
\\
a_\kappa(s)v_\kappa(s;r)v_\kappa(s;-r') & \text{for $r'\leq r$}\,.
\label{eq:v_kappa_definition}
\end{cases}
\end{align}

\begin{lemma}\label{lem:fourier-cylinder}
The resolvent admits a Fourier-type expansion
\begin{align}\label{eq:fourier-exp-hyper2}
R_{C_\ell,\twist}(s;z,z')^j = \frac{1}{\ell} \sum_{k \in \Z} e^{i \phi(k+\vartheta_j) } b_k^j(s;r,r') e^{-i\phi' (k+\vartheta_j)}\,.
\end{align}
\end{lemma}

\begin{remark}
We note that the choice of the branch of the logarithm to define~$\vartheta_j$ 
does not affect the series \eqref{eq:fourier-exp-hyper2}.
\end{remark}

\begin{proof}[Proof of Lemma~\ref{lem:fourier-cylinder}]
We set 
\begin{align*}
L_r^0 \coloneqq L_r^0(s) &\coloneqq - \partial_r^2 - \tanh r \cdot \partial_r - 
s(1-s)\,.
\end{align*}
Applying $\Delta_{C_\ell} - s(1-s)$ to  \eqref{eq:fourier-exp-hyper}, we obtain  
\begin{align*}
(\Delta_{C_\ell} - s(1-s))&R_{C_\ell, \xi} (s;r,\phi, r', \phi')^j 
\\  
& = (\Delta_{C_\ell} - s(1-s)) \frac{1}{\ell} \sum_{k \in \Z} e^{i \phi (k+\vartheta_j)} b_k^j(s;r,r') e^{-i \phi' (k+\vartheta_j)}  
\\
&=\frac{1}{\ell} \sum_{k\in \Z} e^{i (\phi-\phi') (k+\vartheta_j)} \left(L_r^0 + \frac{\omega^2 (k+\vartheta_j)^2}{\cosh^2(r)}\right)b_k^j(s;r,r')\,.
\end{align*} 
We note that for an orthonormal basis, $(\psi_j)_{j}$, and any $j = 1, \ldots, n$, 
\[
\ang*{\frac{\omega}{\cosh(r)} \delta(r-r') \hat{\delta}(\phi - \phi') \psi_j , \psi_j} = \frac{\omega }{\cosh(r)} \delta(r-r') \hat{\delta}_{\vartheta_j}(\phi - \phi') \,.
\]
We can use the Fourier series for the untwisted delta distribution to obtain
\begin{align*}
\frac{\omega }{\cosh(r)} \delta(r-r') \hat{\delta}(\phi - \phi')
&= \frac{\delta(r-r')}{\cosh(r)} \cdot \frac{1}{\ell}\sum_{k \in \Z} e^{i (k+\vartheta_j) (\phi-\phi')} \,.
\end{align*}
Comparing the Fourier coefficients, we obtain the 
differential equation
\[
\frac{\delta(r-r')}{ \cosh(r)} = \left(L_r^0 + \frac{\omega^2 (k+\vartheta_j)^2}{\cosh^2(r)}\right)b_k^j(s;r,r')
\]
for $k \in \Z$. The fundamental solution is given by $v_{k + \vartheta_j}(s;r)$. Applying the standard argument involving the Wronskian 
(see~\cite[Proposition~5.2]{Borthwick_book}), we obtain that
\begin{align}\label{eq:fourier-cylinder-diag}
b_k^j(s;r,r') \coloneqq  v_{k + \vartheta_j}(s;r,r')
\end{align}
solves the equation.
\end{proof}

\begin{prop}\label{prop:fourier-cylinder}
The multiset of resonances for the hyperbolic cylinder is given by
\begin{align*}
\ResSet_{C_\ell,\twist} = \bigcup_{\lambda \in \EV(\twist(h_\ell))} \bigcup_{p \in \{\pm 1\}} 
\left(-\N_0 + p \ell^{-1} \left( \log\lambda + 2\pi i \Z\right) \right)\,,
\end{align*}
where $\EV(\twist(h_\ell))$ is the multiset of eigenvalues of $\twist(h_\ell)$.
\end{prop}

\begin{remark}
For $\twist$ being the trivial one-dimensional representation, we recover the classical case: since $\log\lambda = 0$, the resonances are given by $s = -k + 2\pi i\ell^{-1} m$ for $k \in \N_0$ and $m \in \Z$ . Each resonance is of multiplicity~$2$.
\end{remark}

\begin{example}
To illustrate the result in the case of $\twist$ being a non-trivial twist, we consider the diagonal matrix $\twist(h_\ell) = \diag(i, -1)$. For the eigenvalue~$\lambda_1 = i$, we have that $\log \lambda_1 + 2\pi i \Z = 2\pi i (1/4 + \Z)$. Hence, the eigenvalue $\lambda_1$ contributes to resonances at $-\N_0 + \pi i \ell^{-1}(1/2 + \Z)$ with multiplicity $1$. For $\lambda_2 = -1$, we calculate $\log \lambda_2 + 2\pi i \Z = 2\pi i (1/2 + \Z)$ and therefore the eigenvalue contributes to resonances located at $-\N_0 + 2\pi i \ell^{-1} (1/2 + \Z)$ with multiplicity $2$ since $1/2 + \Z = - (1/2 + \Z)$. Summing up, for $\twist(h_\ell) = \diag(i, -1)$, the resonances are given as a set by
\begin{align*}
\ResSet_{C_\ell,\twist} &= -\N_0 + \frac{\pi i}{2 \ell} \left( \Z\setminus 4\,\Z\right) \,.
\end{align*}
\end{example}

\begin{proof}[Proof of Proposition~\ref{prop:fourier-cylinder}]
By \eqref{eq:fourier-cylinder-diag} it suffices to calculate the poles of the function $v_\kappa(\cdot,r,r')$, which is given by \eqref{eq:v_kappa_definition}.
First, we note that $v_\kappa(s;\cdot)$ does not vanish identically.
For this, we may use the $0$-th order Taylor expansion of the hypergeometric function (see~\cite[(5.10)]{Borthwick_book}) to observe that
\begin{align*}
v_\kappa(s;r) \sim \frac{\cosh(r)^{-s}}{\gammafunc(s+1/2)} \sim \frac{2^s e^{-rs}}{\gammafunc(s+1/2)}\,,\quad r \to +\infty\,.
\end{align*}
Moreover, the function $v_\kappa(\cdot;r)$ is entire. Therefore the poles of $v_\kappa(\cdot, r, r')$ are exactly poles of $a_\kappa$ defined in \eqref{eq:a_kappa_definition}.
The poles of $a_\kappa(s)$ are given by the multiset $-\N_0 \pm 2\pi i\ell^{-1}\kappa$.
\end{proof}

Proposition~\ref{prop:fourier-cylinder} implies upper and lower bounds on the 
number of resonances of the same (polynomial) order as in the untwisted 
situation.

\begin{prop}\label{prop:rescount_hypcyl}
Denote by $C_\ell = \ang{h_\ell} \bs \h$ the hyperbolic cylinder and let $\twist\colon \ang{h_\ell} \to \Unit(V)$ be a unitary representation. The resonance counting function satisfies
\begin{align*}
N_{C_\ell,\twist}(r) \asymp r^2 \qquad\text{as $r\to\infty$}\,.
\end{align*}
\end{prop}

\subsection{Model Funnel}\label{sec:model-funnel}
We consider the model funnel 
\[
F_\ell \coloneqq C_\ell \cap \{(r,\phi) \in C_\ell \setmid r \geq 0\}\,.
\]
We restrict the Laplacian on $C_\ell$ to $F_\ell$ and impose Dirichlet boundary conditions at $r = 0$. This gives a self-adjoint operator, which we denote by $\Delta_{F_\ell,\twist}$. We denote its resolvent by $R_{F_\ell,\twist}(s)$. Let $(\psi_j)_{j=1}^n$ be a set of eigenvectors of $\twist(h_\ell)$ with respective eigenvalues $e^{2 \pi i \vartheta_j}$ and $\vartheta_j \in [0,1)$. We define
\begin{align*}
R_{F_\ell,\twist}(s;z,z')^j \coloneqq \ang{R_{F_\ell,\twist}(s;z,z') \psi_j, \psi_j}\,.
\end{align*}
As in the untwisted case (see, e.g., \cite[Section 5.2]{Borthwick_book}), the Fourier decomposition of $R_{F_\ell,\twist}(s;z,z')$ can be calculated using the method of images, 
\begin{align}\label{eq:images}
R_{F_\ell,\twist}(s;z,z') = R_{C_\ell,\twist}(s;z,z') - R_{C_\ell,\twist}(s;z,-\overline{z'})\,,
\end{align}
where $z,z' \in \funddom = \{ z \in \h \colon 1 \leq \abs{z} < e^\ell\}$.
Hence, $R_{F_\ell,\twist}(s)$ admits a meromorphic continuation to $s \in \C$ and with the multiset of poles denoted by~$\ResSet_{F_\ell,\twist}$.
Recall $\omega = 2 \pi/ \ell$.
For $r, r' \ge 0$, $k \in \Z$, and $\kappa \in \R$, the Fourier-type expansion of~$R_{F_\ell,\twist}(s;z,z')$ can be formulated with the help of the functions
\begin{align}\label{eq:vtilde-kappa}
\tilde{v}_\kappa(s;r,r') \coloneqq 
\begin{cases}
\beta_{\kappa}(s) v_{\kappa}^0(s;r) v_{\kappa}(s;r') & \text{for $0 \leq r \leq r'$}\,,
\\
\beta_{\kappa}(s) v_{\kappa}(s;r) v_{\kappa}^0(s;r') & \text{for $0 \leq r' \leq r$}\,,
\end{cases}
\end{align}
where for $s \in \C \setminus \left( -1 - 2 \N_0 \pm i \omega \kappa \right)$,
\begin{align*}
\beta_\kappa(s) & \coloneqq \frac{1}{2} \gammafunc(\frac{s + i \omega \kappa + 1}{2}) \gammafunc(\frac{s - i \omega \kappa + 1}{2})\,,
\\
v_\kappa^0(s;r) & \coloneqq \tanh(r) (\cosh(r))^{-s}
\FF( \frac{s + i \omega \kappa + 1}{2}, \frac{s - i \omega \kappa + 1}{2}; \frac32; \tanh(r)^2)
\end{align*}
and $v_\kappa(s;r)$ is as in \eqref{eq:v_definition}. Therefore, we obtain the same Fourier-type expansion for the resolvent as in the case of the hyperbolic cylinder,
\begin{align}\label{eq:fourier-exp-funnel}
R_{F_\ell,\twist}(s;z,z')^j = \frac{1}{\ell} \sum_{k \in \Z} e^{i \phi(k+\vartheta_j) }
\tilde{v}_{k + \vartheta_j}(s;r,r') e^{-i\phi' (k+\vartheta_j)}\,.
\end{align}

\begin{prop}\label{prop:fourier-funnel}
The resolvent of $\Delta_{F_\ell, \twist}$ admits a meromorphic continuation to $\C$ as an operator
\begin{align*}
R_{F_\ell,\twist}(s) \colon  L^2_\cpt(F_\ell,\bundle) \to L^2_\loc(F_\ell,\bundle)\,.
\end{align*}
Let $\psi \in \CI_b(F_\ell \times_0 F_\ell)$ with $\supp \psi \cap \diagon_0 = \emptyset$ and $s \in \C$ not a pole of the resolvent. We have
\begin{align*}
\psi \beta^* R_{F_\ell,\twist}(s; \cdot, \cdot ) &\in (\rho_0 \rho_0')^{s} \CI(F_\ell \times_0 F_\ell, \bundle \boxtimes_0 \bundle)
\\
&\hphantom{\in \rho_0 \rho_0')^{s}} + \beta^*\left( (\rho_f \rho_f')^{s}\CI(\overline{F_\ell} \times \overline{F_\ell}, \bundle \boxtimes \bundle) \right)\,.
\end{align*}
Moreover, the multiset of resonances $\ResSet_{F_\ell,\twist}$ is given by
\begin{align*}
\ResSet_{F_\ell,\twist} = \bigcup_{\lambda \in \EV(\twist(h_\ell))} \bigcup_{p \in \{\pm 1\}}
\left( -(1+2\N_0) + p \ell^{-1} \left( \log\lambda + 2\pi i \Z\right)\right)\,,
\end{align*}
where $\EV(\twist(h_\ell))$ is the multiset of eigenvalues of $\twist(h_\ell)$.
\end{prop}

\begin{proof}
The structure of the resolvent follows from Proposition~\ref{prop:r_hypcyl_merom} and \eqref{eq:images}. The resonances can be read off from the Fourier coefficients.
\end{proof}

\begin{remark}\label{rem:upper-bound-funnel}
As in the case of the hyperbolic cylinder this gives upper and lower bounds 
on the number of resonances:
\[
N_{F_\ell,\twist}(r) \asymp r^2 \qquad\text{as $r \to \infty$}\,.
\]
\end{remark}

We would like to define the Poisson operator and the scattering matrix for the model funnel. For this we note that the boundary at infinity of the model funnel is given by
\begin{align*}
\pa_\infty F_\ell &= \R / 2\pi \Z\,,
\\
g_\pa & \coloneqq \rho^2 g_{F_\ell}|_{\pa_\infty F_\ell} =
\omega^{-2} d\phi^2\,.
\end{align*}
Hence, the induced Riemannian volume is 
\[
d\mu_{\pa X} = \frac{d\phi}{\omega}\,.
\]
Proposition~\ref{prop:fourier-funnel} implies that the function
\begin{align}\label{eq:poisson-model-funnel}
E_{F_\ell,\twist}(s; r, \phi, \phi') \coloneqq \lim_{r' \to \infty} 
(\rho_f'(r'))^{-s} R_{F_\ell,\twist}(s; r, \phi, r', \phi')
\end{align}
is well-defined. This is the Schwartz kernel corresponding to the operator 
\begin{align*}
E_{F_\ell,\twist}(s) &\colon \CI(\pa_\infty F_\ell, \bundle|_{\pa_\infty F_\ell}) \to \CI(F_\ell, \bundle)\,,
\\
(E_{F_\ell,\twist}(s) f)(r,\phi) &\coloneqq \frac{1}{\omega} \int_0^{2\pi} 
E_{F_\ell,\twist}(s;r,\phi,\phi') f(\phi') \, d\phi'\,.
\end{align*}
Since $\rho_f'(r') = \cosh(r')^{-1}$, we can read off the Fourier expansion of 
the integral kernel of $E_{F_\ell,\twist}(s)$ from the Fourier expansion of the 
integral kernel of $R_{F_\ell,\twist}(s)$. As above, let $(\psi_j)_{j=1}^n$ be an eigenbasis of $\twist(h_\ell)$ with eigenvalues $\lambda_j = e^{2\pi i\vartheta_j}$. With the help of \eqref{eq:fourier-exp-funnel} we obtain that
\begin{align*}
E_{F_\ell,\twist}(s;r,\phi,\phi') \psi_j
= \sum_{k\in\Z} E_{F_\ell,\twist}(s;r)_k^j \psi_j e^{ik(\phi-\phi')}\,,
\end{align*}
where 
\begin{align*}
E_{F_\ell,\twist}(s;r)_k^j \coloneqq
\frac{1}{\ell} \frac{\beta_{k + \vartheta_j}(s) v_{k + \vartheta_j}^0(s;r)}{\gammafunc(s+1/2)}\,.
\end{align*}

The following lemma is a generalization of \cite[Proposition 7.9]{Borthwick_book} for the funnel end.

\begin{lemma}\label{lem:funnel-res-diff}
For $s\in \C$ other than a pole of $R_{F_\ell,\twist}(s)$ or $R_{F_\ell,\twist}(1-s)$, we have that
\begin{align*}
R_{F_\ell,\twist}(s) - R_{F_\ell,\twist}(1-s) = (1-2s) E_{F_\ell,\twist}(s)E^T_{F_\ell,\twist}(1-s)\,.
\end{align*}
Here, 
\[
E^T_{F_\ell,\twist}(s) \colon  L^2(F_\ell,\bundle) \to L^2(\pa_\infty F_\ell, \bundle|_{\pa_\infty F_\ell})
\]
denotes the transpose of the operator \[E_{F_\ell,\twist}(s) \colon  L^2(\pa_\infty F_\ell, \bundle|_{\pa_\infty F_\ell}) \to L^2(F_\ell,\bundle)\,.\]
\end{lemma}

\begin{proof}
We choose any basis $(e_j)_{j=1}^n$ of $V$ and set, for each $1 \le j, k \le n$, 
\begin{align*}
R_{jk}(s;z,w) & \coloneqq \ang{R_{F_\ell,\twist}(s;z,w)e_j,e_k} \,.
\end{align*}
Denote by $R_{F_\ell,\twist}^T(s;z,w)$ the integral kernel of $R_{F_\ell,\twist}(s)^T$ and let
\begin{align*}
R^T_{jk}(s;z,w) \coloneqq \ang{R_{F_\ell,\twist}^T(s;z,w)e_j,e_k}\,.
\end{align*}
It follows that
\begin{align*}
R^T_{jk}(s;z,w) = R_{kj}(s;w,z)\,.
\end{align*}
At least formally, 
\begin{align*}
R_{F_\ell,\twist}(s) & - R_{F_\ell,\twist}(1-s) 
\\
& = R_{F_\ell,\twist}(s) (\Delta_{F_\ell,\twist} - s(1-s)) R_{F_\ell,\twist}(1-s) \\
& \phantom{= R_{F_\ell,\twist}(s)} - (\Delta_{F_\ell,\twist} - s(1-s)) R_{F_\ell,\twist}(s) 
R_{F_\ell,\twist}(1-s) \\
&= R_{F_\ell,\twist}(s) \Delta_{F_\ell,\twist}  R_{F_\ell,\twist}(1-s) - 
\Delta_{F_\ell,\twist}  R_{F_\ell,\twist}(s) R_{F_\ell,\twist}(1-s)\,.
\end{align*}
We can make the above calculation rigorous by considering the Schwartz kernel, 
\begin{align*}
R_{jk}(s;z,w) & - R_{jk}(1-s;z,w) 
\\
& =
\lim_{\eps \to 0} \int_{\rho(z') > \eps} \sum_{m=1}^n \bigg(R_{jm}(s;z,z') \Delta_{F_\ell} R_{mk}(1-s;z',w)
\\
& \qquad - \Delta_{F_\ell}  R_{jm}(s;z,z') R_{mk}(1-s;z',w) \bigg)\, d\mu_{F_\ell}(z') 
\\
& =  \lim_{\eps \to 0} \int_{\rho(z') > \eps} \sum_{m=1}^n \bigg(R_{jm}(s;z,z') \Delta_{F_\ell} R^T_{km}(1-s;w,z')
\\
& \qquad - \Delta_{F_\ell}  R_{jm}(s;z,z') R^T_{km}(1-s;w,z') \bigg)\, d\mu_{F_\ell}(z') \,.
\end{align*}
Here, $\Delta_{F_\ell}$ acts on the $z'$ variable. We apply the Green formula to obtain  
\begin{equation}\label{eq:number}
\begin{aligned}
R_{jk}&(s;z,w)  - R_{jk}(1-s;z,w) 
\\
& = \lim_{\eps \to 0} \int_{\rho(z') = \eps} \sum_{m=1}^n
\bigg(-R_{jm}(s;z,z') \pa_\nu R^T_{km}(1-s;w,z')
\\
& \hphantom{\lim_{\eps \to 0} \int_{\rho(z') = \eps} \sum_{m=1}^n}  + \pa_\nu R_{jm}(s;z,z') R^T_{km}(1-s;w,z') \bigg)\, d\mu_{F_\ell}(z')\,,
\end{aligned}
\end{equation}
where $\nu$ is the outward pointing unit normal vector to the set $\{\rho(z') > \eps\}$. We use the coordinates $z' = (\rho', \phi')$, where $\rho' = \rho(z')$ and $\phi'$ is as in \eqref{eq:cylinder-iso}. Since $g_{F_\ell}(\pa_\rho, \pa_\rho) = \rho^2$, we have $\pa_\nu = - \rho'\pa_{\rho'}$. Inserting
\begin{align*}
R_{F_\ell,\twist}(s;z,z') & =
(\rho')^s E_{F_\ell,\twist}(s;z,\phi') + O( (\rho')^{s+1})
\intertext{and}
\pa_\nu R_{F_\ell,\twist}(s;z,z') & =
-s (\rho')^s E_{F_\ell,\twist}(s;z,\phi') + O( (\rho')^{s+1})
\end{align*}
in \eqref{eq:number}, we obtain
\begin{align*}
& R_{jk}(s;z,w)  - R_{jk}(1-s;z,w)
\\
&\quad = (1-2s) \lim_{\rho' \to 0} \int_0^{2\pi} \bigg( \sum_{m=1}^n E_{jm}(s;z,\phi') E^T_{km}(1-s;w,\phi') 
\\ 
&\hphantom{ \quad = (1-2s) \lim_{\rho' \to 0} \int_0^{2\pi} \bigg( \sum_{m=1}^n E_{jm}(s;z,\phi')} + O(\rho')\bigg)\, \rho' d\mu_{F_\ell}(\rho',\phi')\,.
\end{align*}
Now taking $\rho' = \eps \to 0$ and using that 
\[
\lim_{\rho' \to 0} \rho' d\mu_{F_\ell}(\rho',\phi') = d\mu_{\pa_\infty F_\ell}(\phi') = \frac{1}{\omega}\,d\phi'
\]
yields the assertion.
\end{proof}

Our next step is to define the scattering matrix 
\[
S_{F_\ell,\twist}(s)\colon \CI(\pa_\infty F_\ell, \bundle|_{\pa_\infty F_\ell}) \to \CI(\pa_\infty F_\ell, \bundle|_{\pa_\infty F_\ell})\,.
\]
We use an orthonormal basis of eigenvectors $(\psi_j)_{j=1}^n$ of $\twist(h_\ell)$ with eigenvalues $\lambda_j = e^{2\pi i\vartheta_j}$, where $\vartheta_j \in [0,1)$, to define (see \cite[Lemma 4.2]{GuZw95})
\begin{align*}
S_{F_\ell,\twist}(s):  [\psi_j e^{i k \phi}] \mapsto S_{F_\ell,\twist}(s)_k^j \psi_j e^{i k \phi}\,,
\end{align*}
where 
\begin{align}\label{eq:smatrix-funnel}
S_{F_\ell,\twist}(s)_k^j \coloneqq \frac{\gammafunc\left(\frac{1}{2}-s\right) 
\gammafunc\left(\frac{s + i \omega (k+\vartheta_j) + 1}{2}\right) \gammafunc\left(\frac{s - i \omega (k+\vartheta_j) + 1}{2}\right)}{ \gammafunc\left(s-\frac{1}{2}\right) \gammafunc\left(\frac{2-s + i 
\omega (k+\vartheta_j)}{2}\right) \gammafunc\left(\frac{2-s - i \omega 
(k+\vartheta_j)}{2}\right)}\,.
\end{align}
Using a Kummer connection formula (see \cite[\S 2.9, (33)]{ErdelyiI}) for 
the function $v_{k+\vartheta_j}^0(s;r)$, we obtain the asymptotic expansion
\begin{align}\label{eq:poisson-asymptotics-funnel}
(2s-1) E_{F_\ell,\twist}(s;r) f \sim \sum_{m=0}^\infty \rho_f^{1-s + 2m} 
a_m(s) + \sum_{m=0}^\infty \rho_f^{s +2m} b_m(s)\,,
\end{align}
where $a_0(s) = f$ and $b_0(s) = S_{F_\ell,\twist}(s) f$ (see~\cite[Proposition~5.6]{Borthwick_book}). Hence for $\Re s < 1/2$, we have that
\begin{align*}
S_{F_\ell,\twist}(s) = (2s-1) (\rho \rho')^{-s} 
R_{F_\ell,\twist}(s)|_{\pa_\infty F_\ell \times \pa_\infty F_\ell}\,.
\end{align*}
In the same way as in~\cite[pp.~92--93]{Borthwick_book} we obtain that
\begin{align*}
E_{F_\ell,\twist}(1-s) S_{F_\ell,\twist}(s) = - E_{F_\ell,\twist}(s)\,.
\end{align*}
Using this and Lemma~\ref{lem:funnel-res-diff}, we arrive at the identity
\begin{align}\label{eq:resolvent-scattering-matrix}
R_{F_\ell,\twist}(s) - R_{F_\ell,\twist}(1-s) = (2s-1) 
E_{F_\ell,\twist}(1-s) S_{F_\ell,\twist}(s) E_{F_\ell,\twist}(1-s)^T\,.
\end{align}

\subsection{Parabolic Cylinder}

We consider the parabolic cylinder $C_\infty$ as the quotient
\begin{align*}
C_\infty \coloneqq \ang{T} \bs \h \cong \R / \Z \times (0,\infty)\,,
\end{align*}
where $T.z = z+1$. Let $V$ be a finite-dimensional real vector space and $B \in \GL(V)$. We define the representation $\twist$ by
\begin{align*}
\twist(T) \coloneqq B\,.
\end{align*}
We denote by $R_{C_\infty,\twist}(s)$ the resolvent of the Laplace operator on 
the model parabolic cylinder. For meromorphic continuation of the $R_{C_\infty,\twist}(s)$ it suffices to assume that~$\twist$ has \emph{non-expanding cusp monodromy}, that is all eigenvalues of $B$ have modulus $1$. As for the hyperbolic cylinder, weakening the assumption of $\twist$ being unitary to non-expanding cusp monodromy does not complicate the proof. As before, the resolvent is given by
\begin{equation}\label{eq:resolvent_parcyl}
R_{C_\infty,\twist}(s; z,z') = \id_V R_\h(s;z,z') + \sum_{k\in \ZZ\setminus\{0\}} B^k R_\h(s; z,z'+k)\,,
\end{equation}
where $z, z' \in \funddom = \{z \in \h \setmid \Re z \in [0,1)\}$. Before we can prove the meromorphic continuation of $R_{C_\infty,\twist}(s)$, we have to extend a statement in \cite[Lemma~1.4]{Guillope}. For $\xi \in \C$ with $\abs{\xi} = 1$, $a\in \R$, $b \in (0,\infty)$ and $\Re s > 1/2$, define the function 
\[ 
S_\xi(s;a,b) \coloneqq \sum_{k\in \Z} \xi^k (\abs{k + a}^2 + b^2)^{-s}\,.
\]

\begin{lemma}\label{lemma:holomorphicity_of_Sxisab}
For every $m \in \N_0$, the function $\pa_\xi^m S_\xi(s;a,b)$ admits a 
meromorphic continuation to $s \in \C$ with possible poles at 
\begin{equation}\label{lemma:holomorphicity_of_Sxisabploes}
s \in m + 1/2 - \N_0
\end{equation} 
if $\xi = 1$. If $\xi \not = 1$, then $\pa_\xi^m S_\xi(s;a,b)$ is holomorphic.
\end{lemma}

\begin{proof}
We write $\xi = e^{2\pi i \lambda}$ with $\lambda \in [0,1)$. For $\Re s > 1/2$, we use the definition of the Gamma-function and the Poisson summation formula to obtain 
\begin{align*}
\gammafunc(s) S_\xi(s;a,b) &= \sum_{k \in \Z} \int_0^\infty \xi^k e^{-(\abs{a+k}^2 + b^2)t} t^s \frac{dt}{t} 
\\
&= \sum_{k \in \Z} \int_{-\infty}^{\infty} e^{-2 \pi i k y} \left( \int_0^\infty \xi^y e^{-(\abs{a+y}^2 + b^2)t} t^s \frac{dt}{t}  \right) dy 
\\
&= \sqrt{\pi } \sum_{k \in \Z} \int_0^\infty e^{-tb^2} t^{s-1/2} e^{-\frac{\pi^2  (\lambda +k)^ 2}{t} + 2 i a \pi (\lambda +k) } \frac{dt}{t} 
\\
&= \sqrt{\pi} b^{1-2s} \int_0^\infty e^{-u} u^{s-1/2} \left( \sum_{k \in \Z} e^{-\frac{\pi^2 b^2 (k + \lambda)^2}{u} + 2\pi i a (k + \lambda)}\right) \frac{du}u\,.
\end{align*}
For $u > 0$, we write
\begin{align*}
f(u) &\coloneqq \sum_{k \in \Z \setminus \{0\}} e^{-\pi^2 b^2 (k + \lambda)^2 u^{-1} + 2\pi i a (k + \lambda)} 
\\
&\ = \sum_{k=1}^\infty e^{-\pi^2 b^2 (k + \lambda)^2 u^{-1} + 2\pi i a (k + \lambda)} + \sum_{k=1}^\infty e^{-\pi^2 b^2 (k - \lambda)^2 u^{-1} - 2\pi i a (k - \lambda)}\,.
\end{align*}
Since $x \mapsto e^{-x^2}$ is monotonically decreasing, we may use the integral convergence test to compare the two series to the integrals
\begin{align*}
\int_1^\infty e^{-\pi^2 b^2 (x \pm \lambda)^2 u^{-1}} \,dx &= \sqrt{\frac{u}{\pi b^2}} \erfc\left( \frac{\pi b (1 \pm \lambda)}{\sqrt{u}} \right)\,,
\end{align*}
where $\erfc$ denotes the complementary error function defined as 
\[
\erfc(x) \coloneqq 1 - \erf(x)\,.
\]
It is well-known that $\erfc(x) = O(x^{-\infty})$ as
$x \to \infty$ (see \cite[8.254]{Gradshteyn}). Since $1 \pm \lambda \not = 0$, we obtain that $f(u) = O(u^{-\infty})$ as $u \to 0$. For $u > 0$ sufficiently large, we have that $f(u) = O(1)$. 

The estimate allows us to conclude that the integral and the series converge absolutely, and hence we can differentiate under the integral and sum with respect to $\xi$. The only term that might prohibit the convergence of the sum and the integral comes from $e^{- \pi^2 b^2 (k+\lambda)^2 u^{-1}}$ with respect to $\lambda$. For differentiating $m$-times this yields the terms $u^{-m}$ and $k^m$. The convergence of the integral is not affected due to the presence of the exponential factor $e^{-u}$ and the sum is still converging since for large $k$ we have $k + \lambda \not = 0$. Moreover, we see that if $\xi \not = 1$ we can take $\lambda \in (0,1)$ and therefore the sum and the integral converge for arbitrary $s \in \C$. For $k + \lambda = 0$, we obtain terms of form $\gammafunc(s-m-1/2)$, which are meromorphic and have poles for $s$ as in \eqref{lemma:holomorphicity_of_Sxisabploes}.
\end{proof}

We denote by $d_B$ is the length of the longest Jordan chain of $\twist(T) = B$.

\begin{prop}\label{prop:resolv-cont-cusp}
The resolvent $R_{C_\infty,\twist}(s)$ is well-defined for $\Re s > d_B/2$
and admits a meromorphic continuation to $\C$ with possible poles at
\begin{align*}
\frac{d_B - \N_0}{2}\,.
\end{align*}
\end{prop}

\begin{proof}
Formula~\eqref{eq:sigma} implies that
\begin{equation}\label{eq:sigma_parab}
\sigma(z,z'+k) \asymp k^2 \qquad\text{as $k\to\pm\infty$}
\end{equation}
with implied constants depending continuously on $z$ and $z'$. Hence, taking $f(n) = \log(n)$, we can use Proposition~\ref{prop:Rhkernel_conv} with $a = 2$, $b = d_B - 1$, and $c = 0$ to obtain that $R_{C_\infty,\twist}$ is well-defined for $\Re s > d_B /2$. For the meromorphic continuation, we formally define the function 
\begin{align}\label{eq:Hinfty}
H_\infty(s;z,z') \coloneqq \sum_{k\in\Z} B^k \sigma(z,z'+k)^{-s}\,.
\end{align}
For $\Re s > 1$, $H_\infty(s;z,z')$ is well-defined by \eqref{eq:sigma_parab}.
To use the second part of Proposition~\ref{prop:Rhkernel_conv}, we have to 
prove that $H_\infty(s;z,z')$ admits a meromorphic continuation in~$s$ to all of~$\C$. We note that
\begin{align*}
\sigma(z,z'+k)^{-s} = (4yy')^s \left( (x - x' - k)^2 + (y + y')^2 
\right)^{-s}\,.
\end{align*}
Conjugating $B$ to the Jordan normal form, $B = U^{-1} \tilde{B} U$, where $U$ 
is unitary, we have 
\begin{align}\label{eq:Hinfty_with_paranth}
H_\infty = (4yy')^s U^{-1} \left(\sum_{k\in \Z} \tilde{B}^k \left( (x - x' - 
k)^2 + (y + y')^2 \right)^{-s} \right) U\,.
\end{align}
For a Jordan block $J_\lambda$ of length $d_\lambda$, we have
\begin{align*}
J_\lambda^k = 
\begin{pmatrix}
\lambda^k & \pa_\lambda \lambda^k & \dotsc & 
\frac{1}{(d_\lambda-1)!}\pa_\lambda^{d_\lambda-1} \lambda^k
\\
& \lambda^k &  & \frac{1}{(d_\lambda-2)!}\pa_\lambda^{d_\lambda-2} 
\lambda^k 
\\
& & \ddots & \vdots
\\
& & & \lambda^k
\end{pmatrix}\,.
\end{align*}
Therefore, the sum in parentheses in \eqref{eq:Hinfty_with_paranth} can be written as a matrix with entries 
\[
\frac1{k!}\pa_\lambda^k S_\lambda(s;-x+x',y+y')\,,
\]
where $k = 0, \dotsc, d_\lambda-1$. Hence, the meromorphic continuation of 
\eqref{eq:Hinfty} follows from Lemma \ref{lemma:holomorphicity_of_Sxisab}.
\end{proof}

\subsection{Fourier Analysis of the Parabolic Cylinder}\label{sec:parab_cyl}
Recall that $C_\infty = \ang{T} \bs \h$, where $T.z = z + 1$, and let $\twist\colon  \ang{T} \to \Unit(V)$ be a finite-dimensional unitary representation. The goal of this section is to calculate the structure of the resolvent on the parabolic cylinder similar to the case of the hyperbolic cylinder in Proposition~\ref{prop:fourier-cylinder}. The statement is slightly different since the Schwartz kernel of the resolvent is not smooth on the product space, due to the very restrictive nature of smooth functions in cusp ends.

\begin{prop}\label{prop:fourier-cusp}
Let $\psi \in \CI(\overline{C_\infty})$ be supported away from $\{y = 0\}$.
The meromorphically continued resolvent $R_{C_\infty,\twist}(s)$ defines a continuous map
\begin{align*}
\psi R_{C_\infty,\twist}(s) \colon  \CcI(C_\infty,\bundle) \to 
\rho_c^{s-1} \CI(\overline{C_\infty}, \bundle)
\end{align*}
provided that $s \not = 1/2$. Here, $\psi$ denotes the multiplication 
operator with the function $\psi$ on $\CI(\overline{C_\infty},\bundle)$.
The only pole of $R_{C_\infty,\twist}(s)$ is at the point $s=1/2$.
Its multiplicity is equal to the (algebraic) multiplicity of the eigenvalue~$1$ of $\twist(T)$.
\end{prop}

We denote by $R_{C_\infty,\twist}(s; z, z')$ the integral kernel of $R_{C_\infty,\twist}(s)$. We choose an eigenbasis $(\psi_j)_{j=1}^n$ of $\twist(T)$ with eigenvalues $\lambda_j = e^{2\pi i \vartheta_j}$. We may take $\vartheta_j \in [0,1)$, see also Remark~\ref{rem:vartheta_change} below. Set
\begin{align}\label{RCinftychiszz}
R_{C_\infty,\twist}(s; z, z')^j \coloneqq \ang{R_{C_\infty,\twist}(s; z, z') \psi_j, \psi_j}\,.
\end{align}
The resolvent kernel solves the equation
\begin{equation*}
\left(-y^2(\pa_y^2 + \pa_x^2) - s(1-s)\right) R_{C_\infty,\twist}(s;x,y,x',y')
= y^2 \id_V \delta(x-x') \delta(y-y')\,,
\end{equation*}
since $-y^2(\pa_y^2 + \pa_x^2)$ is the Laplacian of the half-plane
model~$\h$ for the hyperbolic plane. We set
\begin{align}\label{eq:ukappa_parab}
u_\kappa(s;y,y') \coloneqq \begin{cases}
\sqrt{yy'} I_{s-1/2}(\abs{\kappa}y) K_{s-1/2}(\abs{\kappa}y') & \text{for $y\leq 
y'$}\,,
\\
\sqrt{yy'} K_{s-1/2}(\abs{\kappa}y) I_{s-1/2}(\abs{\kappa}y') & \text{for $y > y'$}
\end{cases}
\end{align}
for $\kappa \in \R \setminus\{0\}$ and
\begin{align}\label{eq:uzero_parab}
u_0(s;y,y') \coloneqq \frac{1}{2s-1} 
\begin{cases}
y^s (y')^{1-s} & \text{for $y \leq y'$}\,,
\\
y^{1-s} (y')^{s} & \text{for $y > y'$}\,.
\end{cases}
\end{align}

\begin{remark}
Let $s \not = 1/2$ and $y,y'>0$. From the asymptotic expansion of the modified Bessel functions~$I_{s-1/2}(x)$ and~$K_{s-1/2}(x)$ as $x \to 0+$ (see \cite[p.77(2), p.78(6-7)]{Watson44}), we obtain
\begin{align*}
u_\kappa(s;y,y') \to u_0(s;y,y')\,, \quad \kappa \to 0\,.
\end{align*}
\end{remark}

\begin{lemma}\label{lem:fourier-cusp}
The integral kernel of the resolvent for the parabolic cylinder $R_{C_\infty,\twist}(s)$ admits a Fourier decomposition.
The Fourier decomposition of the non-vanishing matrix coefficients are given by, for any $j\in\{1,\ldots, n\}$,
\begin{equation}\label{eq:fourier-exp-par}
R_{C_\infty,\twist}(s; z, z')^j = \sum_{k \in \Z} e^{2 \pi i (k +  
\vartheta_j) x} b_k^j(s;y, y') e^{-2\pi i (k + \vartheta_j) x'}\,,
\end{equation}
where
\begin{align*}
b_k^j(s;y,y') = u_{2\pi(k + \vartheta_j)}(s;y,y')\,.
\end{align*}
\end{lemma}

\begin{proof}
It is straightforward to see that the coefficients $b_k^j$ have to satisfy 
the equation
\begin{align*}
\left(-y^2 \pa_y^2 + y^2(2\pi (k+\vartheta_j))^2 - s(1-s)\right) 
b_k^j(y,y') &= y^2 \delta(y-y')\,.
\end{align*}
From \cite[p.~96]{Borthwick_book} we see that $b_k^j$ is given by
\begin{align*}
b_k^j(s,y,y') = u_{2\pi(k+\vartheta_j)}(s,y,y')\,.
\end{align*}
\end{proof}

\begin{remark}\label{rem:vartheta_change}
We note that replacing $\vartheta_j \mapsto \vartheta_j + 1$ is absorbed in 
the Fourier expansion by shifting the index.
\end{remark}

\begin{proof}[Proof of Proposition~\ref{prop:fourier-cusp}]
For the mapping properties of $\psi R_{C_\infty,\twist}(s)$, we note that the boundary defining function is given by $\rho_c(y) = 1/y$. For $\rho_c' =  1/y'$ supported in a compact set away from $0$, the zero Fourier mode is given by \[u_0(s;y,y') = \frac{\rho_c^{s-1} (\rho_c')^{-s}}{2s-1}\] for $\rho_c > 0$ small enough. Moreover, the non-zero Fourier modes are rapidly decaying as $\rho_c \to 0$. Hence for $u \in \CcI(C_\infty, \bundle)$,
\begin{align*}
\psi R_{C_\infty,\twist}(s) u \in \frac{\rho_c^{s-1}}{2s-1} \cdot 
\CI(\overline{C_\infty}, \bundle)\,.
\end{align*}
For the resonances, we observe that $u_\kappa(\cdot, y,y')$ is entire for $\kappa \not = 0$. Further, we have that $\kappa = 2\pi (\vartheta_j + k) = 0$ if and only if $k = \vartheta_j = 0$ by using that $\vartheta_j \in [0,1)$.
\end{proof}

\subsection{Estimating the Resolvent}
Let $X \in \{F_\ell, C_\ell, C_\infty\}$, where $\ell \in (0,\infty)$.
Set $g = h_\ell$ if $X \in \{F_\ell, C_\ell\}$ and $g = T$ if $X = C_\infty$
and let $\twist\colon \group = \ang{g} \to \Unit(V)$ be a unitary finite-dimensional representation. We choose $\psi_0,\psi_1 \in \CI(\overline{X})$ such that $\psi_0,\psi_1$ extend smoothly up to the boundary $\{r = 0\}$ (if $X = F_\ell$)
and are supported away from the boundary at infinity. Moreover we choose $\tilde\psi_0 \in \CI(\overline{X})$ with $\supp \nabla_X \tilde\psi_0 \subset \{\psi_0 = 1\}$ and $\Omega \subset \C$.

\begin{prop}\label{prop:bound-resolvent-determinant}
If there exists $C_1 = C_1(\Omega)$ and $\tau \in \R$ such that 
\begin{align*}
\norm{\psi_0 R_{X,\twist}(s) \psi_1} \le C_1 \ang{s}^\tau\,,\quad s \in \Omega\,,
\end{align*}
then there exist $C_2 = C_2(\Omega) > 0$ such that
\begin{align*}
\det\left( \I + c \abs{K_{\tilde\psi_0,\psi_1}(s)} \right) \lesssim e^{C_2\, \ang{s}^{(5+\tau)/2}}\,,\quad s \in \Omega\,,
\end{align*}
where $K_{\tilde\psi_0,\psi_1}(s) \coloneqq [\LapTwist, \tilde\psi_0] R_{X,\twist}(s) \psi_1$.
\end{prop}

In order to prove this proposition, we need the following two lemmas.

\begin{lemma}\label{lem:bound-cylinder-resolvent-sobolev} 
Let $\Omega \subset \C$ open and $\tau \in \R$. Suppose that there exists $C_\Omega > 0$ such that for all $s \in \Omega$
\begin{align*}
\norm*{ \psi_0 \ResTwist(s) \psi_1 } \leq C_\Omega \ang{s}^\tau\,.
\end{align*}
Then for all $m \in \N$ there exists $C_{\Omega,m} > 0$ such that
\begin{align*}
\norm*{\LapTwist^m \psi_0 \ResTwist(s) \psi_1} \leq C_{\Omega,m} \ang{s}^{\tau+2m}
\end{align*}
and
\begin{align*}
\norm*{\LapTwist^m [\LapTwist, \tilde\psi_0] \ResTwist(s) \psi_1} \leq C_{\Omega,m} \ang{s}^{\tau+2m+1}\,.
\end{align*}
\end{lemma}

\begin{proof}
We start with proving the lemma for $m=0$. Let $u \in L^2(X,\bundle)$, we have to show that
\begin{align*}
\norm*{[\LapTwist, \tilde\psi_0] \ResTwist(s) \psi_1 u}_{L^2} \leq C_{\Omega,m} \ang{s}^{\tau+2m+1} \norm{u}_{L^2}\,.
\end{align*}
We choose an orthonormal basis $\{e_j\}_{j=1}^{\dim(V)}$ of $V$ that diagonalizes $\twist(g) \in \Unit(V)$. Denote by $v \in L^2_\loc(\h,V)$ the lift of $u$ to $\h$. We write
\begin{align*}
v = \sum_{j=1}^{\dim(V)} v_j e_j\,.
\end{align*}
From \eqref{eq:resolvent-generic_cylinder} (and \eqref{eq:images} in the case $X = F_\ell$), we see that $\ResTwist(s)$ acts as
\begin{align*}
\ResTwist(s) v = \sum_{j=1}^{\dim(V)} e_j R_j(s) v_j \,,
\end{align*}
where $R_j(s) = \ang{ R_{X,\twist}(s) e_j, e_j}$ and $\ang{\cdot,\cdot}$ is the inner product on $V$. We have
\begin{align*}
[\Delta_{X,\twist}, \tilde\psi_0] R_{X,\twist}(s) \psi_1 e_j =
e_j (\Delta_{X} \tilde\psi_0) R_j(s) \psi_1 + e_j \cdot g_\h(\nabla_\h \tilde\psi_0, \nabla_\h R_j(s) \psi_1)\,,
\end{align*}
where $g_\h$ is the hyperbolic metric on $\h$ and $\nabla_\h$ denotes the corresponding gradient.

The first summand is bounded by hypothesis. For the second summand, we set $f_j \coloneqq R_j(s) v_j$. We can estimate the $L^2$-norm by
\begin{align*}
\norm*{ \ang{\nabla_X \tilde\psi_0, \nabla_{X,\twist} R_{X,\twist}(s) \psi_1 u} }_{L^2(X,\bundle)}^2
&\leq \sum_j \int_\funddom \abs*{ g_\h(\nabla_\h \tilde\psi_0, \nabla_\h f_j(s)) }^2 \,d\mu_\h 
\\
&\leq C \sum_j \int_\funddom \abs*{ \psi_0 \nabla_\h f_j(s)}^2 \,d\mu_\h \,,
\end{align*}
where $\funddom$ is a fundamental domain of $X$. Since $\twist(g)$ is unitary, we can integrate by parts to obtain
\begin{align*}
\int \abs*{ \psi_0 \nabla_\h f_j(s)}^2 \,d\mu_\h &= - \int \psi_0^2 f_j(s) \Delta_\h f_j(s)\,d\mu_\h 
\\
&\hphantom{=  - \int} - 2 \int g_\h( \nabla_\h \psi_0, \nabla_\h f_j(s) ) \psi_0 f_j(s) \,d\mu_\h\,.
\end{align*}
We note that $f_j(s) = R_j(s) u_j$ solves the equation $(\Delta_X - s(1-s))f_j(s) = u_j$ and, since the supports of $\psi_0$ and $\psi_1$
are disjoint, we have $f_j \psi_0 = 0$. Therefore the first term is bounded by
\begin{align*}
\abs*{ \int \psi_0^2 f_j(s) \Delta_\h f_j(s)\,d\mu_\h } \leq \ang{s}^2 \, \norm{u}^2_{L^2(X,\bundle)}\,.
\end{align*}
The second term can be estimated by
\begin{align*}
\left| \int \ang{\nabla_\h \psi_0, \nabla_X f_j(s)}\right.&\left.\!\!\psi_0 f_j(s) \,d\mu_\h \vphantom{\int}\right|
\\
&\leq \int \abs{ \nabla_\h \psi_0 } \abs{ \nabla_\h f_j(s) } \abs{\psi_0} \abs{f_j(s)} \,d\mu_\h 
\\
&\leq \frac{1}{2\eps} \int \abs{ f_j(s) \nabla_\h \psi_0 }^2 \,d\mu_\h + \frac{\eps}{2} \int \abs{ \psi_0 \nabla_X f_j(s) }^2 \,d\mu_\h
\end{align*}
for any $\eps > 0$, using Young's inequality. Taking $\eps = 2$ we conclude that
\begin{align*}
\int \abs*{ \psi_0 \nabla_\h f_j(s)}^2 \,d\mu_\h \leq C \int \abs{ f_j(s) \nabla_\h \psi_0 }^2 \,d\mu_\h\,.
\end{align*}
Combining the two estimates proves the case $m=0$. Now we can use the same induction argument as in~\cite[Lemma~3.2]{GuZw95}.
\end{proof}

As in~\cite[p. 189]{Borthwick_book}, it suffices to bound the $L^2$-norm of the cutoff resolvent to obtain bounds on the singular values.

\begin{lemma}\label{lem:L2-to-trace}
We let $Z\subset X$  be an open bounded subset of $X$ with smooth boundary. Then for any bounded operator 
\[
B \colon  L^2_\cpt(Z,\bundle|_Z) \to L^2_\cpt(Z,\bundle|_Z)
\]
there exists $C = C(Z, \twist)$ such that for any $k, m \in \N$ we have
\begin{align*}
\mu_k(B) \leq (C k)^{-m} \, \norm*{\Delta_{X,\twist}^m B}\,.
\end{align*}
\end{lemma}

\begin{proof}
Denote by $\Delta_{Z,\twist}$ the Laplace operator  on $\bundle|_{Z}$ with Dirichlet boundary conditions. Since $B$ acts on functions that are compactly supported inside~$Z$, we have that for any $m \in \N$, $\Delta^m_{Z,\twist} B = \Delta_{X,\twist}^m B$. The operator $\Delta_{Z,\twist}$ is invertible, that allows us to write
\begin{align*}
B = \Delta_{Z,\twist}^{-m} \Delta_{Z,\twist}^{m} B\,.
\end{align*}
Since we can diagonalize the representation and therefore the Laplacian, we have a Weyl law for the operator $\Delta_{Z,\twist}$. Hence, we obtain the singular value bound $\mu_k(\Delta_{Z,\twist}) \leq C k$ for some $C > 0$ depending on $Z$ and $\twist$. The claim follows from the estimate
\begin{align*}
\mu_k(B) \leq \mu_k(\Delta_{Z,\twist}^{-m}) \, \norm*{\Delta_{Z,\twist}^{m} B}\,.
\end{align*}
\end{proof}

We are now prepared to prove Proposition~\ref{prop:bound-resolvent-determinant}.

\begin{proof}[Proof of Proposition~\ref{prop:bound-resolvent-determinant}]
Let $m \in \N$. Lemmas~\ref{lem:L2-to-trace} and \ref{lem:bound-cylinder-resolvent-sobolev} imply that for all $k \in \N$,
\begin{align*}
\mu_k(K_{\tilde\psi_0,\psi_1}(s)) & \le (C k)^{-m} \| \Delta_{Z,\twist}^{m} K_{\tilde\psi_0,\psi_1}(s) \|
\\
& \le (C k)^{-m} C_{\Omega, m} \ang{s}^{2m+\tau + 1}\,.
\end{align*}
Substituting $m=2$, we have by \cite[p. 375]{GuZw95} that 
\begin{align*}
\abs*{\det( \I + c\abs{K_{\tilde\psi_0,\psi_1}(s)} ) }
& \leq \prod_{k=1}^{\infty} \left(1 + c \mu_k\left(K_{\tilde\psi_0, \psi_2}(s)\right)\right)
\\
& \leq \prod_{k=1}^\infty \left(1 + \frac{C_3^2 \ang{s}^{5+\tau}}{k^2}\right)\,,
\end{align*}
where $C_2 \coloneqq \sqrt{ C^{-2} C_{\Omega, 2} c}$. We recall from~\cite[1.431(2)]{Gradshteyn} that, for all $z \in \C$, the hyperbolic sine has the infinite product representation
\begin{align*}
\sinh(z) = z \prod_{k=1}^\infty \left( 1 + \frac{z^2}{(\pi k)^2} \right)\,.
\end{align*}
This implies that
\[
\sinh(C_2\ang{s}^{(5+\tau)/2}) = C_2\ang{s}^{(5+\tau)/2} \prod_{k=1}^\infty \left( 1 + \frac{C_2^2 \ang{s}^{5+\tau}}{(\pi k)^2} \right)\,.
\]
Hence, we obtain the estimate
\begin{align*}
\abs*{ \det( \I + c\abs{K_{\tilde\psi_0,\psi_1}(s)} ) }
&\leq  \abs*{  \frac{\sinh(C_2\ang{s}^{(5+\tau)/2})}{C_2 \ang{s}^{(5+\tau)/2}} } \\
&\leq C_3 e^{C_2\ang{s}^{(5+\tau)/2}}
\end{align*}
for some $C_3 = C_3(\Omega) > 0$.
\end{proof}

\section{Meromorphic Continuation of the Twisted Resolvent}\label{sec:resolvent}
The goal of this section is to prove Theorem~\ref{thm:resolv_merom}, the meromorphic continuation of the resolvent $\ResTwist(s)$.

\begin{prop}\label{prop:resolvent-estimate}
The resolvent
\[
\ResTwist(s) = (\LapTwist - s(1-s))^{-1}\colon L^2(X, \bundle) \to L^2(X,\bundle)
\]
is well-defined and holomorphic in $\Re s > 1/2$ and $s \not \in [1/2, 1]$. We have the estimate
\begin{align}\label{est:resolvent-twist}
\norm{\ResTwist(s)}_{L^2 \to L^2} \leq \frac{2}{2\Re s - 1}
\end{align}
for $\Re s > 1/2$ and $\abs{\Im s} > 1/4$ as well as for $\Re s > 2$ (and no restriction on~$\Ima s$). Moreover, if $\LapTwist$ has no spectrum in the interval $[0,1/4)$ then for any $\eps > 0$ there exists $C > 0$ such that
\begin{align}\label{est:resolvent-twist2}
\norm{\ResTwist(s)}_{L^2 \to L^2} \leq C \ang{s}^{-1}
\end{align}
for $\Re s > 1/2 + \eps$.
\end{prop}

\begin{proof}
Since the Laplacian is self-adjoint and non-negative, we obtain for $\Re s > 1/2$ and $s \not \in [1/2,1]$ the resolvent bound
\begin{align*}
\norm{\ResTwist(s)}_{L^2 \to L^2} \leq \frac{1}{d_\C(s(1-s), \R_+)}\,.
\end{align*}
We write $s = a + b i$ with $a,b \in \R$. Then
\begin{align*}
s(1-s) = -a(a-1) + b^2 + (1-2a)b i\,.
\end{align*}
We consider two cases.
\paragraph{First case: $b^2 > 1/4$}
We have that
\begin{align*}
d_\C(s(1-s)), \R_+) &\geq \abs{1 - 2a} \abs{b} \geq \frac{1}{2}\abs{1 - 2\Re(s)}\,.
\end{align*}
\paragraph{Second case: $b^2 \geq 1/4$}
For any $\lambda \in \R$, we have the trivial estimate
\begin{align*}
d_\C(s(1-s)), \R_+) &\geq \abs{ a(1-a) + b^2 - \lambda^2}\,.
\end{align*}
For $a > 2$, we have that $a(1-a) < -1/4$ and therefore we obtain
\begin{align*}
\abs{ a(1-a) + b^2 - \lambda^2} &= a(a-1) - b^2 + \lambda^2 
\\
&\geq a(a-1) - \frac{1}{4} 
\\
&\geq \frac{1}{2}\abs{1 - 2\Re(s)}\,.
\end{align*}
This proves the first resolvent estimate.

For the second estimate we consider the two cases $\abs{b} > \abs{2 a - 1}/4$ and $\abs{b} \leq \abs{2 a - 1}/4$. In the first case, we note that
\begin{align*}
\abs{s} \leq \abs{a} + \abs{b} \leq 3 \abs{b} + \frac{1}{2}
\end{align*}
and we have that
\begin{align*}
d_\C(s(1-s)), [1/4,\infty)) &\geq \abs{2 a - 1} \abs{b} \geq \frac{\eps}{3} (2 \abs{s} - 1)\,.
\end{align*}
In the second case, applying the same arguments we obtain that
\begin{align*}
\abs{s} \leq \frac{3}{4} \abs{2a - 1} + \frac{1}{2}
\end{align*}
and for any $\lambda \in \R$ we have
\begin{align*}
d_\C(s(1-s)), [1/4,\infty)) &\geq \abs*{\frac{1}{4} - \frac{1}{4} \abs{2a - 1}^2 + b^2 - \left(\lambda^2 + \frac{1}{4}\right)} 
\\
&= \frac{1}{4} \abs{2 a - 1}^2 - b^2 + \lambda^2 
\\
&\geq \frac{3}{16} \abs{2 a - 1}^2\,.
\end{align*}
Consequently, we have that
\begin{align*}
d_\C(s(1-s), [1/4, \infty)) \geq C^{-1} \ang{s}
\end{align*}
for some $C > 0$ depending on $\eps > 0$.
\end{proof}

To prove Theorem~\ref{thm:resolv_merom}, we will follow the arguments of \cite[Theorem 6.8]{Borthwick_book} and use the explicit structure of the model resolvents from Section~\ref{sec:model}. As before, we have the decomposition $X = K \sqcup X_f \sqcup X_c$, where $K$ is compact and $X_f$ and~$X_c$ are a finite collection of funnels and cusps, respectively. For $\bullet \in \{f,c\}$ and $r \in [0,\infty)$, we choose a cutoff function $\eta_{\bullet,r} \in \CI(X)$ such that
\begin{align}\label{eq:etaend}
\eta_{\bullet,r}(x) = 
\begin{cases} 
1 & \text{if $d(X \setminus X_\bullet,x) < r$}\,,
\\ 
0 & \text{if $d(X\setminus X_\bullet,x) > r+1/2$}\,.
\end{cases}
\end{align}
We remark that $\eta_{f,r}$ vanishes ``far'' in funnels and $\eta_{c,r}$ vanishes ``far'' in cusps, i.e., in the parametrization of funnels from Section~\ref{sec:funnel_ends},  $\eta_{f,r}$ vanishes on $(r_0,\infty)\times \R/2\pi\Z$ for $r_0$ sufficiently large, and analogously for cusps. By Proposition~\ref{prop:model-ends}, for each $X_{f,j}$ ($X_{c,j}$) there exists a hyperbolic (parabolic) element $\gamma_j \in \group$ such that 
\begin{align*}
\bundle|_{X_{\bullet,j}} = (\h \times_{\twist_j} V)|_{X_{\bullet,j}}\,,
\end{align*}
where $\twist_j \coloneqq \twist|_{\ang{\gamma_j}}\colon \ang{\gamma_j} \to \Unit(V)$  and $\bullet \in \{f,c\}$. Let $R_{X_{\bullet,j},\twist_j}(s)$ be the resolvent of the model funnel (cusp) with the representation $\twist_j$ as in Section~\ref{sec:model}. For $s \in \C$ with $\Re s > 1/2$, we set
\begin{equation}\label{eq:res-X_f}
\begin{gathered}
R_{X_f,\twist}(s)\colon L^2(X_f,\bundle) \to L^2(X_f,\bundle)\,,
\\
R_{X_f,\twist}(s) \coloneqq  R_{X_{f,1},\twist_1}(s) \oplus \dotsc \oplus R_{X_{f,{n_f}},\twist_{n_f}}(s)
\end{gathered}
\end{equation}
and 
\begin{equation}\label{eq:res-X_c}
\begin{gathered}
R_{X_c,\twist}(s)\colon L^2(X_c,\bundle) \to L^2(X_c,\bundle)\,,
\\
R_{X_c,\twist}(s) \coloneqq R_{X_{c,1},\twist_1}(s) \oplus \dotsc \oplus R_{X_{c,{n_c}},\twist_{n_c}}(s)\,.
\end{gathered}
\end{equation}
We define the parametrices in the ends by 
\begin{align}
M_f(s) &\coloneqq (1 - \eta_{f,0}) R_{X_f,\twist}(s) ( 1 - \eta_{f,1})\,,\label{eq:resolvent-mf}
\intertext{and}
M_c(s) &\coloneqq (1 - \eta_{c,0}) R_{X_c,\twist}(s) ( 1 - \eta_{c,1})\,.\label{eq:resolvent-mc}
\end{align}
We now consider $s_0 \in \C$ with $\Re s_0 > 1$. We define the interior parametrix by 
\begin{align*}
M_i(s_0) \coloneqq \eta_{f,2}\eta_{c,2} \ResTwist(s_0) \eta_{f,1}\eta_{c,1}\,.
\end{align*}
We set $M(s,s_0) \coloneqq M_i(s_0)  + M_f(s) + M_c(s)$ and define 
\begin{align}\label{eq:parametrix-error}
L(s,s_0) \coloneqq L_i(s,s_0) + L_f(s) + L_c(s)\,,
\end{align}
where 
\begin{align}
L_i(s,s_0) & \coloneqq -[\LapTwist, \eta_{f,2}\eta_{c,2}] \ResTwist(s_0) \eta_{f,1}\eta_{c,1}\label{eq:resolvent-li} 
\\
&\hphantom{\coloneqq - \LapTwist} + (s(1-s)- s_0(1-s_0)) M_i(s_0)\,, \nonumber
\\
L_f(s) &\coloneqq [\LapTwist, \eta_{f,0}] R_{X_f,\twist}(s) (1 - \eta_{f,1})\,,\label{eq:resolvent-lf} 
\\
L_c(s) &\coloneqq [\LapTwist, \eta_{c,0}] R_{X_c,\twist}(s) (1 - \eta_{c,1})\,,\label{eq:resolvent-lc}
\end{align}
where $[\cdot,\cdot]$ is the commutator. It follows that
\begin{align}\label{eq:parametrix-applied}
(\LapTwist - s(1-s)) M(s,s_0)= \I - L(s,s_0)\,.
\end{align}

\begin{proof}[Proof of Theorem \ref{thm:resolv_merom}]
We have to check that for $s_0$ with $\Re(s_0)$ large enough, the operator $L(s,s_0)$ admits a meromorphic continuation as a function of $s \in \C$.
Moreover, we want to show that for any $s \in \C$ except for the poles of $L(s,s_0)$, the operator $L(s,s_0)$ is compact. We will treat each term in \eqref{eq:parametrix-error} separately.

\smallskip

\paragraph{\textbf{Interior term}}
The operator $L_i(s,s_0)$ is trivially holomorphic in $s \in \C$, since it is polynomial in $s$. Using Selberg's lemma, we find a finite cover $\tilde{X}$ of $X$ where $\tilde{X}$ is a smooth manifold and and we can lift $L_i(s,s_0)$ to an operator
\begin{equation}\label{eq:resolvent-tli}
\begin{aligned}
\widetilde{L}_i(s,s_0) & \coloneqq -[\LapTwist, \eta_{f,2}\eta_{c,2}] \widetilde{\ResTwist}(s_0) \eta_{f,1}\eta_{c,1} 
\\
&\hphantom{\coloneqq -\LapTwist} + (s(1-s)- s_0(1-s_0)) \widetilde{M}_i(s_0)
\end{aligned}
\end{equation}
on $\tilde{X}$. Here, $\widetilde{\ResTwist}(s_0)$ and $\widetilde{M}_i(s_0)$ denote the lifts of $\ResTwist(s_0)$ and $M_i(s_0)$ to $\tilde{X}$, respectively. We note that the Schwartz kernel of $\widetilde{L}_i(s, s_0)$ is compactly supported in $\tilde{X} \times \tilde{X}$. Moreover, the supports of $[\LapTwist, \eta_{f,2}\eta_{c,2}]$ and $\eta_{f,1}\eta_{c,1}$ are disjoint,
and $\widetilde{\ResTwist}(s_0)$ is pseudo-local (i.e., the singular support does not change). Hence the first summand in \eqref{eq:resolvent-tli} is smoothing.
By elliptic regularity, $\widetilde{M}_i(s_0) \colon L^2(\tilde{X}, \widetilde\bundle) \to H_\cpt^2(\tilde{X}, \widetilde\bundle)$. This implies that $\widetilde{L}_i(s,s_0)$ is a compact operator on $L^2(\tilde{X}, \widetilde\bundle)$ for all $s \in \CC$. Consequently, $L_i(s,s_0)$ is a compact operator on $L^2(\tilde{X}, \bundle)$.

\smallskip

\paragraph{\textbf{Funnel term}}
We note that the meromorphic continuation of $L_f(s)$ follows from explicit description of model on the funnel from Proposition~\ref{prop:fourier-funnel}.
As in the case of the interior term, we see that $[\LapTwist, \eta_{f,0}]$ and $(1 - \eta_{f,1})$ have disjoint supports and $R_{X_f,\twist}(s)$ is pseudo-local, hence $L_f(s)$ is a smoothing operator. To prove compactness, we have to consider the explicit form of the model resolvent. As before, we denote by $R_{X_f,\twist}(s;\cdot,\cdot)$ the Schwartz kernel of the operator $R_{X_f,\twist}(s)$. By Proposition~\ref{prop:fourier-funnel}, we have that for  $s  \in \C$ not a resonance of the twisted funnel ends  
\begin{align*}
R_{X_f,\twist}(s; \cdot, \cdot) \in \rho_f^s(\rho_f')^s \CI(\overline{X_f} \times \overline{X_f}, \bundle \boxtimes \bundle')
\end{align*}
away from the diagonal. We have that the image of $[\LapTwist, (1- \eta_{f,0})]$ belongs to compactly supported sections of~$\bundle$. We note that $L_f(s)$ can be written as $[\LapTwist, (1 -\nobreak\eta_{f,0})] R_{X_f,\twist}(s) \eta_{f,1}$. Hence its integral kernel satisfies
\begin{align*}
L_f(s;\cdot,\cdot) \in \rho_f^\infty (\rho_f')^s \CI(\overline{X_f} \times \overline{X_f}, \bundle \boxtimes \bundle')\,.
\end{align*}
This implies that for any $N,N' \in \N$ and $\Re s > 1/2 - N$, the operator
\[L_f(s)\colon \rho_f^N L^2(X_f,\bundle) \to \rho_f^{-N'} L^2(X_f,\bundle)\]
is compact.

\smallskip

\paragraph{\textbf{Cusp term}}
As in the previous case, $L_c(s)$ is meromorphic since the cusp model resolvents are meromorphic by Proposition~\ref{prop:resolv-cont-cusp}. The operator $L_c(s)$ is smoothing since the supports of $[\LapTwist, \eta_{c,0}]$ and $(1 - \eta_{c,1})$ are disjoint and $R_{X_c,\twist}(s)$ is pseudo-local. For compactness, we will show that $L_c(s)$ maps $\rho_c^s(X_c,\bundle)$ to $\CI(X_c, \bundle)$ for $\Re s > 1/2$. We choose $f \in \CmI_c(X_c, \bundle')$ and denote the transpose of $L_c(s)$ with respect to the (real) inner product $\ang{\cdot,\cdot}$ by $L_c(s)^T$. Since $L_c(s)$ is smoothing, so is $L_c(s)^T$ and hence $L_c(s)^T f$ is well-defined. To evaluate $L_c(s)^T f$, we can use the Fourier analysis of Section~\ref{sec:parab_cyl}. Since $f$ is compactly supported, the non-zero Fourier modes in \eqref{eq:ukappa_parab} are asymptotically given by $y^{1/2} K_{s-1/2}(\abs{\kappa}y)$. By \cite[7.23(1)]{Watson44}, the non-zero Fourier modes decay exponentially. The zero Fourier modes contribute terms $\rho_c^{s-1} \CI(\overline{X_c})$. Hence, we have for any $N \in \N$ and  $\Re s > 1/2 - N$ that
\begin{align*}
L_c(s)^T f \in \rho_c^{s-1} \CI(\overline{X_c}, \bundle) \subset \rho_c^{-N} L^2(X_c, \bundle)\,.
\end{align*}
By duality, we have that $L_c(s)\colon \rho_c^N L^2(X_c, \bundle) \to \CI(X_c, \bundle)$. Together with the property that the image of $L_c(s)$ belongs to compactly supported sections of~$\bundle$, we obtain for $\Re s > 1/2 - N$,
\begin{align*}
L_c(s)\colon \rho_c^N L^2(X_c, \bundle) \to \CcI(X_c,\bundle)\,,
\end{align*}
and $L_c(s)\colon \rho_c^N L^2(X_c, \bundle) \to \rho_c^{-N'} L^2(X_c, \bundle)$ is compact for all $N' \in \N$.

\smallskip

\paragraph{\textbf{Application of the analytic Fredholm theorem}}
Using the explicit formula for $L(s,s_0)$, we see that there exists $C > 0$ such that
\begin{align*}
\norm{L_i(s,s_0)} &\leq C \left(1 + \abs{s(1-s) - s_0(1-s_0)} \right)\norm{\ResTwist(s_0)}
\intertext{as well as}
\norm{L_f(s)} &\leq C \norm{ R_{X_f,\twist}(s)} \qquad\text{and}\qquad 
\norm{L_c(s)} \leq C \norm{ R_{X_c,\twist}(s)}\,.
\end{align*}
Thus, for any $s_0 \in \C$ with sufficiently large real part, there exists $s_1 \in \C$ that is not a pole of $R_{X_f,\twist}(s)$ and $R_{X_c,\twist}(s)$ and $s_1 - s_0$ is small such that
\begin{align*}
\norm{L(s_1,s_0)} < 1
\end{align*}
by using the resolvent estimate~\eqref{est:resolvent-twist}. We may invert $\I - L(s_1,s_0)$ using the Neumann series. We can apply the analytic Fredholm theorem to $\I - L(s,s_0)$ to obtain that for any $N \in \N$,
\[
(\I - L(s,s_0))^{-1}\colon \rho^N L^2(X, \bundle) \to \rho^{-N} L^2(X, \bundle)\]
extends to a meromorphic family of operators for $\Re s > \frac{1}{2} - N$.
The poles of $(\I - L(s,s_0))^{-1}$ are of finite multiplicity. The resolvent can now be written as
\begin{align}\label{eq:R_expressed_via_Q_and_L}
\ResTwist(s) = M(s,s_0) (\I - L(s,s_0))^{-1}
\end{align}
and since both factors on the right hand side are meromorphic in $s\in \C$, we conclude that $s \mapsto \ResTwist(s)$ is meromorphic as well.
\end{proof}

\section{Resonance Counting}\label{sec:upper-bound}
Let $z_0 \in \C$ be a resonance.
The \emph{multiplicity} of the resonance $z_0$ is the number
\begin{align*}
m_{X,\twist}(z_0) \coloneqq \rank \int_{\gamma_{\eps,z_0}} \ResTwist(s) \, ds\,,
\end{align*}
where  $\eps > 0$ is chosen such that the path 
\[
\gamma_{\eps,z_0}\colon [0,1] \to \C\,,\qquad \gamma_{\eps,z_0}(t) \coloneqq z_0 + \eps e^{2\pi i t}\,,
\]
encloses exactly one resonance. We define the \emph{multiset of resonances} as
\begin{align}\label{eq:resset}
\ResSet_{X,\twist} \coloneqq \{ (z_0, m) \in \C \times \N \setmid  \text{$z_0$ is a resonance, $m = m_{X,\twist}(z_0)$}\}\,.
\end{align}
We define the \emph{resonance counting function}, $N_{X,\twist}(r)$, as
\begin{align}\label{eq:resonance_counting_function}
N_{X,\twist}(r) \coloneqq \sum_{\substack{(s,m) \in \ResSet_{X,\twist} \\ \abs{s} < r}} m\,.
\end{align}
The goal of this section is to prove Theorem~\ref{thm:upper-bound}, that is, to show 
\[
N_{X,\twist}(r) = O(r^2)\,, \quad r \to \infty\,.
\]

Let $L(s) \coloneqq L(s,s_0)$ be as in \eqref{eq:parametrix-error}. In what follows, we suppress the dependence on $s_0$ for the reader's convenience.

\begin{lemma}\label{lem:zeros-determinant} 
The multiset of resonances of $\ResTwist$ is contained in the union of three copies of $\ResSet_{X_f,\twist} \cup \{(1/2,1)\}$ and the multiset of zeros of
\begin{align*}
D(s) \coloneqq  \det(\I - L_3(s)^3)\,, 
\end{align*}
where $L_3(s) \coloneqq L(s)\eta_{f,3}\eta_{c,3}$. The function $D$ is meromorphic and does not vanish identically on $\C$.
\end{lemma}

\begin{proof}
By \eqref{eq:parametrix-error} we have that $L(s) = \eta_{f,3}\eta_{c,3} L(s)$, from which we obtain that
\begin{align*}
Q(s) \eta_{f,3}\eta_{c,3} &= \ResTwist(s) \eta_{f,3}\eta_{c,3} (\I - L(s) \eta_{f,3}\eta_{c,3})
\\
&= \ResTwist(s) \eta_{f,3}\eta_{c,3} (\I - L_3(s))\,.
\end{align*}
The operator $L_3(s)$ is a compactly supported pseudo-differential operator of order $-1$. Therefore, the operator $L_3(s)^3$ is trace-class and
\begin{align*}
\ResTwist(s) \eta_{f,3} \eta_{c,3} &= Q(s) \eta_{f,3}\eta_{c,3} \,(\I - L_3(s))^{-1} 
\\
&= Q(s) \eta_{f,3}\eta_{c,3} \left( \I + L_3(s) + L_3(s)^2\right) (\I - L_3(s)^3)^{-1}\,.
\end{align*}
Now the first claim follows from \cite[Appendix]{Vodev94}. 

For the second claim, we can use the same arguments as for $L(s)$ to show that $\I - L_3(s)^3$ is invertible as a meromorphic family of operators $\rho^N L^2 \to \rho^{-N} L^2$ for $\Re s > 1/2 - N$. It is well-known that the Fredholm determinant is non-zero if and only if $\I - L_3(s)^3$ is boundedly invertible.
\end{proof}

The previous lemma together with the following proposition will prove Theorem~\ref{thm:upper-bound}.

\begin{prop}\label{prop:zeros-determinant}
We denote by $\mathcal{L} \subset \C \times \N$ the multiset of zeros of $s\mapsto D(s)$. We have that
\begin{align*}
\sum_{\substack{(s,m) \in \mathcal{L}\\\abs{s} < r}} m = O(r^2)
\end{align*}
as $r \to \infty$. 
\end{prop}

The remainder of this section is devoted to the proof of this proposition. We write
\begin{align*}
L_3(s) = L_i(s) + K_f(s) + K_c(s)\,,
\end{align*}
where $K_\bullet(s) \coloneqq L_\bullet(s) \eta_{f,3}\eta_{c,3}$ for $\bullet \in \{f,c\}$. The funnel term $K_f(s)$ is given by
\begin{align}
K_f(s) &= L_f(s) \eta_{f,3}\eta_{c,3}\nonumber
\\
&= [\LapTwist, \eta_{f,0}] R_{X_f,\twist}(s) (1 - \eta_{f,1}) \eta_{f,3}\eta_{c,3}\,,\nonumber
\\
&= [\LapTwist, \eta_{f,0}] R_{X_f,\twist}(s) (\eta_{f,3} - \eta_{f,1})\,.\label{eq:K_funnel}
\end{align}
Similarly, we have
\begin{align}\label{eq:K_cusp}
K_c(s) &= [\LapTwist, \eta_{c,0}] R_{X_c,\twist_c}(s) (\eta_{c,3} - \eta_{c,1})\,.
\end{align}
The terms $L_i(s)^3, K_f(s)^3$, and $K_c(s)^3$ are all trace-class. Therefore we can apply \cite[Lemma 6.1]{GuZw95} to obtain the inequality 
\begin{equation}\label{eq:determinant-into-model}
\begin{aligned}
\abs{D(s)} &= \abs*{ \det\left( \I - (L_i(s) + K_f(s) + K_c(s))^3 \right) }
\\
&\leq \left[ \det( \I + 9\abs{L_i(s)}^3) \det( \I + 9\abs{K_f(s)}^3) \det( \I + 9\abs{K_c(s)}^3)\right]^9\,.
\end{aligned}
\end{equation}
In what follows, we estimate \eqref{eq:determinant-into-model} termwise. It is straight-forward to estimate the first term, $\det( \I + 9\abs{L_i(s)}^3)$.

\begin{lemma}\label{lem:interior-determinant}
For any $c > 0$ there exists $C > 0$ such that for all~$s\in\C$ we have
\begin{align*}
\abs*{\det \left( \I + c \abs{L_i(s)}^3 \right)} \leq e^{C \ang{s}^2}\,.
\end{align*}
\end{lemma}

\begin{proof}
The proof follows the lines of \cite[Lemma~9.7]{Borthwick_book}. However, our setup requires a different explanation for the singular value estimates. By \eqref{eq:resolvent-li} and \cite[Lemma 6.1]{GuZw95}, we have that
\begin{align*}
\det \left( \I + c \abs{L_i(s)}^3 \right) \leq C_1 \det\left( \I + 4c^2 \abs*{ (s(1-s) - s_0(1-s_0) ) M_i(s_0)}^3 \right)^6\,.
\end{align*}
The operator $M_i(s_0)$ is a vector-valued pseudo-differential operator of order $-2$. The corresponding Weyl law (see~\cite{GrubbFunctionalCalculus}) implies the bound
\begin{align*}
\mu_k(M_i(s_0)^3) = O(k^{-3})\,.
\end{align*}
Hence, we have that
\begin{align*}
\det \left( \I + c \abs{L_i(s)}^3 \right) \leq \prod_{k=1}^\infty \left( 1 + C \frac{\ang{s}^6}{k^3} \right)\,.
\end{align*}
By the same arguments as in the proof of \cite[Lemma~9.7]{Borthwick_book}, we obtain the claimed estimate.
\end{proof}

To estimate $\det( \I + 9\abs{K_f(s)}^3) \det( \I + 9\abs{K_c(s)}^3)$, we have to consider three distinct cases for $s \in \C$. For $\eps \in (0,1/2)$, we define 
\begin{align*}
\Omega_1 & \coloneqq \{s \in \C \setmid \Re s \in [1/2 + \eps,\infty) \}\,,
\\
\Omega_2 & \coloneqq \{s \in \C \setmid \Re s \in (1/2-\eps,1/2+\eps) \}\,,
\\
\Omega_3 & \coloneqq \{s \in \C \setmid \Re s \in (-\infty,1/2-\eps] \}\,.
\end{align*}
For $s \in \Omega_1$, we use the resolvent estimate~\eqref{est:resolvent-twist2}. For $s \in \Omega_2$, we use the model calculations from Section~\ref{sec:model}. And for $s \in \Omega_3$, we obtain the estimate by a  reflection argument using the Fourier decomposition of the resolvent and the case $s \in \Omega_1$.

\subsection{The Cusp Term}\label{sec:cusp-term}
The operators $K_f(s)$ and $K_c(s)$ need not be holomorphic in $s \in \C$ as there might exist isolated singularities coming from the resonances. In the case of $K_c(s)$, we know from Proposition~\ref{prop:fourier-cusp} that there is at most a single resonance at $s = 1/2$, namely with multiplicity $n_c^\twist$ as defined in \eqref{eq:nctwist}. Motivated by this, we define
\begin{equation}\label{eq:def_gc}
g_c(s) \coloneqq (2s - 1)^{n^\twist_c}\,.
\end{equation}
It suffices to estimate $g_c(s) \det(\I + c_1 \abs{K_c(s)})$ for $c_1 > 0$. Without loss of generality, we may assume that we have only one cusp, that is $X_c = F_\infty$. We recall that $C_\infty = \ang{T} \bs \h$, where $T.z = z+1$
and the restriction of $\twist$ to $X_c$ corresponds to a unitary representation $\ang{T} \to \Unit(V)$, which we also denote by $\twist$. The number $n^\twist_c$ is then given by the dimension of the $1$-eigenspace of $\twist(T)$.

\begin{prop}\label{prop:cusp-determinant}
For any $c_1 > 0$ there exists $C > 0$ such that
\begin{align*}
\abs*{g_c(s) \det( \I + c_1 \abs{K_c(s)}) } \leq e^{C \ang{s}^2}
\end{align*}
for all $s \in \C$.
\end{prop}

We start by proving Proposition~\ref{prop:cusp-determinant} in the case that $s \in \Omega_1$.

\begin{lemma}\label{lem:cusp-a1}
For any $c_1 > 0$ there exists $C > 0$ such that
\begin{align*}
\abs*{\det( \I + c_1 \abs{K_c(s)}) } \leq e^{C \ang{s}^2}
\end{align*}
for all $s \in \Omega_1$.
\end{lemma}

\begin{proof}
The operator $\Delta_{C_\infty,\twist}$ does not have any discrete eigenvalues, since there are no resonances with $\Re(s) > 1/2$. Hence its spectrum is contained in $[1/4,\infty)$. Therefore the resolvent estimate~\eqref{est:resolvent-twist2} holds for $s \in \Omega_1$ and we obtain the $L^2$-bound
\begin{align*}
\norm{R_{C_\infty,\twist}(s)}_{L^2 \to L^2} \lesssim \ang{s}^{-1}
\end{align*}
for $\Re s \geq 1/2 + \eps$. Applying Proposition~\ref{prop:bound-resolvent-determinant} with $\tau = - 1$ yields the claimed bound on the determinant.
\end{proof}

For the remaining regions, we use the Fourier coefficients that we provided in Section~\ref{sec:parab_cyl}. We fix an eigenbasis $(\psi_j)_{j=1}^{\dim V}$ of $\twist(T) \in \Unit(V)$ with respective eigenvalues $e^{2\pi i \vartheta_j}$, $\vartheta_j \in [0, 1)$, and let $R_{C_\infty,\twist}(s;z,z')^j$ be as in \eqref{RCinftychiszz}. Lemma~\ref{lem:fourier-cusp} implies that we have the Fourier expansion 
\begin{align*}
R_{C_\infty,\twist}(s;z,z')^j = \sum_{k \in \Z} e^{2\pi i(k + \vartheta_j) x} u_{2\pi (k + \vartheta_j)}(s;y,y') e^{-2\pi i(k + \vartheta_j) x'}\,,
\end{align*}
where $u_\kappa(s;y,y')$ is given by \eqref{eq:ukappa_parab} and \eqref{eq:uzero_parab}. We distinguish the two cases: 
\[
\vartheta_j + k = 0\qquad\text{and}\qquad \vartheta_j + k \not = 0\,.
\]
In the former case, we have that $k = \vartheta_j = 0$.

\begin{lemma}\label{lem:est-uparab_zero}
Let $\varphi_0,\varphi_1 \in \CcI( \R_+)$ with disjoint supports. For $s \in \Omega_2$, $\abs{\Im s} > 1/2$, we have that
\begin{align*}
\sup_{y,y' \in \R_+} \,\abs*{ \varphi_0(y) u_0(s;y,y') \varphi_1(y')} \lesssim \ang{s}^{-1}\,.
\end{align*}
The implicit constant may depend on the support of $\varphi_0$ and $\varphi_1$.
\end{lemma}

\begin{proof}
By~\eqref{eq:uzero_parab}, we have that 
\begin{align*}
u_0(s;y,y') \coloneqq \frac{1}{2s-1} 
\begin{cases}
y^s (y')^{1-s} & \text{for $y \leq y'$}\,,
\\
y^{1-s} (y')^{s} & \text{for $y > y'$}\,.
\end{cases}
\end{align*}
This directly gives the estimate.
\end{proof}

\begin{lemma}\label{lem:est-uparab_kappa}
    Let $\kappa \not = 0$ and $\varphi_0,\varphi_1$ as in Lemma~\ref{lem:est-uparab_zero}. There exists $C > 0$ such that for $s \in \Omega_2$, we have
    \begin{align*}
    \sup_{y,y' \in \R_+} \,\abs*{ \varphi_0(y) u_\kappa(s;y,y') \varphi_1(y')} \leq \ang{s} e^{-C \abs{\kappa}}\,.
    \end{align*}
\end{lemma}

\begin{proof}
By \eqref{eq:ukappa_parab}, we have that
\begin{align*}
u_\kappa(s;y,y') = \begin{cases}
\sqrt{yy'} I_{s-1/2}(\abs{\kappa}y) K_{s-1/2}(\abs{\kappa}y') & \text{for $y\leq 
y'$}\,,
\\
\sqrt{yy'} K_{s-1/2}(\abs{\kappa}y) I_{s-1/2}(\abs{\kappa}y') & \text{for $y > y'$}\,,
\end{cases}
\end{align*}
where $I_\nu$ and $K_\nu$ are the modified Bessel functions. For $\Re s \geq 1/2 - \eps$, $w > 0$, and $\delta \in (0,1)$, we have the estimates (see \cite[p.~202--203]{Borthwick_book})
\begin{align}
\abs{I_{s-1/2}(w)} \lesssim \abs{\gammafunc(s)}^{-1} (w/2)^{\Re s - 1/2} e^w\,,\label{eq:I-bessel-bound}
\\
\abs{K_{s-1/2}(w)} \lesssim \abs{s\gammafunc(s)} (w/2)^{-\Re s - 1/2} e^{-\delta w}\,. \label{eq:K-bessel-bound}
\end{align}
Without loss of generality, we may suppose that $y < y'$ on $\supp (\varphi_0 \otimes \varphi_1)$. Then the function $u_\kappa$ becomes
\begin{align*}
u_\kappa(s;y,y') = \sqrt{yy'} I_{s-1/2}(\abs{\kappa}y) K_{s-1/2}(\abs{\kappa}y')
\end{align*}
There exists $C_1 > 0$ such that for any $\delta \in (0,1)$, the function $u_k$ can be estimated by
\begin{align*}
\abs{u_\kappa(s,y,y')} \leq C_1 \frac{\abs{s}}{\abs{\kappa}} \left(\frac{y}{y'}\right)^{\Re s} e^{\abs{\kappa} (y-\delta y')}\,.
\end{align*}
Let $\delta \in (0,1)$ and $C > 0$ such that $y - \delta y' \leq -\frac{C}{2\pi} < 0$ on  $\supp(\varphi_0 \otimes \varphi_1)$. We obtain the bound
\begin{align*}
\abs{u_\kappa(s; y,y')} \lesssim \ang{s} e^{-C \abs{\kappa}}\,.
\end{align*}
\end{proof}

\begin{lemma}\label{lem:cusp-a2}
For any $c_1 > 0$ there exists $C > 0$ such that
\begin{align*}
\abs*{(2s-1)^{n_c^\twist} \det( \I + c_1 \abs{K_c(s)}) } \leq e^{C \ang{s}^3}
\end{align*}
for all $s \in \Omega_2$.
\end{lemma}

\begin{proof}
Lemma~\ref{lem:est-uparab_zero} and Lemma~\ref{lem:est-uparab_kappa} together with Proposition~\ref{prop:bound-resolvent-determinant} taking $\tau = 1$ imply the bound of $\det( \I + c_1 \abs{K_c(s)})$ for $s \in \Omega_2$ and $\abs{\Im s > 1/2}$. For $s \in \Omega_2$ and $\abs{\Im s \leq 1/2}$, we note that the number of $u_0$ Fourier modes is at most $n_c^\twist$. Therefore each prefactor $(2s-1)^{-1}$ in $u_0$ is canceled out by $(2s - 1)^{n_c^\twist}$. Consequently, $(2s-1)^{n_c^\twist} \det( \I + c_1 \abs{K_c(s)})$ is bounded.
\end{proof}

\begin{lemma}\label{lem:cusp-a3}
For any $c_1 > 0$ there exists $C > 0$ such that
\begin{align*}
\abs*{\det( \I + c_1 \abs{K_c(s)}) } \leq e^{C \ang{s}^2}
\end{align*}
for all $s \in \Omega_3$.
\end{lemma}

\begin{proof}
For $s \in \Omega_3$, we write
\begin{align}\label{eq:Rcinftys1-s}
R_{C_\infty,\twist}(s) = R_{C_\infty,\twist}(1-s) + A(1-s)\,,
\end{align}
where
\begin{align*}
A(s) & \coloneqq \sum_{j=1}^{\dim(V)} \sum_{k \in \Z} A_k^j(s;z,z')
\intertext{with}
A_k^j(s;z,z') & \coloneqq e^{2\pi i (\vartheta_j + k)(x-x')} a_k^j(s;y,y') \ang{\psi_j, \cdot} \psi_j
\intertext{and}
a_k^j(s;y,y') & \coloneqq u_{2\pi (k + \vartheta_j)}(1-s;y,y') - u_{2\pi (k + \vartheta_j)}(s;y,y')\,.
\end{align*}
We use \eqref{eq:Rcinftys1-s} to investigate the behavior of $R_{C_\infty,\twist}(s)$ for $s \in \Omega_3$. It is sufficient to estimate $A(s)$ and use the result of Lemma \ref{lem:cusp-a1} to estimate $R_{C_\infty,\twist}(s)$ for $s \in \Omega_1$ (that is, for $1-s \in \Omega_3$). As before, we distinguish between the cases $k + \vartheta_j = 0$ and $k + \vartheta_j \not = 0$.

For the first case, we have that
\begin{align*}
a_0^j(s;y,y') = \frac{1}{2\pi (1 - 2s)} (y^{1-s} (y')^s - y^s (y')^{1-s})
\end{align*}
and hence for any fixed $\ell \in \N_0$, 
\begin{align*}
\abs*{\pa_y^\ell a_0^j(s;y,y')} \lesssim e^{C\Re s}
\end{align*}
for $s \in \Omega_1$ and any $j=1, \ldots, \dim(V)$ and the implicit constant and $C$ are independent of $s$. 

The definition of the modified Bessel function of the second kind (\cite[p. 78, (6-7)]{Watson44} and \cite[p. 79, (8)]{Watson44}) implies that 
\begin{align}
I_{-\nu}(w) &= I_\nu(w) + \frac{2}{\pi} \sin(\pi \nu) K_\nu(w)\,, \label{eq:I-bessel-connection} 
\\ 
K_{-\nu}(w) &= K_\nu(w)\,.\label{eq:K-bessel-connection}
\end{align}
Moreover, we need \cite[p.79, (4)]{Watson44}
\begin{align}
\pa_w K_\nu(w) &= -K_{\nu + 1}(w) + \frac{\nu}{w}K_\nu(w)\,.\label{eq:K-bessel-deriv}
\end{align}
For $k + \vartheta_j \not = 0$, we observe that similarly as in \cite[p. 204]{Borthwick_book}, we have that
\begin{align*}
a_k^j(s;y,y') = - \frac{2}{\pi} \cos(\pi s) \sqrt{yy'} K_{s-1/2}(2\pi \abs{k + \vartheta_j} y) K_{s-1/2}(2\pi \abs{k + \vartheta_j} y')\,,
\end{align*}
which is a direct consequence of \eqref{eq:K-bessel-connection} and \eqref{eq:I-bessel-connection}. For any fixed $\ell \in \N_0$ we obtain the following bound on the function $a_k^j$ by using \eqref{eq:K-bessel-bound} and \eqref{eq:K-bessel-deriv}:
\begin{align*}
\abs*{\pa_y^\ell a_k^j(s;y,y')} \lesssim e^{C\ang{s}} \left( \frac{\ang{s}}{\abs{k+\vartheta_j}}\right)^{\Re s} e^{-c\,\abs{k + \vartheta_j}}\,,
\end{align*}
where the constants are independent of $k+\vartheta_j$ and $s$. This gives an estimate
\begin{align*}
& \norm*{ [\LapTwist,  \eta_{c,0}] A_k^j(s) (\eta_{c,3} - \eta_{c,1})}
\\
& \hphantom{ [\LapTwist,  \eta_{c,0}] A_k^j(s) } \lesssim
\begin{cases}
e^{C\Re s} & \text{if $k+\vartheta_j = 0$}\,,
\\
e^{C\ang{s}} \left( \frac{\ang{s}}{\abs{k+\vartheta_j}}\right)^{\Re s} e^{-c\abs{k + \vartheta_j}} & \text{if $k + \vartheta_j \not = 0$}\,.
\end{cases}
\end{align*}
The operator $A_k^j(s)$ has rank $1$ for $k + \vartheta_j \not = 0$. For $k + \vartheta_j = 0$, its rank is at most $2$. The operators $A_k^j(s)$ and $A_{k'}^{j'}(s)$ are orthogonal if $j \neq j'$ or $k \neq k'$. Since $\eta_{c,\bullet}$ is independent of $x$, we have
\[
[\LapTwist,\eta_{c,0}] = - y^2 \left(\pa_y\eta_{c,0} \pa_y + \pa_y^2 \eta_{c,0}\right)\,.
\]
Therefore the orthogonality continues to hold for $[\LapTwist, \eta_{c,0}] A_k^j(s) (\eta_{c,3} - \eta_{c,1})$ and $[\LapTwist, \eta_{c,0}] A_{k'}^{j'}(s) (\eta_{c,3} - \eta_{c,1})$. Since we decomposed the operator into a sum of finitely many rank $2$ operators and infinitely many rank $1$ operators,
we can use the min-max characterization of the singular values to estimate the singular values by the norm of the rank $1$ operators. Thus, for any $m \in \N_{+}$ we have
\begin{align*}
\mu_m\left([\LapTwist, \eta_{c,0}] A(s) (\eta_{c,3} - \eta_{c,1})\right) \leq e^{C\ang{s}} \left( \frac{\ang{s}}{m} \right)^{\Re s} e^{-cm}
\end{align*}
for some constants $C,c > 0$ independent of $s$ and $m$. For $s \in \Omega_1$, using the elementary inequality $\mu_m(A + B) \leq \mu_1(A) + \mu_m(B)$ and $\mu_1(A) = \norm{A}$, we obtain
\begin{align*}
\mu_m(K_c(1-s)) \leq \norm{K_c(s)} + \mu_m([\LapTwist, \eta_{c,0}] A(s) (\eta_{c,3} - \eta_{c,1}))\,.
\end{align*}
This implies that
\begin{align*}
\abs*{\det(\I + c_1\abs{K_c(1-s)})} &\lesssim \prod_{m=1}^\infty \left( 1 + \mu_m( [\LapTwist, \eta_{c,0}] A(s) (\eta_{c,3} - \eta_{c,1}) ) \right)
\\
&\lesssim \prod_{m = 1}^\infty e^{C\ang{s}} \left( \frac{\ang{s}}{m} \right)^{\Re s} e^{-c m}
\\
&\lesssim e^{C\ang{s}^2}\,.
\end{align*}
This completes the proof.
\end{proof}

We have now estimates for $\det(\I + c_1 \abs{K_c(s)})$ for all $s \in \C$ and hence can prove Proposition~\ref{prop:cusp-determinant}.

\begin{proof}[Proof of Proposition~\ref{prop:cusp-determinant}]
Combining the estimates for the three regions, Lemma~\ref{lem:cusp-a1}, Lemma~\ref{lem:cusp-a2} and Lemma~\ref{lem:cusp-a3}, we obtain
\begin{align*}
\abs*{g_{c}(s) \det( \I + c_1 \abs{K_c(s)}) } \leq
\begin{cases}
e^{C\ang{s}^2}\,, & \abs{\Im s - 1/2} \geq \eps\,,
\\
e^{C\ang{s}^3}\,, & \abs{\Im s - 1/2} \leq \eps\,.
\end{cases}
\end{align*}
Using the Phragm\'en--Lindelöf theorem (see for instance~\cite[Sect.~XII, Theorem 6.1]{LangComplexAnalysis}), we obtain the claimed bound.
\end{proof}

\subsection{The Funnel Term}
As in the case of the cusp term, without loss of generality, we suppose without loss of generality that ${X_f = F_\ell}$. That is, we treat only one funnel end $F_\ell$, which is half of the hyperbolic cylinder $C_\ell = \ang{h_\ell} \bs \h$ with length $\ell \in \R_+$. Let $(\psi_j)_{j=1}^{\dim V}$ be an eigenbasis of $\twist(h_\ell)$ with respective eigenvalues $e^{2\pi i \vartheta_j}$ and let $R_{F_\ell,\twist}(s;z,z')^j$ be as in \eqref{eq:RCell_matcoeff}. We recall from \eqref{eq:fourier-exp-funnel} that the resolvent admits a Fourier expansion,
\begin{align*}
R_{F_\ell,\twist}(s;z,z')^j = \frac{1}{\ell} \sum_{k \in \Z} e^{i \phi(k+\vartheta_j) } \tilde{v}_{k + \vartheta_j}(s;r,r') e^{-i\phi' (k+\vartheta_j)}\,,
\end{align*}
where $\tilde{v}_\kappa(s;r,r')$ is as in \eqref{eq:vtilde-kappa}. To estimate the determinant $\det(\I + c \abs{K_f(s)})$, we have to cancel the poles coming from the model resolvent. Recall that $\ResSet_{F_\ell,\twist}$, defined in Proposition \ref{prop:fourier-funnel}, is the multiset of resonances of the resolvent $R_{F_\ell,\twist}(s)$. We also define the set $\widetilde{\ResSet_{F_\ell,\twist}}$ as 
\[
 \widetilde{\ResSet_{F_\ell,\twist}} \coloneqq \mc \ResSet_{F_\ell,\twist} \cup i\mc \ResSet_{F_\ell,\twist}\,,
\]
counted with multiplicities. The Weierstrass product of~$\widetilde{\ResSet_{F_\ell,\twist}}$ is then given by 
\begin{align}\label{def:gFellchi}
g_{F_\ell,\twist}(s) \coloneqq \mc P_{F_\ell,\twist}(s) \cdot \mc P_{F_\ell,\twist}(-is)\,,
\end{align}
where $\mc P_{F_\ell,\twist}(s)$ denotes the Weierstrass product of $\ResSet_{F_\ell,\twist}$. By the upper bound on the resonances for the model funnel, Remark~\ref{rem:upper-bound-funnel}, we see that the convergence exponent of the Weierstrass product $\mc P_{F_\ell,\twist}(s)$ is $2$, so by \cite[Theorem~2.6.2]{Boas_entire} the function $g_{F_\ell,\twist}$ is entire and of order $2$. Due to the cancellation of zeros
\begin{align*}
\sum_{\lambda \in\widetilde{\ResSet_{F_\ell,\twist}}\,, \abs{\lambda} < r} \lambda^{-2} = 0\,.
\end{align*}
Lindelöf's theorem~\cite[Theorem~2.10.1]{Boas_entire} then implies that the function $g_{F_\ell,\twist}$ is of finite type, which means that
\begin{align}\label{eq:gf-bound}
\abs{g_{F_\ell,\twist}(s)} \leq e^{C\ang{s}^2}\,.
\end{align}

\begin{lemma}\label{lem:funnel-determinant}
For any $c > 0$ there exists $C > 0$ such that
\begin{align*}
\abs*{g_{F_\ell,\twist}(s) \det( \I + c \abs{K_f(s)}) } \leq e^{C \ang{s}^2}
\end{align*}
for all $s \in \C$.
\end{lemma}

As in the case of the cusps, we split $\C = \Omega_1 \cup \Omega_2 \cup \Omega_3$, where
\begin{align*}
\Omega_1 &= \{s \in \C \setmid \Re s \in [1/2 + \eps,\infty) \}\,,
\\
\Omega_2 &= \{s \in \C \setmid \Re s \in (1/2-\eps,1/2+\eps) \}\,,
\\
\Omega_3 &= \{s \in \C \setmid \Re s \in (-\infty,1/2-\eps] \}
\end{align*}
and $\eps \in (0,1/2)$. For $s \in \Omega_1$, we have the following estimate. 

\begin{lemma}\label{lem:funnel-a1}
Let $c > 0$, then there exists $C>0$ such that for all $s \in \Omega_1$,
\begin{align*}
\det ( \I + c \abs{K_f(s)}) \leq e^{C\ang{s}^2}\,.
\end{align*}
\end{lemma}

\begin{proof}
The proof is the same as for Lemma~\ref{lem:cusp-a1} with $C_\infty$ replaced by $F_\ell$.
\end{proof}

The estimate for $s \in \Omega_2$ is as follows.

\begin{lemma}\label{lem:funnel-a2}
Let $c > 0$, then there exists $C>0$ such that for all $s \in \Omega_2$,
\begin{align*}
\det ( \I + c \abs{K_f(s)}) \leq e^{C\ang{s}^{5/2}}.
\end{align*}
\end{lemma}

\begin{proof}
Suppose that $\varphi_0,\varphi_1 \in \CcI(F_\ell)$ have disjoint supports and let $s \in \Omega_2$. We begin by showing that
\begin{align}\label{eq:norm-funnel-a2}
\norm{\varphi_0 R_{F_\ell,\twist}(s) \varphi_1} \leq C
\end{align}
for some $C > 0$ only depending on $\eps$. It suffices to show the bound for the hyperbolic cylinder $C_\ell$. By \eqref{eq:def_Rhkernel}, \eqref{eq:resolvent_hypcyl} and \cite[2.1.3]{ErdelyiI}, we have the explicit formula for the integral kernel of the resolvent
\begin{align*}
R_{C_\ell,\twist}(s;z,z') = \frac{1}{4\pi} \sum_{k \in \Z} \twist(h_ \ell)^k \int_0^1 \frac{(t(1-t))^{s-1}}{(\sigma(z,e^{k\ell}z') - t)^s} \,dt
\end{align*}
for $\Re s > 0$. By  \cite[p. 193]{Borthwick_book}, we have the following bound for $z \in \supp \varphi_0$ and $z' \in \supp \varphi_1$:
\[
\sigma(z,e^{k\ell}z') - 1 \geq e^{c_1 \abs{k\ell} - c_2}
\]
for some $c_1,c_2 > 0$ depending only on the minimum of $d_\h(z,z')$ on the set $\supp\varphi_0 \times \supp \varphi_1$. Together with the fact that $\twist(h_ \ell)$ is unitary, we arrive at the bound
\begin{align*}
\norm*{\varphi_0(z) R_{C_\ell,\twist}(s;z,z') \varphi_1(z')}_{\mathcal{L}(V,V)} & \leq \\  \frac{1}{4\pi} \sum_{k\in \Z} e^{(c_2-c_1\abs{k\ell})\Re s} &  \int_0^1 (t(1-t))^{\Re s - 1}\, dt\,,
\end{align*}
which proves \eqref{eq:norm-funnel-a2}. To finish the proof of the lemma, we apply Proposition~\ref{prop:bound-resolvent-determinant} with $\tau = 0$.
\end{proof}

Finally, for $s\in\Omega_3$, we obtain the following estimate.

\begin{lemma}\label{lem:funnel-a3}
Let $c > 0$, then there exists $C>0$ such that for all $s \in \Omega_3$,
\begin{align*}
\abs*{ g_{F_\ell,\twist}(s)\det ( \I + c \abs{K_f(s)}) } \leq e^{C\ang{s}^2}\,.
\end{align*}
\end{lemma}

We use the similar reflection argument as for the proof of Lemma~\ref{lem:cusp-a3}. From \eqref{eq:resolvent-scattering-matrix}, we recall that
\begin{align*}
R_{F_\ell,\twist}(s) = R_{F_\ell,\twist}(1-s) + (2s-1) E_{F_\ell,\twist}(1-s) S_{F_\ell,\twist}(s) E_{F_\ell,\twist}(1-s)^T\, .
\end{align*}
In order to prove Lemma~\ref{lem:funnel-a3}, we need to estimate $S_{F_\ell,\twist}(s)$ for $s \in \Omega_3$ and $E_{F_\ell,\twist}(s)$ for $s \in \Omega_1$. We introduce the sets
\begin{align*}
\tilde{\mathcal{R}}_0 & \coloneqq 1 - 2\N_0\,,
\\
\mathcal{R}_0 & \coloneqq 1 - 2\N_0 + i\omega \Z\setminus\{0\}\,,
\\
\mathcal{R}_{1/2} & \coloneqq 1 - 2\N_0 + i \omega (1/2 + \Z)\,,
\\
\mathcal{R}_{\vartheta} & \coloneqq \bigcup_{p \in \{\pm 1\}}
\left( 1 - 2\N_0 + ip\hspace{1pt}\omega (\vartheta + \Z)\right),\quad \vartheta \not \in \{0,1/2\}\,,
\end{align*}
where we denote by $m_{\vartheta}$ the multiplicity of an eigenvalue $\lambda = e^{2\pi i\vartheta}$ of $\twist(h_\ell)$ and we recall that $\omega = 2\pi / \ell$. We note that
\begin{align*}
\ResSet_{F_\ell,\twist} = \tilde{\mathcal{R}}_0^{2m_0}
\cup \mathcal{R}_0^{2m_0}
\cup \mathcal{R}_{1/2}^{2m_{1/2}}
\cup \bigcup_{\vartheta_j \not = \{0,1/2\}} \mathcal{R}_{\vartheta_j}^{m_{\vartheta_j}}
\end{align*}
with multiplicities. For $n \in \N_{+}$ we define $d_{n,\vartheta}(s)$ as follows: for $\vartheta \in (0,1)$, $\vartheta\not=1/2$, we set
\begin{align*}
d_{n,\vartheta}(s) \coloneqq 
\begin{cases}
\dist(s,\mathcal{R}_{\vartheta})^{-1} & \text{if $n \leq m_{\vartheta}$}\,,
\\
1 & \text{if $n > m_{\vartheta}$}\,.
\end{cases}
\end{align*}
For $\vartheta \in \{0,1/2\}$, we set
\begin{align*}
d_{n,\vartheta}(s) \coloneqq 
\begin{cases}
\dist(s,\mathcal{R}_{\vartheta})^{-1} & \text{if $n \leq 2m_{\vartheta}$}\,,
\\
1 & \text{if $n > 2m_{\vartheta}$}\,.
\end{cases}
\end{align*}
Moreover, we define the function $\tilde{d}_{k,0}$ by
\begin{align*}
\tilde{d}_{n,0}(s) \coloneqq 
\begin{cases}
\dist(s,\tilde{\mathcal{R}}_{0})^{-2} & \text{if $n \leq m_0$}\,,
\\
1 & \text{if $n > m_0$}\,.
\end{cases}
\end{align*}
Lastly, we define the product
\begin{align}\label{eq:def_dks}
d_n(s) \coloneqq \tilde{d}_{n,0}(s) \cdot \prod_{\vartheta} d_{n,\vartheta}(s)\,.
\end{align}
We remark that this product is well-defined even if we consider it as an infinite product over the uncountable index set $[0,1)$. In this case,  almost all factors are equal to $1$ because for $e^{2\pi i\vartheta}$, $\vartheta\in[0,1)$, not being an eigenvalue of~$\twist(h_\ell)$ we have set $m_\vartheta = 0$.
For later use we also define
\begin{align*}
d_{\ResSet_{F_\ell,\twist}}(s) &\coloneqq \prod_n d_n(s) 
\\
&= \dist(s,\tilde{\mathcal{R}}_0)^{-2m_0} \dist(s,\mathcal{R}_0)^{-2m_0}
\dist(s,\mathcal{R}_{1/2})^{-2m_{1/2}} 
\\
&\hphantom{= \quad} \times 
\prod_{\vartheta_j \not \in \{0,1/2\}} \dist(s,\mathcal{R}_{\vartheta_j})^{-m_{\vartheta_j}}\,.
\end{align*}

\begin{lemma}\label{lem:svalues-smatrix}
For $s \in \Omega_3$, we have
\begin{align*}
\mu_n(S_{X_f,\twist}) \leq e^{C\ang{s}} \ang{s}^{1 - 2\Re s} \times  \begin{cases}
d_n(s) & \text{for $n \leq \max \{m_0, 2m_{\vartheta_j}\}$}\,,
\\
n^{2\Re s - 1} & \text{for $n > \max \{m_0, 2m_{\vartheta_j}\}$}\,.
\end{cases}
\end{align*}
\end{lemma}

\begin{proof}
The proof is similar to~\cite[Lemma 4.2]{GuZw95} and~\cite[Lemma 9.15]{Borthwick_book}, but we have to be more careful when estimating the singular values due to the multiplicities of the eigenvalues of~$\twist(h_\ell)$. By definition \eqref{eq:smatrix-funnel}, the Fourier coefficients $S_{F_\ell,\twist}(s)^j_k$ of the twisted scattering matrix are closely related to the untwisted Fourier coefficients. Thus, 
\begin{align*}
S_{F_\ell,\twist}(s)_k^j = \frac{4^{-s} (s-1/2) \cos(\pi s) \gammafunc(1/2 - s)^2}{f(s+i(k + \vartheta_j)\omega) f(s - i(k + \vartheta_j)\omega)}\,,
\end{align*}
where $\omega = \frac{2\pi}{\ell}$, and
\begin{align*}
f(z) = \cos(z\cdot \pi/2) \gammafunc(1 - z)\,.
\end{align*}
We use Stirling's formula to estimate both the numerator and the denominator. We note that  
\begin{align}\label{eq:smatrix-numerator}
\abs*{ 4^{-s} (s-1/2) \cos(\pi s) \gammafunc(1/2-s)^2} \leq e^{C\ang{s}} \ang{s}^{1-2\Re s}\,.
\end{align}
Moreover, for any $\tau > 0$ there exists $C > 0$ such that
\begin{align}\label{eq:smatrix-denominator}
\abs{f(z)}^{-1} \leq C 
\begin{cases}
\dist(z, 2\N_0 - 1)^{-1} & \text{if $\abs{\Im z} < \tau$}\,,
\\
e^{-\Rea z} \abs{\Ima z}^{\Rea z - 1/2} & \text{if $\abs{\Im z} \geq \tau$}
\end{cases}
\end{align}
for $\Re z \leq 1/2 - \eps$. If all eigenvalues $\lambda_j = e^{2 \pi i\vartheta_j}$ of $\twist(h_\ell)$ are equal, we set $\delta \coloneqq \omega/2$, otherwise
\begin{align*}
\delta \coloneqq \frac{\omega}{2}\,\min_{\substack{\vartheta_j \not = \vartheta_k\\ m \in \Z}} \abs{\vartheta_j - \vartheta_k + m}\,.
\end{align*}
This implies that for any $s \in \C$ and $p \in \{\pm 1\}$ there is at most one $\vartheta_j$ and $k \in \Z$ such that $\abs{\Im s + p\hspace{1pt}\omega(\vartheta_j + k)} < \delta$. On the set 
\[
\left\{s \in \C \setmid \abs{\Im s + p\hspace{1pt}\omega(\vartheta_j + k)} < \delta\right\}\,,
\] 
we can estimate the Fourier coefficients of the scattering matrix using \eqref{eq:smatrix-numerator} and \eqref{eq:smatrix-denominator} as follows:
\begin{align}\label{eq:estimate_smatrix_near_resonance}
\abs*{S_{F_\ell,\twist}(s)^j_k} \leq e^{C\ang{s}} \ang{s}^{1-2\Re s} \begin{cases}
\dist(s, \tilde{\mathcal{R}}_0)^{-2} & \text{if $\vartheta_j + k = 0$}\,,
\\
\dist(s, \mathcal{R}_{\vartheta_j})^{-1} & \text{if $\vartheta_j + k \not = 0$}\,.
\end{cases}
\end{align}
For $s \in \C$ with $\abs{\Im s + p\hspace{1pt}\omega(\vartheta_j + k)} \geq \delta$, we have
\begin{align}\label{eq:estimate_smatrix_away_resonance}
\abs*{S_{F_\ell,\twist}(s)^j_k} \leq e^{C\ang{s}} \left( \frac{\ang{s}^2}{\abs{ (\Im s)^2 - (k+\vartheta_j)^2 \omega^2}} \right)^{1/2-\Re s}\,.
\end{align}
For $n \leq \max \{m_0, 2m_{\vartheta_j}\}$, the estimate \eqref{eq:estimate_smatrix_near_resonance} always dominates \eqref{eq:estimate_smatrix_away_resonance}. Taking into account the multiplicities of the eigenvalues of $\twist(h_\ell)$, we conclude that
\begin{align*}
\mu_n(S_{F_\ell,\twist}(s)) \leq  e^{C\ang{s}} \ang{s}^{1-2\Re s} d_n(s)\,.
\end{align*}
For $n > \max \{m_0, 2m_{\vartheta_j}\}$, we can bound $\mu_k(S_{F_\ell,\twist}(s))$ by a decreasing rearrangement of \eqref{eq:estimate_smatrix_away_resonance}. The $n$-th element of the increasing rearrangement of the set
\[
\left\{\abs{(\Im s)^2 - \omega^2(k+\vartheta_j)^2} \right\}_{k \in \Z, j = 1,\dotsc,\dim(V)}
\]
can be bounded by $C n^2$ for some $C > 0$ independent of $s$, but depending on $\twist(h_\ell)$. Thus, we arrive at the estimate
\begin{align*}
\mu_n(S_{F_\ell,\twist}(s)) \leq  e^{C\ang{s}} \left(\frac{\ang{s}}{n}\right)^{1-2\Re s}\,.
\end{align*}
This completes the proof. 
\end{proof}

\begin{lemma}\label{lem:svalues-poisson}
For $\varphi \in \CcI(F_\ell)$ and $s \in \Omega_1$, we have
\begin{align*}
\mu_n(\varphi E_{F_\ell,\twist}(s)) \leq e^{C \ang{s} - c n}\,.
\end{align*}
\end{lemma}

\begin{proof}
As in \cite[Lemma 9.16]{Borthwick_book}, we have an explicit formula for the Poisson operator:
\begin{align*}
E_{F_\ell,\twist}(s;z,\phi') = \frac{1}{4\pi} \frac{\gammafunc(s)^2}{\gammafunc(2s)} \left( h_s(x,y,\phi') + h_s(-x,y,\phi') \right)\,,
\end{align*}
where
\begin{align*}
h_s(x,y,\phi') = \sum_{k \in \Z} \twist(h_\ell)^k
\left( \frac{4ye^{\ell (\frac{\phi'}{2\pi} + k)}}{\left(x-e^{\ell(\frac{\phi'}{2\pi} + k)}\right)^2 + y^2}\right)^s\,,
\end{align*}
where $z=x+iy \in \funddom$, where $\funddom$ is given by \eqref{eq:funddom_hypcyl}.
Since $\norm{\twist(h_\ell)} = 1$, for every $p \in \N_0$ we have the bound (see \cite[(9.48)]{Borthwick_book})
\begin{align}\label{eq:EFell_bound}
\norm*{\pa_{\phi'}^p \varphi(z) E_{F_\ell,\twist}(s;z,\phi')} \leq C^p p!\, e^{C\ang{s}}\,.
\end{align}
Using the flat Laplacian $\Delta_{\pa_\infty F_\ell}$ on $L^2(\pa_\infty F_\ell, \bundle|_{\pa_\infty F_\ell})$, we have the singular value bound
\begin{align*}
\mu_n(\varphi E_{F_\ell,\twist}) &\leq \mu_n( (\Delta_{\pa_\infty F_\ell} + \I)^{-m}) \, \norm*{ (\Delta_{\pa_\infty F_\ell} + \I)^m E_{F_\ell,\twist}^T \varphi}
\\
&\leq C^{2m} n^{-2m} (2m)!\, e^{C\ang{s}}\,.
\end{align*}
Note that $\Delta_{\pa_\infty F_\ell}$ has eigenvalues $k^2$ with multiplicity $\dim V$, therefore $\mu_n( (\Delta_{\pa_\infty F_\ell} + \I)^{-m}) \lesssim n^{-2m}$.
Choosing $m \in \N$ such that $\frac{2 C m}{n} = 1 - \eps$ for some $\eps \in [0,1)$ and applying Stirling's formula, the desired estimate follows.
\end{proof}

We are now ready to prove the estimate 
\[
\abs*{ g_{F_\ell,\twist}(s)\det ( \I + c \abs{K_f(s)})} \lesssim e^{C\ang{s}^2}\]
for $s \in \Omega_3$.

\begin{proof}[Proof of Lemma~\ref{lem:funnel-a3}]
Let $s \in \Omega_3$. From \eqref{eq:resolvent-scattering-matrix} and \eqref{eq:K_funnel}, we have that
\begin{align*}
K_f(s) &- K_f(1-s)
\\
& = (2s-1) [\Delta_{F_\ell,\twist}, \eta_{f,0}] E_{F_\ell,\twist}(1-s) S_{F_\ell,\twist}(s) E_{F_\ell,\twist}(1-s)^T (\eta_{f,3} - \eta_{f,1})\,.
\end{align*}
The resolvent estimate \eqref{est:resolvent-twist2} and Lemma~\ref{lem:bound-cylinder-resolvent-sobolev} imply that $\norm{K_f(1-s)}$ is bounded independently of $s \in \Omega_3$. From \eqref{eq:EFell_bound}, we obtain that for any compactly supported vector field $W$ on $F_\ell$, we have that
\begin{align*}
\norm*{W E_{F_\ell,\twist}(s;z,\phi')} \leq e^{C\ang{s}}\,.
\end{align*}
Combining these two estimates with $W = [\Delta_{F_\ell,\twist}, \eta_{f,0}]$ yields
\begin{align*}
\mu_{n_1+n_2+1}(K_f(s)) &\leq e^{C\ang{s}} \mu_{n_1}(S_{F_\ell,\twist}(s)) \mu_{n_2}( (\eta_{f,3} - \eta_{f,1}) E_{F_\ell,\twist}(1-s) )\,.
\end{align*}
If $n = n_1 + n_2 + 1 > \ang{s}$, then we can take $n_1 = \floor{\ang{s} + 1}$, so that $n_2 \sim n - \ang{s}$. Hence for sufficiently large $a > \max\{m_0, 2 m_{\vartheta_j}\} > 0$ and $n > a \ang{s}$, we have by Lemma~\ref{lem:svalues-smatrix} and Lemma~\ref{lem:svalues-poisson} that
\begin{align*}
\mu_{n_1+n_2+1}(K_f(s)) \leq e^{-b \ang{s}}
\end{align*}
for some $b > 0$. We obtain that
\begin{align*}
\det(\I + c\abs{K_f(s)}) \leq C \prod_{n \leq a\ang{s}} \left(1 + c\mu_n (K_f(s)) \right)\,.
\end{align*}
Taking $n_2 = 1$ and $n_1 = n$, we obtain
\begin{align*}
\det(\I + c\abs{K_f(s)}) &\leq d_{\ResSet_{X,\twist}}(s) \prod_{n \leq a\ang{s}} e^{C\ang{s}} \left(\frac{\ang{s}}{n} \right)^{1-2\Re(s)}
\\
&\leq d_{\ResSet_{X,\twist}}(s) e^{C\ang{s}^2} \left(\frac{\ang{s}^{a\ang{s}}}{\gammafunc(a \ang{s})} \right)^{1-2\Re(s)}
\\
&\leq d_{\ResSet_{X,\twist}}(s) e^{C_1\ang{s}^2}\,,
\end{align*}
where we have used Stirling's formula to estimate the term in parenthesis by $e^{C\ang{s}}$. Since the zero set of $g_{F_\ell,\twist}$ includes the poles with correct multiplicities and using \eqref{eq:gf-bound}, we find 
\begin{align*}
\abs*{ d_{\ResSet_{F_\ell,\twist}}(s) g_{F_\ell,\twist}(s) } \leq e^{C\ang{s}^2}\,.
\end{align*}
\end{proof}

\begin{proof}[Proof of Lemma~\ref{lem:funnel-determinant}]
Combining the estimates of the funnel determinant for all three regions, i.e., Lemmas~\ref{lem:funnel-a1}, \ref{lem:funnel-a2} and~\ref{lem:funnel-a3}, we obtain the estimate
\begin{align*}
\abs*{g_{F_\ell,\twist}(s) \det( \I + c \abs{K_f(s)}) } \leq
\begin{cases}
e^{C\ang{s}^2}\,, & \abs{\Im s - 1/2} \geq \eps\,,
\\
e^{C\ang{s}^3}\,, & \abs{\Im s - 1/2} \leq \eps\,.
\end{cases}
\end{align*}
An application of the Phragm\'en--Lindelöf theorem shows the claim.
\end{proof}

\subsection{The Endgame}
We are now ready to prove Proposition~\ref{prop:zeros-determinant} and to deduce Theorem~\ref{thm:upper-bound}.

\begin{proof}[Proof of Proposition~\ref{prop:zeros-determinant}]
Recall that $g_c(s) = (2s-1)^{n^\twist_c}$, where $n^\twist_c$ is given by \eqref{eq:nctwist}. Moreover, let $g_f(s)$ be the product of all $g_{X_f,\twist}(s)$ as in \eqref{def:gFellchi}. From Lemma~\ref{lem:funnel-determinant}, we have by \cite[Lemma~9.6]{Borthwick_book} that
\begin{align*}
\abs{g_f(s)}^9 \det(\I + 9 \abs{K_f(s)}^3) &\leq \left( \abs{g_f(s)}\det(\I + 3^{2/3} \abs{K_f(s)})\right)^9
\\
&\leq e^{C_1\ang{s}^2}\,,
\\
\intertext{and from Proposition~\ref{prop:cusp-determinant}, we have }
\abs{g_c(s)}^9 \det(\I + 9 \abs{K_c(s)}^3) &\leq \left( \abs{g_c(s)}\det(\I + 3^{2/3} \abs{K_c(s)})\right)^9
\\
&\leq e^{C_2\ang{s}^2}\,,
\end{align*}
for some $C_1, C_2 > 0$. Hence, by \eqref{eq:determinant-into-model} and Lemma~\ref{lem:interior-determinant},
\begin{align*}
\abs{g_c(s)}^9 \abs{g_{f}(s)}^9 \abs{D(s)} \leq e^{C_3\ang{s}^2}
\end{align*}
for some $C_3 > 0$. Since $g_c(s)^9 g_f(s)^9 D(s)$ is entire and not identically zero, we find $z_0 \in \C$ in a $1/4$-neighborhood of~$0$ such that $g_c(z_0)^9 g_f(z_0)^9 D(z_0) \not = 0$. Applying Jensen's formula \cite[Sect.~XII, Theorem~1.2]{LangComplexAnalysis} (or the weaker Number of Zeros Theorem by Titchmarsh~\cite{Titchmarsh_book}) we find $C > 0$ such that for any $r > 1$ the number of zeros of $g_c(s)^9 g_f(s)^9 D(s)$ in the ball $\{\abs{s - z_0} \leq r\}$ is bounded by $C r^2$. Thus,  for any $r > 1$ the number of zeros of $D(s)$ in the ball $\{\abs{s} \leq r\}$ is bounded by $C r^2$ for some $C > 0$.
\end{proof}

For completeness we detail the proof of Theorem~\ref{thm:upper-bound}.

\begin{proof}[Proof of Theorem~\ref{thm:upper-bound}]
By Lemma~\ref{lem:zeros-determinant}, the resonance counting function~$N_{X,\twist}$ is bounded above by the resonance counting function for the funnel ends of~$X$,  the zero counting function for the map~$D$ that is defined in Lemma~\ref{lem:zeros-determinant} and a constant of at most~$6$. By  Remark~\ref{rem:upper-bound-funnel} and Proposition~\ref{prop:zeros-determinant}, the two relevant counting functions are of order~$O(r^2)$ as $r\to\infty$. Thus, 
\[
 N_{X,\twist}(r) = O(r^2) \qquad \text{as $r \to \infty$}\,,
\]
as claimed. 
\end{proof}

%% file: blowup.tex
\begin{tikzpicture} % Blowup

%% some definitions

\def\R{2} % sphere radius
\def\angEl{15} % elevation angle
\def\angAz{-100} % azimuth angle

%% working planes
\LatitudePlane[equator]{\angEl}{0}

% Blown-up picture
% Stuff in the x-y plane
\draw[equator] (\angAz:\R) to[bend right=45] (0:\R);
\draw[equator,dashed] (0:\R) to[bend right=45] (\angAz+180:\R);
\draw[equator,->] (\angAz:\R) to (\angAz:3*\R) node[below right] {$\omega$};
\draw[equator,dashed] (\angAz:-\R) to (\angAz:-3*\R);
\draw[equator] (\angAz:-3*\R) to (\angAz:-\R*5);
\draw[equator,->] (0:\R) to (0:2*\R) node[above] {$\eta$};
% Stuff in the x-z plane
\LongitudePlane[xzplane]{\angEl}{0}
\draw[xzplane,->] (90:\R) to (90:2*\R) node[right] {$\eta'$};
\draw[xzplane,thick] (45:\R) to (45:2*\R-0.05) node[below right] {$\diagon_0$};
\coordinate[mark coordinate] (N) at (44.2:\R-0.05);
% Stuff in the y-z plane
\LongitudePlane[yzplane]{\angEl}{\angAz}
\draw[yzplane] (0:\R) to[bend right=52] (102:\R);
\draw[yzplane,dashed] (0:-\R) to[bend left=41] (100:\R); % last parameter is around 1 (tweak to make figure nice)
% enclosing circle segment
\draw (\R,0) arc (0:88.2:\R);

% Blown-down picture
\draw[<->] (6*\R-0.5,0) node[above] {$y$} -- (4*\R,0) -- (4*\R,2*\R) node[right] {$y'$};
\draw[->] (4.5*\R,\R) -- (3.5*\R,-1*\R) node[below right] {$u = x' - x$};
\draw[thick] (4*\R,0) -- (5.5*\R,1.5*\R) node[below right] {$\diagon$};
\coordinate[mark coordinate] (D) at (4*\R,0);
%\draw[dashed] (4,0) -- (4.5,1);

\end{tikzpicture}